\documentclass[11pt]{article}
\usepackage[utf8]{inputenc}

\usepackage{epsfig,epsf,fancybox}
\usepackage{amsmath}
\usepackage{mathrsfs,bbm}
\usepackage{amssymb}
\usepackage{graphicx}
\usepackage{color}
\usepackage{multirow}
\usepackage{paralist}
\usepackage{verbatim}
\usepackage{galois}
\usepackage{algorithm}
\usepackage{algorithmic}
\usepackage{boxedminipage}
\usepackage{booktabs}
\usepackage{accents}
\usepackage{stmaryrd}
\usepackage{subfig}
\usepackage{appendix}
\usepackage{graphicx}
\usepackage{epstopdf}
\usepackage{appendix}
\usepackage{amsthm}
\usepackage{cases}  
\usepackage[top=1in,bottom=1.2in,left=1in,right=1in,xetex]{geometry}
\usepackage{extarrows}
\usepackage{bm}

\usepackage{titlesec}
\usepackage{titletoc}

\usepackage[unicode=true,pdfusetitle,
bookmarks=true,bookmarksnumbered=true,bookmarksopen=true,bookmarksopenlevel=3,
breaklinks=false,pdfborder={0 0 1},backref=false,colorlinks=false]
{hyperref}

\numberwithin{equation}{section}

\newtheorem{theorem}{Theorem}[section]

\newtheorem{lemma}[theorem]{Lemma}
\newtheorem{proposition}[theorem]{Proposition}

\newtheorem{problem}{Problem}[section]

\newcommand{\f}{\mathscr{F}}
\newcommand{\lr}{\mathcal{L}}
\newcommand{\sr}{\mathcal{S}}
\newcommand{\mr}{\mathcal{M}}
\newcommand{\kr}{\mathcal{K}}
\newcommand{\e}{\mathbb{E}}
\newcommand{\he}{\widehat{\mathbb{E}}}
\newcommand{\br}{\mathbb{R}}
\newcommand{\pr}{\mathcal{P}}
\newcommand{\dd}{\partial}

\newcommand{\brn}{{\mathbb{R}^n}}
\newcommand{\brnn}{{\mathbb{R}^{n\times n}}}
\newcommand{\brd}{{\mathbb{R}^d}}
\newcommand{\de}{\Delta}

\newcommand{\hv}{\widehat{v}}

\newcommand{\hys}{\widehat{Y_s}}
\newcommand{\hyT}{\widehat{Y_T}}

\newcommand{\hps}{\widehat{P_s}}

\newcommand{\hqsj}{\widehat{Q^j_s}}
\newcommand{\hrs}{\widehat{R_s}}
\newcommand{\hr}{\mathcal{H}}

\newcommand{\dr}{\mathcal{D}}

\newcommand{\argmin}{\mathop{\rm argmin}}

\allowdisplaybreaks[2]

\title{Mean Field Type Control Problems Driven by Jump-diffusions}

\usepackage{authblk}

\author[a]{Alain Bensoussan\footnote{E-mail: axb046100@utdallas.edu}}
\author[b]{Ziyu Huang\footnote{E-mail: zyhuang19@fudan.edu.cn}}
\author[c]{Shanjian Tang\footnote{E-mail: sjtang@fudan.edu.cn}}
\author[b]{Sheung Chi Phillip Yam\footnote{E-mail: scpyam@sta.cuhk.edu.hk}}

\affil[a]{\small \it International Center for Decision and Risk Analysis, Naveen Jindal School of Management, University of Texas at Dallas, Dallas, Texas, USA}
\affil[b]{\small \it Department of Statistics and Data Science, The Chinese University of Hong Kong, Shatin, N.T., Hong Kong SAR}
\affil[c]{\small \it Department of Finance and Control Sciences, School of Mathematical Sciences,  Fudan University,  and Key Laboratory of Mathematics for Nonlinear Sciences (Fudan University), Ministry of Education, Shanghai 200433, China}

\begin{document}

\maketitle

\begin{abstract}
In this article, we apply a probabilistic approach to study general mean field type control (MFTC) problems with jump-diffusions, and give the first global-in-time solution. %and then also the unique existence of the classical solution for the corresponding HJB integro-partial differential equation. 
We allow the drift coefficient $b$ and the diffusion coefficient $\sigma$ to nonlinearly depend on the state, distribution and control variables, and both can be unbounded and possibly degenerate; besides, the jump coefficient $\gamma$ is allowed to be non-constant. %To the best of our knowledge, our settings and results are new. 
To tackle the non-linear and  control-dependent diffusion $\sigma$, we further formulate a joint cone property and estimates  for both  processes $P$ and $Q$ of the corresponding adjoint process (where $(P,Q,R)$ is the solution triple of the associated adjoint process as a backward stochastic differential equation with jump), in contrast to our previous single cone property of the only process $P$. %, under a general setting which extends the commonly used control-independent assumption on the coefficient $\sigma$. 
%mild assumption on the control variable for the coefficients $b$ and $\sigma$, which states that the dependencies of these different coefficients on the control argument do not overlap, and is indeed an extension of the commonly used control-independent assumption on the coefficient $\sigma$. 
We  study first the system of forward-backward stochastic differential equations (FBSDEs) with jumps arising from the maximum principle, and then  the related Jacobian flows, which altogether yield the classical regularity of the value function and thus allow us to show that the value function  is  the unique classical solution of the  HJB integro-partial differential equation. 
%based on the It\^o's formula for measure-dependent functionals and SDE with jump diffusion. 
Most importantly, our proposed probabilistic approach can apparently handle the MFTC problem driven by  a fairly general  process far beyond Brownian motion, in a relatively easier manner than the existing analytic approach.\\ %the latter may even encounter an immediate challenge of laying down a proper problem formulation.
%The classical solution of the HJB integro-partial differential equation is challenging to be obtained by the classical PDE based analytical method, and our stochastic control method from the probabilistic perspective provides even a simple approach. Overall, our method is capable of generalizing mean field problems to a more general framework beyond that driven just by Brownian motions.\\ 



\end{abstract}

%\tableofcontents

\section{Introduction}\label{sec:intro}

%Literature review
Mean field type control (MFTC) problems (also called McKean-Vlasov control problems) and mean field games (MFGs) have received a lot of attentions in the last two decades. The common point of both kinds of problems is that the controlled dynamical system process  depends on a probability distribution flow of the state. The state process for a MFTC problem is affected by the state, the control and the law of the current state; while the MFG is a fixed point problem with the state depending on the equilibrium law. The literature in this area is huge now. For probabilistic approaches to MFTC problems, we refer to \cite{BR,CDLL,CR,book_mfg,CJF}; for the dynamic programming principle (DDP) and Hamilton–Jacobi–Bellman (HJB) equation of McKean-Vlasov control problem, we refer to \cite{DMF,PH}; for the Fokker-Planck (FP) equations for McKean-Vlasov SDEs, we refer to \cite{barbu2023uniqueness,ren2022linearization}; for the lifting method and Hilbert space approaches for MFTC problems, we refer to \cite{AB7,AB6,AB8,AB5,AB}; for MFGs and McKean-Vlasov type differential games, we refer to \cite{AB_book,CA,HM}; for solution of linear-quadratic (LQ) MFTC problems or games, we refer to \cite{MR4393348,li2020indefinite}; for study on Riccati partial differential 
equations, we refer to \cite{byrnes1998riccati,byrnes1992shock}.
%For probabilistic approaches to MFTC problems, we refer to Buckdahn-Li-Peng-Rainer \cite{BR}, Cardaliaguet-Delarue-Lasry-Lions \cite{CDLL}, Carmona-Delarue \cite{CR,book_mfg} and Chassagneux-Crisan-Delarue \cite{CJF}. For the dynamic programming principle and Hamilton–Jacobi–Bellman (HJB) equation of McKean-Vlasov control problem, we refer to Djete-Possamai-Tan \cite{DMF} and Pham-Wei \cite{PH}. For the lifting method and a Hilbert space approach for mean field type control problem, we refer to Bensoussan-Graber-Yam \cite{AB7}, Bensoussan-Huang-Yam \cite{AB6,AB8}, Bensoussan-Tai-Yam \cite{AB5} and Bensoussan-Yam \cite{AB}. For mean field games and stochastic differential games of McKean-Vlasov type, we refer to Bensoussan-Frehse-Yam \cite{AB_book}, Cosso-Pham \cite{CA} and Huang-Malham\'e-Caines \cite{HM}.

% MFTC with jump: stochastic control method, compared with PDE method 
The presence of the jump-diffusion likely much complicates the MFTC problem, especially for the analytical methods for the HJB-FP approaches, since the HJB equation includes an integral term derived from the jump driving noise; see \cite{guo2023ito} for instance. Therefore, for classical solution of such an equation,
%this HJB-FP integro-partial differential system
the conventional analytical method meets with a fundamental challenge; indeed, even for the classical stochastic optimal control problem with jump-diffusion in the absence of the mean field term, classical solution of the HJB integro-partial differential equations is not easy; see \cite{li2009stochastic} for viscosity solution of that type of HJB equations. For the  FP equation for optimal control problems for jump-diffusion process, we refer the reader to the work \cite{Oksendal-Sulem}; in particular, for the FP equation for McKean-Vlasov SDE with jump diffusion, we refer to \cite{agram2023stochastic}. For the It\^o’s formula and DDP for stochastic control problems with jump-diffusion, we refer to \cite{guo2023ito,gyongy2021ito,Oksendal-Sulem}. In \cite{agram2023stochastic} and \cite{guo2023ito}, the verification theorems are given on  the HJB-FP systems for the MFTC problems with jump-diffusion, and also LQ cases are solved; and  in \cite{li2016controlled}, the unique viscosity solution of the HJB integro-partial differential equation is given for BSDEs coupled with their value function and the mean field SDEs with jump-diffusions. For mean field BSDE with jump-diffusion which is associated with a mean field SDE with jump-diffusion, see \cite{LI20183118} for the classical well-posedness. For solution of the mean field type LQ differential games with jump-diffusion systems, we refer to \cite{barreiro2019linear,MR4064662,MR4885026,MR4396401}. Ensuring classical solutions of HJB integro-partial differential equation even for the standard general optimal control problems with jump-diffusions, let alone the generic MFTC problem, is widely regarded to be highly challenging. Our current probabilistic approach turns out to be a resolution, which is conceptually even simple. 

% Main result and novelty
More precisely, in this article, we apply a probabilistic approach and use stochastic control method to study the following MFTC problem with jump-diffusion for any initial time $t$ and initial condition $\xi$ with its law $\lr(\xi)=\mu$:
\begin{equation}\label{problem:intro}
    \left\{
    \begin{aligned}
        &J_{t,\xi}(v):=\e\left[\int_t^T f\left(s,X_s^{v},\lr\left(X_s^{v}\right),v_s\right) ds + g\left(X_T^{v},\lr\left(X_T^{v}\right)\right)\right],\quad v\in \mr_\f^2(t,T),\\
        &\text{such that } X_s^{v}=\xi+\int_t^s b\left(r,X_r^{v},\lr\left(X_r^{v}\right),v_r\right) ds +\int_t^s \sigma \left(r,X_r^{v},\lr\left(X_r^{v}\right),v_r\right)dB_s\\
        &\quad\qquad\qquad\qquad +\int_t^s \int_E \gamma\left(r,X_{r-}^{v},\lr\left(X_{r-}^{v}\right),v_r,e\right) \mathring{N}(de,ds),\quad s\in[t,T],
    \end{aligned}
    \right.
\end{equation}
and define the value function as 
\begin{align}\label{def:value:intro}
    V(t,\lr(\xi)):= \inf_{v\in\mr_\f^2(t,T)} J_{t,\xi}(v);
\end{align}
see Subsection~\ref{lem01_1} for the precise definition for coefficients $(b,\sigma,\gamma,f,g)$. By studying the well-posedness and differentiability (see Section~\ref{sec:FBSDE}) of forward-backward stochastic differential equations (FBSDEs) with jump arising from the maximum principle (see Section~\ref{sec:maxi}) associated with Problem \eqref{problem:intro}, our main result is that, when $\gamma$ is independent of the control variable, the value function $V$ defined in \eqref{def:value:intro} is the unique classical solution of the following HJB integro-partial differential equation for MFTC problem with jump: for $(t,\mu)\in [0,T]\times\pr_2(\brn)$,
\begin{equation}\label{HJB:intro}
\left\{
    \begin{aligned}
        &\frac{\dd V}{\dd t}(t,\mu)+ \int_\brn \inf_{v\in\brd} \hr \bigg(t,x,\mu,v,D_y\frac{dV}{d\nu}(t,\mu)(x),D_y^2\frac{dV}{d\nu}(t,\mu)(x),\frac{dV}{d\nu}(t,\mu)(\cdot)\bigg) \mu(dx)=0,\\
        &V(T,\mu)= \int_\brn g(x,\mu)\mu(dx).
    \end{aligned}
\right.
\end{equation}
where the functional $\hr:[0,T]\times\brn\times\pr_2(\brn)\times\brd\times\brnn\times L_\lambda^2(E)\to\br$ is defined as
\begin{align*}
    \hr(t,x,\mu,v,p,\Gamma,k(\cdot)):= \ & \frac{1}{2}\text{Tr}\left[\left(\sigma\sigma^\top\right) \left(t,x,\mu,v\right) \Gamma\right] +p^\top \left[b\left(t,x,\mu,v\right)-\int_E \gamma(t,x,\mu,e)\lambda(de)\right]\notag \\
    & +f\left(t,x,\mu,v\right) +\int_E \left[k\left(x+\gamma\left(t,x,\mu,e\right)\right)-k(x) \right] \lambda(de).
\end{align*}
Note that our problem \eqref{problem:intro} is an open-loop problem (though the optimal control is shown to be of a feedback form). In Problem \eqref{problem:intro}, we  allow the diffusion coefficient $\sigma$ to depend on the control, and allow $b$ and $\sigma$ to be nonlinear in the state, distribution and control variables. Equation \eqref{HJB:intro} is also solved in the  works \cite{agram2023stochastic,guo2023ito}, where the verification results are given and  interesting examples are solved. In \cite{li2016controlled}, the authors studied the integro-partial differential equation associated with BSDEs coupled with their value function and prove the unique existence of the viscosity solution. In the motivating work \cite{LI20183118}, an integro-partial differential equation---a simpler version of  \eqref{HJB:intro}---for mean field BSDE with jumps and an  associated  mean field SDE,  is shown to have a unique classical solution. To the best of our knowledge, our solution of Problem \eqref{problem:intro}  and existence and uniqueness result for the classical solution for the HJB integro-partial differential equation \eqref{HJB:intro} are completely new.

% FBSDEs, cone property and separable assumption for control variable

%To classically solve the HJB integro-partial differential equation with a traditional PDE method is difficult, however, our work gives a probabilistic approach. 
From the viewpoint of control theory, we solve the MFTC problem \eqref{problem:intro} via the global solution of the system of FBSDEs with jumps derived from the maximum principles for their optimal controls (also see \cite{Oksendal-Sulem,oksendal2018stochastic} for maximum principle for control problems with jumps); then, we elaborately study the Jacobian flows of this system of FBSDEs with jumps to identify the derivatives of the value functional $V$ defined in \eqref{def:value:intro} with respect to the distribution argument, which is crucial for the global-in-time well-posedness for the HJB integro-partial differential equation \eqref{HJB:intro}. Our stochastic control approach also takes advantage of imposing less regularity on the coefficients over the analytical approach; furthermore, to obtain a solution for the MFTC problem \eqref{problem:intro}, we show that one actually needs fewer regularity assumptions on the coefficients than in the situation of obtaining the classical solution to Equation \eqref{HJB:intro}. Specifically, when the coefficient functions are $C^1$, MFTC \eqref{problem:intro} can still be warranted with a unique solution; besides, when the coefficient functions are $C^2$, the value function will be shown to be the unique classical solution of the HJB integro-partial differential equation \eqref{HJB:intro}.
%The proof of the global-in-time well-posedness for the HJB integro-partial differential equation \eqref{HJB:intro} is based heavily on the elaborate study for the regularity of the solutions to the system of FBSDEs with jumps associated with Problem \eqref{problem} arising from maximum principle. 
For the study on FBSDEs with jumps, we also refer to \cite{MR3346708,MR4199898,zhen1999forward}. In comparison with the existing literature, we study not only this system of FBSDEs, but also the related Jacobian flows, which are also systems of forward-backward equations with jump-diffusions. One  major difficulty in our problem consists in the nonlinearity of the coefficients $b$ and  $\sigma$ in $(x,m,v)$. To overcome the difficulty, the crucial step is to use the so-called ``cone property" for the adjoint process $P$, which was first used in \cite{AB10''} for the first order MFTC problem. This cone property is also used in our last work \cite{AB13} for the second order MFGs (without jump) to incorporate nonlinear  coefficients $b$ and control-independent $\sigma$.  In this work,  our cone properties refers to both  processes $P$ and  $Q$, where $(P,Q,R)$ is the solution triple of the associated adjoint process as a BSDE with jump (see \eqref{adjoint}). The cone property for $Q$ is required to prove the well-posedness of the FBSDEs, to tackle the nonlinear dependence of the coefficient $\sigma$ on the control variable $v$. It is naturally formulated to monitor the growth of various sensitivities of adjoint processes with respect to the initial data. And it is based on a mild assumption on the control variable $v$ for the coefficients $b$ and $\sigma$ (see Assumption (A3)), about the variance of the drift coefficient $b$ and the diffusion coefficients $\sigma^j$ between control argument components; and it indeed generalizes the commonly used control-independent assumption on the coefficient $\sigma$ (see \cite{SA,AB13} for instance). %or the usual linearly-dependence in control condition for $\sigma$ (see \cite{barreiro2019linear,AB11,CR,li2020indefinite} for instance).
 
The rest of this article is organized as follows. In  Section~\ref{sec:notation}, we give the formal formulation of our problem. In section~\ref{sec:maxi}, we prove the necessary and sufficient maximum principle for our MFTC problem with jump-diffusion and derive the associated system of FBSDEs with jump-diffusions, and also introduce the cone property for the processes $P$ and $Q$. In Section~\ref{sec:FBSDE}, we prove the well-posedness of the global-in-time solution for the FBSDEs with jump-diffusions, and also give the G\^ateaux differentiability of the solution processes of the FBSDEs with jump-diffusion with respect to the initial condition. In Section~\ref{sec:value}, we give the regularity of the value function $V$. In Section~\ref{sec:HJB}, we eventually establish that the value function $V$ is the unique classical solution of the HJB integro-partial differential equation. Some statements in Sections~\ref{sec:maxi}, \ref{sec:FBSDE},  \ref{sec:value} and \ref{sec:HJB} are proven in Appendices \ref{sec:pf:maxi}, \ref{pf:thm:Gateaux}, \ref{sec:pf:value} and \ref{pf:thm:value_regu_mu'}, respectively.

\section{Preliminaries and problem formulation}\label{sec:notation}

\subsection{Probability spaces and notations}

Let $(\Omega,\f,\{\f_t,0\le t\le T\},\mathbb{P})$ be a completed filtered probability space (with the filtration being  augmented by all the $\mathbb{P}$-null sets) on which the following two independent stochastic processes are defined and are $\f_t$-adapted (with the future increments after time $t$ of both processes being independent of $\f_t$ for every $t \in[0,T)$):
\begin{enumerate}[(i)]
    \item An $n$-dimensional Brownian motion $\{B(t),\ 0\le t\le T\}$;
    \item A right continuous Poisson random process $N$ taking values in $E$ with the deterministic intensity measure $\hat{N}(de,dt)=\lambda(de)dt$ satisfying $\lambda(E)<\infty$, where $E:=\br^l \setminus\{0\}$ is equipped with its Borel $\sigma$-field $\mathcal{B}(E)$ (also see \cite{li2009stochastic,Oksendal-Sulem}); here, $\lambda$ is a $\sigma$-finite measure on $(E, \mathcal{B}(E))$ with $\int_E (1+|e|^2)\lambda(de)<\infty$. The process $\left\{\mathring{N}(A,(t,s]):=\left(N-\hat{N}\right)(A,(t,s]) \right\}$ is an associated compensated $\f$-martingale random (Poisson) measure of $N$ for any $A\in\mathcal{B}(E)$ satisfying $\lambda(A)<\infty$ and $0\le t\le s\le T$. 
\end{enumerate}
For any $\xi\in L^2(\Omega,\f,\mathbb{P};\brn)$, we denote by $\lr(\xi)$ its law and by $\|\xi\|_2$ its $L^2$-norm. For every $t\in[0,T]$, we denote by $L^2_{\f_t}$ the set of all $\f_t$-measurable square-integrable $\brn$-valued random vectors. We also introduce the following spaces:
\begin{itemize}
    %\item $\lr^2_{\f}(0,T)$ is the family of all $\f$-progressively measurable $\brn$-valued processes $\alpha_\cdot=\{\alpha_t,\ 0\le t\le T\}$ such that $\e\left[\int_0^T |\alpha_t|^2dt\right]<\infty$;
    \item $\mr^2_{\f}(t,s)$ is the family of all $\f$-predictable $\brn$-valued processes $\alpha_\cdot=\{\alpha_t,\ t\le r\le s\}$, i.e. for all $r$, $\alpha_{r-}=\alpha_r$, such that $\e\left[\int_t^s |\alpha_r|^2dt\right]<\infty$;
    \item $\sr^2_{\f}(t,s)$ is the family of all $\f$-adapted $\brn$-valued c\`adl\`ag processes $\alpha_\cdot=\{\alpha_r,\ t\le r\le s\}$ such that $\e\left[\sup_{t\le r\le s} |\alpha_r|^2\right]<\infty$;
    \item $L^2_\lambda(E)$ is the family of all measurable functions $f:E\to \brn$ such that $\|f\|^2_{L^2_\lambda}:=\int_E |f(e)|^2\lambda(de)<\infty$;
    \item $\kr^2_{\f,\lambda}(t,s)$ is the family of all jointly measurable functions $k:[t,s]\times E\times\Omega\to \brn$ such that $k(r,\cdot)$ is $\mathcal{P}\otimes\mathcal{B}(E)$-measurable predictable with 
    \begin{align*}
        \e\left[\int_t^s \|k(r,\cdot)\|_{L^2_\lambda}^2 dr\right]:=\e\left[\int_t^s \int_E |k(r,e)|^2 \lambda(de) dt\right]<\infty,
    \end{align*}
    where $\mathcal{P}$ denotes the $\sigma$-algebra of $\f$-predictable subsets of $[t,s]\times\Omega$. We also denote by $\|k(r)\|_{L^2(\Omega;L_\lambda^2)}^2:=\e\left[ \int_E |k(r,e)|^2 \lambda(de)\right]$.
\end{itemize}

For the sake of convenience, in this article, we write $f|_{a}^b:=f(b)-f(a)$ for the difference of a  functional $f$ between two points $b$ and $a$. %For any matrix $Q\in\br^{n\times n}$ and vector $x\in\brn$, we use the notation $Qx^{\otimes 2}:=x^\top Qx$.

\subsection{Wasserstein space and derivatives of functionals}

We denote by $\mathcal{P}_{2}(\brn)$ the space of all probability measures of finite second order moments on $\brn$, equipped with the 2-Wasserstein metric: 
\begin{align*}
    W_2\left(m,m'\right):=\inf_{\pi\in\Pi\left(m,m'\right)}\sqrt{\int_{\brn\times\brn}\left|x-x'\right|^2\pi\left(dx,dx'\right)},
\end{align*}
where $\Pi\left(m,m'\right)$ is the set of joint probability measures with respective marginals $m$ and $m'$. We denote by $|m|_1:=\int_\brn |x|m(dx)$, and denote by $|m|_2:=\sqrt{\int_\brn |x|^2 m(dx)}$. We denote by $\delta_0$ the point-mass distribution of the random variable $\xi$ such that $\mathbb{P}(\xi=\mathbf{0})=1$. More results about Wasserstein metric space can be found in \cite{AL}. The linear functional-derivative of a functional $k(\cdot):\mathcal{P}_{2}(\brn)\to\br$ at $m\in \mathcal{P}_{2}(\brn)$ is another functional $\mathcal{P}_{2}(\brn)\times \brn\ni(m,y)\mapsto\dfrac{dk}{d\nu}(m)(y)$,  being jointly continuous and satisfying 
$$\int_{\brn}\Big|\dfrac{dk}{d\nu}(m)(y)\Big|^{2} m(dy)\leq c(m),
$$
for some positive constant $c(m)>0$, which is bounded on any bounded subsets of $\pr_2(\brn)$ such that 
\begin{equation*}
    \lim_{\epsilon\to0}\dfrac{k((1-\epsilon)m+\epsilon m')-k(m)}{\epsilon}=\int_\brn\dfrac{dk}{d\nu}(m)(y)\left(m'(dy)-m(dy)\right), \quad \forall m'\in\mathcal{P}_{2}(\brn);
\end{equation*}
we refer the reader to \cite{AB,book_mfg} for more details about the notion of linear functional-derivatives. In particular, the linear functional-derivatives in $\pr_2(\brn)$ are connected with the G\^ateaux derivatives in $L^2(\Omega,\f,\mathbb{P};\brn)$ in the following manner. For a linearly functional-differentiable functional $k:\pr_2(\brn)\to\br$ such that the derivative $D_y\frac{d k}{d\nu}(m)(y)$ is jointly continuous in $(m,y)$ and $D_y\frac{d k}{d\nu}(m)(y)\le c(m)(1+|y|)$ for $(m,y)\in\pr_2(\brn)\times\brn$,  the functional $K(\xi):=k(\lr(\xi)), \  \xi\in L^2(\Omega,\f,\mathbb{P};\brn)$ has the following G\^ateaux derivative:
\begin{align}\label{lem01_1}
	D_\xi K(\xi)(\omega)=D_y\frac{d k}{d\nu}(\lr(\xi))(\xi(\omega)),
\end{align}
which is also known as the Wasserstein gradient. From here onward, for any random variable $\xi$, we write $\widehat{\xi}$ for its independent copy, and $\widehat{\e}\left[\widehat{\xi}\right]$ for the corresponding expectation taken; and we also use $\widetilde{\xi}$ for its another independent copy, and $\widetilde{\e}\left[\widetilde{\xi}\right]$ for this corresponding expectation taken.

\subsection{Problem formulation}

Let coefficients for the state process
\begin{align*}
    &b:[0,T]\times\brn\times\pr_2(\brn)\times\brd\to\brn,\qquad \sigma:[0,T]\times\brn\times\pr_2(\brn)\times\brd\to\brnn,\\
    &\gamma:[0,T]\times\brn\times\pr_2(\brn)\times\brd\times E\to\brn,
\end{align*}
and the coefficients for the cost functional
\begin{align*}
    f:[0,T]\times\brn\times\pr_2(\brn)\times\brd\to\br,\qquad g:\brn\times\pr_2(\brn)\to\br;
\end{align*}
and the regularity assumptions on $b,\sigma,\gamma,f$ and $g$ will be given in Section~\ref{sec:maxi}. We first state our mean field type control problem at the initial time 0. For any control $v\in \mr_\f^2(0,T)$, the controlled state process is
\begin{equation}\label{SDE^v}
\begin{aligned}
    X_t^v=\xi&+\int_0^t b\left(s,X_s^v,\lr\left(X_s^v\right),v_s\right) ds +\int_0^t \sigma \left(s,X_s^v,\lr\left(X_s^v\right),v_s\right)dB_s\\
    &+\int_0^t \int_E \gamma\left(s,X_{s-}^v,\lr\left(X_{s-}^v\right),v_s,e\right) \mathring{N}(de,ds),\quad t\in[0,T];
\end{aligned}
\end{equation}
and the cost functional is defined as
\begin{equation}\label{cost}
    \begin{aligned}
        J(v):=\e\left[\int_0^T f\left(t,X_t^v,\lr\left(X_t^v\right),v_t\right) dt + g\left(X_T^v,\lr\left(X_T^v\right)\right)\right].
    \end{aligned}
\end{equation}

\begin{problem}
    We denote by the MFTC problem 
    \begin{equation}\label{problem}
        \left(\mathbf{P}^{0,\xi}\right): \qquad\qquad \inf_{v\in\mr_\f^2(0,T)} J(v), \qquad\qquad
    \end{equation}
    with an initial condition $X^v_0=\xi$, state process \eqref{SDE^v} and cost functional  \eqref{cost}. For any initial time $t\in[0,T]$ and initial condition $\xi$, we denote by $\left(\mathbf{P}^{t,\xi}\right)$ the corresponding MFTC problem. 
\end{problem}

We define the Lagrangian $L:[0,T]\times\brn\times\pr_2(\brn)\times\brd\times\brn\times\brnn\times L_\lambda^2(E)\to\br$ as 
\begin{align}
    L(t,x,m,v,p,q,r):= \ & f(t,x,m,v)+b(t,x,m,v)^\top p + \sum_{j=1}^n\sigma^j(t,x,m,v)^\top q^j \notag \\
    & + \int_E \gamma(t,x,m,v,e)^\top r(e)\lambda(de), \label{Lagrangian}
\end{align}
and define the Hamiltonian $H:[0,T]\times\brn\times\pr_2(\brn)\times\brn\times\brnn\times L_\lambda^2(E)\to\br$ as
\begin{align*}
    H(t,x,m,v,p,q,r):=\inf_{v\in\brd} L(t,x,m,v,p,q,r).
\end{align*}
The well-definedness of $H$ and its properties will be given in Subsection~\ref{subsec:sufficient}.

\section{Maximum principle}\label{sec:maxi}

In this section, we first give the necessary condition in Subsection~\ref{subsec:necessary} for the optimal control of our mean field type control problem with jump diffusion, which derives the associated system of FBSDEs with jump-diffusion (see \eqref{FBSDE:MFTC} below) for the control problem; then we also prove the sufficient maximum principle (also see \cite{framstad2004sufficient}) in Subsection~\ref{subsec:sufficient}, which shows that the solution of this forward-backward system \eqref{FBSDE:MFTC} can really give the optimal control for the control problem. Our results are applicable to $\left(\mathbf{P}^{t,\xi}\right)$ for any initial time $t\in[0,T]$, but for notational convenience, we only give the statements and proofs for $\left(\mathbf{P}^{0,\xi}\right)$ here. We also refer to \cite{tang1994necessary} for necessity conditions for stochastic control systems with random jumps.

\subsection{Necessity condition}\label{subsec:necessary}

We begin by giving the necessity condition for Problem $\left(\mathbf{P}^{0,\xi}\right)$ under the following assumptions on coefficients $(b,\sigma,\gamma,f,g)$.

{\bf (A1)} The coefficients $b$, $\sigma$ and $\gamma$ are continuous and satisfy the following growth condition:
\begin{align*}
    |b(t,x,m,v)|,\ \left|\sigma^j(t,x,m,v)\right|\le \ &L(1+|x|+|m|_2+|v|),\\
    \|\gamma(t,x,m,v,\cdot)\|_{L_\lambda^2}\le\ & L(1+|x|+|m|_2+|v|);
\end{align*}
they are differentiable in $(x,v)\in\brn\times\brd$ and also functional-differentiable in $m\in\pr_2(\brn)$, with all derivatives $b_x,b_v,D_y\frac{db}{d\nu},\sigma_x,\sigma_v,D_y\frac{d\sigma}{d\nu},\gamma_x,\gamma_v,D_y\frac{d\gamma}{d\nu}$ being continuous and being bounded in norms by $L$. 

\textbf{(A2)} The functions $f$ and $g$ are continuous and have a quadratic growth in $(x,m,v)$; they are differentiable in $(x,v)\in\brn\times\brd$ and also functional-differentiable in $m\in\pr_2(\brn)$, with the derivatives satisfying: for any $(t,x,m,v)\in[0,T]\times\brn\times\pr_2(\brn)\times\brd$ and $y\in\brn$,
\begin{align*}
    &|(f_x,f_v)(t,x,m,v)|\le L(1+|x|+|m|_2+|v|),\quad |g_x(x,m)|\le L(1+|x|+|m|_2),\\
    & \left|D_y\frac{df}{d\nu}(t,x,m,v)(y)\right|\le L(1+|x|+|m|_2+|v|+|y|),\\
    & \left|D_y\frac{dg}{d\nu}(x,m)(y)\right|\le L(1+|x|+|m|_2+|y|),
\end{align*}
and the derivatives are $L$-Lipschitz continuous in $(x,m,v,y)$. 

For any control $v\in\mr_\f^2(0,T)$ for $\left(\mathbf{P}^{0,\xi}\right)$, the following result shows that the controlled state process $X^v$ belongs to the space $\in\sr_\f^2(0,T)$, whose proof is given in Appendix~\ref{pf:lem:control-state}.

\begin{lemma}\label{lem:control-state}
    Under Assumption (A1), for any control $v\in\mr_\f^2(0,T)$, the controlled state process $X^v\in\sr_\f^2(0,T)$ and satisfies 
    \begin{equation}\label{lem:control-state-1}
        \e\left[\sup_{0\le t\le T}\left|X^v_t\right|^2\right]\le C(L,T)\e\left[1+|\xi|^2+\int_0^T  |v_t|^2dt\right].
    \end{equation}
\end{lemma}

Suppose that $u\in\mr_\f^2(0,T)$ is an optimal condition for Problem $\left(\mathbf{P}^{0,\xi}\right)$, and we denote by $Y$ its associated state process; we also denote by $\theta_t:=(Y_t,\lr(Y_t),u_t)$ for $t\in[0,T]$. For any control $v\in \mr_\f^2(0,T)$, we denote by $\de v:=v-u\in\mr_\f^2(0,T)$, and define $v^\epsilon:=u+\epsilon\de v$ for $\epsilon>0$. We denote by $X^\epsilon$ the controlled state corresponding to $v^\varepsilon$, and denote by $\de X^\epsilon:=X^\epsilon-Y$. We first give the following estimate of $\de X^\epsilon$, which is proven in Appendix~\ref{pf:lem:rhoX}.

\begin{lemma}\label{lem:rhoX}
    Under Assumption (A1), we have 
    \begin{equation}\label{Delta_delta_distance}
        \lim_{\epsilon\to 0}\ \e\left[\sup_{0\le t\le T}\left|\frac{\de X^\epsilon_t}{\epsilon}-\delta X_t\right|^2\right]=0,
    \end{equation}
    where $\delta X\in\sr_\f^2(0,T)$ is the solution of the following SDE
    \begin{equation}\label{SDE:deltaX}
    \begin{aligned}
        \delta X_t=\ & \int_0^t \bigg\{b_x\left(s,\theta_s\right) \delta X_s+b_v\left(s,\theta_s\right) \de v_s +\he\bigg[\left(D_y\frac{db}{d\nu}\left(s,\theta_s\right)\left(\hys\right)\right) \widehat{\delta X_s}\bigg] \bigg\}ds \\
        &+\int_0^t \bigg\{\sigma_x\left(s,\theta_s\right) \delta X_s+\sigma_v\left(s,\theta_s\right) \de v_s +\he\bigg[\left(D_y\frac{d\sigma}{d\nu}\left(s,\theta_s\right)\left(\hys\right)\right) \widehat{\delta X_s}\bigg] \bigg\}dB_s \\
        &+\int_0^t \int_E \bigg\{\gamma_x\left(s,Y_{s-},\lr\left(Y_{s-}\right),u_s,e\right) \delta X_{s-}+\gamma_v\left(s,Y_{s-},\lr\left(Y_{s-}\right),u_s,e\right) \de v_s\\
        &\quad\qquad\qquad +\he\bigg[\left(D_y\frac{d\gamma}{d\nu}\left(s,Y_{s-},\lr\left(Y_{s-}\right),u_s,e\right)\left(\widehat{Y_{s-}}\right)\right) \widehat{\delta X_{s-}}\bigg] \bigg\}\mathring{N}(de,ds),\quad t\in[0,T],
    \end{aligned}
    \end{equation}
    and satisfies
    \begin{equation}\label{deltaX:boundedness}
        \begin{aligned}
            &\e\left[\sup_{0\le t\le T}|\delta X_t|^2\right]\le C(L,T)\e\left[\int_0^T |\de v_t|^2 dt \right].
        \end{aligned}
    \end{equation}
\end{lemma}

Now we define the adjoint process as follows:
\begin{equation}\label{adjoint}
    \begin{aligned}
        P_t=\ & g_x(Y_T,\lr(Y_T))+\he\left[D_y\frac{dg}{d\nu}\left(\hyT,\lr(Y_T)\right)(Y_T)\right]\\
        &+\int_t^T \bigg\{ b_x\left(s,\theta_s\right)^\top P_s+\sum_{j=1}^n \sigma^j_x \left(s,\theta_s\right)^\top Q^j_s+\int_E \gamma_x\left(s,\theta_s,e\right)^\top R_s(e)\lambda(de)+f_x\left(s,\theta_s\right)\\
        &\qquad\qquad +\he\bigg[\left(D_y\frac{db}{d\nu}\left(s,\widehat{\theta_s}\right)\left(Y_s\right)\right)^\top \hps+\sum_{j=1}^n \left(D_y\frac{d\sigma^j}{d\nu}\left(s,\widehat{\theta_s}\right)\left(Y_s\right)\right)^\top \hqsj \\
        &\qquad\qquad\qquad +\int_E \left(D_y\frac{d\gamma}{d\nu}\left(s,\widehat{\theta_s},e\right)\left(Y_s\right)\right)^\top \hrs(e) \lambda(de)+ D_y\frac{df}{d\nu}\left(s,\widehat{\theta_s}\right)\left(Y_s\right) \bigg]\bigg\}ds\\
        &-\int_t^T Q_s dB_s-\int_t^T\int_E R_s(e) \mathring{N}(de,ds), \quad t\in[0,T].
    \end{aligned}
\end{equation}
Equation \eqref{adjoint} is a BSDE for $(P,Q,R)$ driven by the Brownian motion $B$ and the jump process $\mathring{N}$. We have the following well-posedness result, whose proof is given in Appendix~\ref{pf:lem:adjoint}.

\begin{lemma}\label{lem:adjoint}
    Under Assumptions (A1) and (A2), given $u$ and $Y$, there exists a unique solution $(P,Q,R)\in\sr_\f^2(0,T)\times\mr_\f^2(0,T)\times\kr_{\f,\lambda}^2(0,T)$ of BSDE \eqref{adjoint}, and the solution satisfies 
    \begin{equation}\label{adjoint:boundedness}
        \e\bigg[\sup_{0\le t\le T}|P_t|^2+\int_0^T \bigg(|Q_t|^2 +\int_E |R_t(e)|^2 \lambda(de) \bigg) dt\bigg]\le C(L,T)\e\left[1+|\xi|^2+\int_0^T  |u_t|^2dt\right].
    \end{equation}
\end{lemma}

We now compute $\frac{d}{d\epsilon}J(u+\epsilon\de v)|_{\epsilon=0}$ by using BSDE \eqref{adjoint}. The proof of the following result is given in Appendix~\ref{pf:lem:dJ}. 

\begin{lemma}\label{lem:dJ}
    Under Assumptions (A1) and (A2), for the optimal control $u$, the corresponding state process $Y$ and the adjoint process defined in \eqref{adjoint}, we have
    \begin{align}
        &\frac{d}{d\epsilon}J(u+\epsilon\de v)\bigg|_{\epsilon=0}\notag \\
        =\ &\e\bigg\{\int_0^T \bigg[b_v\left(s,\theta_s\right)^\top P_s +\sum_{j=1}^n \sigma^j_v\left(s,\theta_s\right)^\top Q_s^j +\int_E \gamma_v\left(s,\theta_s,e\right)^\top R_s(e) \lambda(de)+f_v(s,\theta_s) \bigg]^\top \de v_s ds \bigg\}. \label{lem:dJ_0}
    \end{align}
\end{lemma}

As a direct consequence of Lemma~\ref{lem:dJ} and the very nature that $u$ is predictable, we now have the following necessary maximum principle for Problem $\left(\mathbf{P}^{0,\xi}\right)$. 

\begin{theorem}
    Under Assumptions (A1) and (A2), suppose that $u\in\mr_\f^2(0,T)$ is an optimal control for Problem $\left(\mathbf{P}^{0,\xi}\right)$, $Y\in\sr_\f^2(0,T)$ is the associated state process and $(P,Q,R)\in \sr_\f^2(0,T)\times\mr_\f^2(0,T)\times\kr_{\f,\lambda}^2(0,T)$ is the corresponding adjoint process. Then, the following condition holds:
    \begin{equation}\label{optimal_condition}
        L_v\left(t,Y_{t-},\lr(Y_{t-}),u_t,P_{t-},Q_t,R_t\right)=0,\quad  t\in[0,T],\ a.e., \quad \mathbb{P}-a.s..
    \end{equation}
\end{theorem}

The controlled SDE \eqref{SDE^v} for $Y$ associated with the optimal control $u$, the corresponding BSDE \eqref{adjoint} for $(P,Q,R)$, and the optimality condition \eqref{optimal_condition} altogether give the following system of forward-backward equations with jump diffusion (together with the optimality condition) for $(Y,P,Q,R,u)\in \left(\sr_\f^2\times\sr_\f^2\times\mr_\f^2\times\kr_{\f,\lambda}^2\times\mr_\f^2\right)(0,T)$:
\begin{equation}\label{FBSDE:MFTC}
    \left\{
        \begin{aligned}
            &Y_t=\xi+\int_0^t b\left(s,\theta_s\right) ds +\int_0^t \sigma \left(s,\theta_s\right)dB_s+\int_0^t \int_E \gamma\left(s,\theta_{s-},e\right) \mathring{N}(de,ds),\\
            &P_t= -\int_t^T Q_s dB_s-\int_t^T\int_E R_s(e) \mathring{N}(de,ds)+g_x(Y_T,\lr(Y_T))+\he\left[D_y\frac{dg}{d\nu}\left(\hyT,\lr(Y_T)\right)(Y_T)\right]\\
            &\qquad +\int_t^T \bigg\{ b_x\left(s,\theta_s\right)^\top P_s+\sum_{j=1}^n \sigma^j_x \left(s,\theta_s\right)^\top Q^j_s+\int_E \gamma_x\left(s,\theta_s,e\right)^\top R_s(e)\lambda(de)+f_x\left(s,\theta_s\right)\\
            &\qquad\qquad\qquad +\he\bigg[\left(D_y\frac{db}{d\nu}\left(s,\widehat{\theta_s}\right)\left(Y_s\right)\right)^\top \hps+\sum_{j=1}^n \left(D_y\frac{d\sigma^j}{d\nu}\left(s,\widehat{\theta_s}\right)\left(Y_s\right)\right)^\top \hqsj \\
            &\qquad\qquad\qquad\qquad +\int_E \left(D_y\frac{d\gamma}{d\nu}\left(s,\widehat{\theta_s},e\right)\left(Y_s\right)\right)^\top \hrs(e) \lambda(de)+ D_y\frac{df}{d\nu}\left(s,\widehat{\theta_s}\right)\left(Y_s\right) \bigg]\bigg\}ds, \\
            &b_v\left(t,\theta_{t-}\right)^\top P_{t-} +\sum_{j=1}^n \sigma^j_v\left(t,\theta_{t-}\right)^\top Q_t^j +\int_E \gamma_v\left(t,\theta_{t-},e\right)^\top R_t(e) \lambda(de)+f_v(t,\theta_{t-})=0, \quad t\in[0,T].
        \end{aligned}
    \right.
\end{equation}
Here and in the rest of this article, we always denote by $\theta_t:=(Y_t,\lr(Y_t),u_t)$ for $t\in[0,T]$; and since $u\in\mr_\f^2(0,T)$ is predictable, we also have $\theta_{t-}:=(Y_{t-},\lr(Y_{t-}),u_t)$. 

\subsection{Sufficient condition}\label{subsec:sufficient}

In this subsection, we give the sufficient condition for Problem $\left(\mathbf{P}^{0,\xi}\right)$ under the following two additional assumptions (A3) and (A4). Namely, (A3) is an assumption on the dependence of the drift coefficient $b$ and the diffusion coefficients $\sigma^j$ on different control argument components %, which requires that the dependencies of these different coefficients on the control argument do not overlap
; and it is indeed an extension of the commonly used control-independent assumption on the diffusion coefficient $\sigma$ (see \cite{SA,AB13} for instance) or the usual condition of linear-dependence on control condition for $\sigma$ (see \cite{barreiro2019linear,AB11,CR,li2020indefinite} for instance). (A4) is a convexity assumption of the cost coefficients in $x$, $m$ and $v$. 

{\bf (A3)} Let $0<d_0\le d$ and $0\le d_1,\dots,d_n\le d$ with $\sum_{j=0}^n d_j=d$. For any $v\in\brd$, we write $v=\left(v^0,v^1,\dots,v^n\right)\in \br^{d_0}\times\br^{d_1}\times\dots\times\br^{d_n}$, with $v^j$ being a sub-vector of $v$ for any $j$ with $d_j>0$.
(i) There exists a map
\begin{align}\label{def:B:v^0}
    B:[0,T]\times\brn\times\pr_2(\brn)\times\br^{d_0}\to\brn,
\end{align}
such that 
\begin{align*}
    b(t,x,m,v)=B\left(t,x,m,v^0\right),\quad v\in\brd.
\end{align*}
with the function $B$ satisfying Assumption (A1), and for any $t\in[0,T]$, $x,x'\in\brn$, $m,m'\in \pr_2(\brn)$, $v^0,{v'}^0\in\br^{d_0}$ and $y,y'\in\brn$,  
\begin{align}
    &\left|B_x\left(t,x',m',{v'}^0\right)- B_x\left(t,x,m,v^0\right)\right| \le \frac{L_0|x'-x|+L_0 W_2(m,m')+L_1\left|{v'}^0-v^0\right|}{1+|x|\vee\left|x'\right|+ |m|_1\vee\left|m'\right|_1+\left|v^0\right|\vee\left|{v'}^0\right|}, \notag \\
    &\left|D_y \frac{dB}{d\nu}\left(t,x',m',{v'}^0\right)(y)-D_y \frac{dB}{d\nu}\left(t,x,m,v^0\right)(y)\right| \notag\\
    &\ \qquad\qquad\qquad\qquad\qquad\qquad\qquad\qquad \le \frac{L_0|x'-x|+L_0 W_2(m,m')+L_0|y'-y| +L_1\left|{v'}^0-v^0\right|}{1+|x|\vee\left|x'\right|+ |m|_1\vee\left|m'\right|_1+\left|v^0\right|\vee\left|{v'}^0\right|},\notag\\
    &\left|B_{v^0}\left(t,x',m',{v'}^0\right)- B_{v^0}\left(t,x,m,v^0\right)\right| \le \frac{L_1|x'-x|+L_1 W_2(m,m')+L_2\left|{v'}^0-v^0\right|}{1+|x|\vee\left|x'\right|+ |m|_1\vee\left|m'\right|_1+\left|v^0\right|\vee\left|{v'}^0\right|},\label{generic:condition:b}
\end{align}
for some $0\le L_0,L_1,L_2\le L$; and there exists $\lambda_0>0$, such that for any $(t,x,m)\in [0,T]$,
\begin{align}\label{lambda:condition:B}
    &\left(B_{v^0}\right)\left(B_{v^0}\right)^\top\left(t,x,m,v^0\right) \geq \lambda_0 I_n,\quad \forall v^0\in\br^{d_0}.
\end{align}
(ii) For $1\le j\le n$ with $d_j>0$, there exists a map
\begin{align*}
    &\qquad A^j:[0,T]\times\brn\times\pr_2(\brn)\times\br^{d_j}\to\brn,
\end{align*}
such that 
\begin{align*}
    \sigma^j(t,x,m,v)=A^j\left(t,x,m,v^j\right), \quad v\in\brd,
\end{align*}
with the functions $A^j$ satisfying Assumption (A1), and for any $t\in[0,T]$, $x,x'\in\brn$, $m,m'\in \pr_2(\brn)$, $v^j,{v'}^j\in\br^{d_j}$ and $y,y'\in\brn$, 
\begin{align}
    &\left|A^j_x\left(t,x',m',{v'}^j\right)- A^j_x\left(t,x,m,v^j\right)\right|\le \frac{L_0|x'-x|+L_0 W_2(m,m')+L_1\left|{v'}^j-v^j\right|}{1+|x|\vee\left|x'\right|+ |m|_1\vee\left|m'\right|_1+\left|v^j\right|\vee\left|{v'}^j\right|}  ,\notag\\
    &\left|D_y \frac{dA^j}{d\nu}\left(t,x',m',{v'}^j\right)(y')-D_y \frac{dA^j}{d\nu}\left(t,x,m,v^j\right)(y)\right| \notag \\
    &\ \qquad\qquad\qquad\qquad\qquad\qquad\qquad\qquad \le \frac{L_0|x'-x|+L_0W_2(m,m')+L_0|y'-y|+L_1\left|{v'}^j-v^j\right|}{1+|x|\vee\left|x'\right|+ |m|_1\vee\left|m'\right|_1+\left|v^j\right|\vee\left|{v'}^j\right|},\notag\\
    &\left|A^j_{v^j}\left(t,x',m',{v'}^j\right)- A^j_{v^j}\left(t,x,m,v^j\right)\right| \le \frac{L_1|x'-x|+L_1 W_2(m,m')+L_2\left|{v'}^j-v^j\right|}{1+|x|\vee\left|x'\right|+ |m|_1\vee\left|m'\right|_1+\left|v^j\right|\vee\left|{v'}^j\right|};\label{generic:condition:A}
\end{align}
and for any $(t,x,m)\in [0,T]$,
\begin{align}\label{lambda:condition:A}
    &\left(A_{v^j}^j\right)\left(A_{v^j}^j\right)^\top\left(t,x,m,v^j\right) \geq \lambda_0 I_n,\quad \forall v^j\in\br^{d_j},\quad d_j>0.
\end{align}
For $1\le j\le n$ with $d_j=0$, the map $\sigma^j$ is linear in $(x,m)$, such that
\begin{align*}
    \sigma^j(t,x,m)=\sigma^j_0(t)+\sigma^j_1(t)x+\sigma^j_2(t)\int_\brn y m(dy),
\end{align*}
with the norms of the deterministic matrices $\sigma^j_0(t)$, $\sigma^j_1(t)$ and $\sigma^j_2(t)$ being bounded by $L$. \\
(iii) The map $\gamma$ is linear in $(x,m)$ and is independent on $v$, such that
\begin{align*}
    \gamma(t,x,m,e)=\gamma_0(t,e)+\gamma_1(t,e)x+\gamma_2(t,e)\int_\brn y m(dy),
\end{align*}
with the $L_\lambda^2$-norms of $\gamma_0(t,\cdot)$, $\gamma_1(t,\cdot)$ and $\gamma_2(t,\cdot)$ being bounded by $L$. \\
(iv) There exists functions $f^j:[0,T]\times\brn\times\pr_2(\brn)\times\br^{d_j}$ for $0\le j\le n$, such that for all $\brd\ni v=\left(v^0,v^1,\dots,v^n\right)\in \br^{d_0}\times\br^{d_1}\times\dots\times\br^{d_n}$,
\begin{align*}
    f(t,x,m,v)=\sum_{j=0}^n f^j\left(t,x,m,v^j\right),
\end{align*}
with the functions $f^j$ satisfying 
\begin{align*}
    \left|\left(f^j_x,f^j_v\right)(t,x,m,v^j)\right|\le \ & L\left(1+|x|+|m|_1+|v^j|\right),\\
    \left|D_y\frac{df^j}{d\nu}(t,x,m,v^j)(y)\right|\le\ & L(1+|x|+|m|_1+|v^j|+|y|),
\end{align*}
and the derivatives being $L$-Lipschitz continuous in $\left(x,m,v^j,y\right)$.\\ 

We now illustrate more on our assumptions in (A3). 
    \begin{enumerate}
        \item We denote by $l$ the number of $j\in\{1,2,\dots,n\}$ such that $d_j>0$. The separable condition in the control argument $v=(v^0,\dots,v^n)$ in Assumption (A3) is actually an extension of the control-independence assumption on the diffusion coefficient $\sigma$ which is commonly used in the literature, such as \cite{SA,AB13}; indeed, (A3) include the control-independence case as a particular example. To see this, we simply let $l=0$, i.e., $d_0=d$ and $d_1=\dots=d_n=0$, then $b=B$ and $f=f^0$, and $\sigma$ is independent of $v$.
        \item The regularity condition \eqref{generic:condition:b} on the function $B$ (resp. Condition \eqref{generic:condition:A} on $A^j$ for $d_j>0$) is an extension of the standard linear assumption on $b$ (resp. $\sigma^j$) in the literature, such as \cite{barreiro2019linear,AB11,CR,li2020indefinite}. Indeed, when the drift function $B$ is linear in $x$, $m$ and $v^0$, then the derivatives $B_x$, $B_{v_0}$ and $D_y\frac{dB}{d\nu}$ are constants, and therefore, $L_0=L_1=L_2=0$. Together with Assumption (A2) for the cost coeffieients and the convexity assumption (A4), we see that our settings can include the LQ cases as particular examples. We here also provide a non-linear example for the drift coefficient in which the function $B$ is defined as follows; for simplicity, we just consider case when the dimensions $n=d_0=1$:
        \begin{align}\label{example}
            &B\left(t,x,m,v^0\right):=x+v^0+\int_\br y \; m(dy)+ \epsilon x \exp\bigg(-x^2-\left|v^0\right|^2-\left|\int_\br \phi(y)\;m(dy)\right|^2 \bigg),
        \end{align}
        where 
        \begin{align*}
            \phi(y):=\left\{
                \begin{aligned}
                    & |y|,\quad \text{for}\ |y|\geq 1;\\
                    & -\frac{1}{8} y^4 + \frac{3}{4} y^2 +\frac{3}{8}, \quad \text{for}\ |y|< 1.
                \end{aligned}
            \right.
        \end{align*}  
        Similar regularity assumptions as Condition \eqref{generic:condition:b} is also used in the literature \cite{AB13,AB10''}; the first one is for the first order MFTC with a generic drift, and the second one is for the second order mean field games with a generic drift and linear diffusion, both without jump diffusion; and the example \eqref{example} is discussed in details in \cite[Section 11]{AB10''}. 
        \item The positive definiteness condition \eqref{generic:condition:b'} (resp. Condition \eqref{lambda:condition:A}) and its \textit{Schur complement} can give a exact formulation \eqref{chara:P} (resp. \eqref{chara:Q}) for the process $P$ (reps. $Q^j$), which together with Condition \eqref{generic:condition:b} (resp. Condition \eqref{generic:condition:A}) imply the cone property \eqref{cone:PQ:without_v} for $P$ (reps. $Q^j$) (also see \cite{AB13,AB10''}). 
    \end{enumerate}

Under Assumption (A3)-(i,ii,iv), the optimality condition in System \eqref{FBSDE:MFTC} also writes
\begin{equation}\label{optimal_condition_split}
    \begin{aligned}
        &B_{v^0}\left(t,Y_{t-},\lr(Y_{t-}),u^0_t\right)^\top P_{t-} +f^0_{v^0}\left(t,Y_{t-},\lr(Y_{t-}),u^0_t\right)=0,\\
        &A^j_{v^j}\left(t,Y_{t-},\lr(Y_{t-}),u^j_t\right)^\top Q_t^j + f^j_{v^j}\left(t,Y_{t-},\lr(Y_{t-}),u^j_t\right)=0,\quad 1\le j\le n\quad\text{with}\quad d_j>0,
    \end{aligned}
\end{equation}
where we write $u_t=\left(u^0_t,u^1_t,\dots u^n_t\right)$ corresponding to the decomposition for $v$ in Assumption (A3). The relations in  \eqref{optimal_condition_split} show that the sub-vector $u^0_t$ of $u_t$ depends on $P_{t-}$, but not on $Q_t$; while for $1\le j\le n$ with $d_j>0$, the sub-vector $u^j_t$ depends on $Q^j_t$, on neither $P_{t-}$ nor $Q^{j'}_t$ for $j'\neq j$. On the other hand, under Condition \eqref{lambda:condition:B} (resp. \eqref{lambda:condition:A}), the process $P_{t-}$ (resp. $Q^j_t$) can be viewed as a map of $\left(t,Y_{t-},\lr(Y_{t-}),u^0_t\right)$ (resp. $\left(t,Y_{t-},\lr(Y_{t-}),u^j_t\right)$), from which we can give some crucial estimates of cone property (see Proposition~\ref{prop:cone} below) for the adjoint processes $P$ and $Q^j$ defined in \eqref{FBSDE:MFTC}. To do so,  we also need the following convexity assumption for $f$ and $g$.

{\bf (A4)} There exists $\lambda_v>0$ and $\lambda_x+\lambda_m> 0$, such that for any $t\in[0,T]$, $x,x'\in\brn$, $v,v'\in\brd$ and $\xi,\xi'\in L^2(\Omega,\f,\mathbb{P};\brn)$,
\begin{align*}
    f(t,x',\lr(\xi'),v')-f(t,x,\lr(\xi),v)\geq\ & f_x(t,x,\lr(\xi),v)^\top (x'-x)+f_v(t,x,\lr(\xi),v)^\top (v'-v)\\
    &+\e\left[\left(D_y\frac{df}{d\nu}(t,x,\lr(\xi),v)(\xi)\right)^\top(\xi'-\xi)\right]\\
    &+\lambda_v|v'-v|^2+\lambda_x|x'-x|^2+\lambda_m\|\xi'-\xi\|_2^2,\\
    g(x',\lr(\xi'))-g(x,\lr(\xi))\geq\ & g_x(x,\lr(\xi))^\top (x'-x)+\e\left[\left(D_y\frac{dg}{d\nu}(x,\lr(\xi))(\xi)\right)^\top(\xi'-\xi)\right];
\end{align*}
and for any $m\in\pr_2(\brn)$ and $y,y'\in\brn$,
\begin{align*}
    \frac{df}{d\nu}(t,x,m,v)(y')-\frac{df}{d\nu}(t,x,m,v)(y)\geq\ & \left(D_y\frac{df}{d\nu}(t,x,m,v)(y)\right)^\top (y'-y),\\
    \frac{dg}{d\nu}(x,m)(y')-\frac{df}{d\nu}(x,m)(y)\geq\ & \left(D_y\frac{dg}{d\nu}(x,m)(y)\right)^\top (y'-y).
\end{align*}

The expectation terms in the convexity condition of Assumption (A4) are due to the chain rule for the functionals defined on $\pr_2(\brn)$ (also see Lemma~\ref{lem:Ito}). We now give the following cone property for any processes satisfying relations \eqref{optimal_condition_split}. 

\begin{proposition}[Cone Property]\label{prop:cone}
    Under Assumptions (A3) and (A4), if the processes $(Y,u,P,Q)$ satisfy the condition \eqref{optimal_condition_split}, then, we have
    \begin{equation}\label{cone:PQ:without_v}
    \begin{aligned}
        \left|P_{t-}\right|\le\ & \frac{L^2}{\lambda_0} \left[1+\left|Y_{t-}\right|+ \left|\lr\left(Y_{t-}\right) \right|_1 +\left|u^0_t\right| \right],\\
        \left|Q_t^j\right|\le\ & \frac{L^2}{\lambda_0} \left[1+\left|Y_{t-}\right|+\left|\lr\left(Y_{t-}\right)\right|_1 +\left|u^j_t\right| \right];
    \end{aligned}
    \end{equation}
    and also 
    \begin{equation}\label{cone:u}
    \begin{aligned}
        \left|u^{0}_t\right| \le\ & \frac{L}{2\lambda_v} \left[1+\left|P_{t-}\right|+\left|Y_{t-}\right|+\left|\lr\left(Y_{t-}\right) \right|_1 \right],\\
        \left|u^{j}_t\right| \le\ & \frac{L}{2\lambda_v} \left[1+\left|Q_t^j\right|+\left|Y_{t-}\right|+ \left|\lr\left(Y_{t-}\right)\right|_1 \right].
    \end{aligned}
    \end{equation}
\end{proposition}

\begin{proof}
From the relation \eqref{optimal_condition_split} and Condition \eqref{lambda:condition:A}, we know that for $1\le j\le n$ with $d_j>0$,
\begin{align}\label{chara:Q}
    Q_t^j =- \left(A^j_{v^j}\left(A^j_{v^j}\right)^\top \right)^{-1}\left(A^j_{v^j}\right)\left(f^j_{v^j}\right)\left(t,Y_{t-},\lr(Y_{t-}),u^j_t\right).
\end{align}
Then, from \eqref{lambda:condition:A}, the boundedness condition for $A^j_{v^j}$ and the growth condition for $f^j_{v^j}$, we obtain the estimate for $Q_t^j$ in \eqref{cone:PQ:without_v}. Again from \eqref{optimal_condition_split}, we know that 
\begin{align*}
    &\left[f^j_{v^j}\left(t,Y_{t-},\lr\left(Y_{t-}\right),u^{j}_t\right)-f_{v^j}\left(s,Y_{t-},\lr\left(Y_{t-}\right),0\right)\right]^\top u^{j}_s\\
    =\ & -\left(u^{j}_t\right)^\top A^j_{v^j}\left(t,Y_{t-},\lr\left(Y_{t-}\right),u^{j}_t\right)^\top Q^{j}_t - f_{v^j}\left(t,Y_{t-},\lr\left(Y_{t-}\right),0\right)^\top u^{j}_t,
\end{align*}
and therefore, from the convexity of $f$ in the argument $v$ in accordance with Assumption (A3), we obtain the estimate for $u^{j}_t$ in \eqref{cone:u}. Similar as \eqref{chara:Q}, we also have the following formulation of $P$:
\begin{align}\label{chara:P}
    P_{t-} =- \left(B_{v^0}\left(B_{v^0}\right)^\top \right)^{-1}\left(B_{v^0}\right)\left(f^0_{v^0}\right)\left(t,Y_{t-},\lr(Y_{t-}),u^0_t\right),
\end{align}
and by applying the similar approach on $P_{t-}$ and $u_t^0$, we obtain \eqref{cone:PQ:without_v} and \eqref{cone:u}. 
\end{proof}

Proposition~\ref{prop:cone} shows that the cone property is automatically satisfied for processes $P$ and $Q^j$ satisfying relations in \eqref{optimal_condition_split} under Assumptions (A3) and (A4). The concept of the ``cone property'' was first proposed in \cite{AB10''} for the process $P$ for the study of the first order generic MFTC problems, and it is then also used in our previous work \cite{AB12,AB13,AB10'}. %for the second order nonlinear MFGs with control-independent $\sigma$
In this article, based on the assumption on the control variable $v=(v^0,\dots,v^n)$ for the coefficients $b$ and $\sigma$ in Assumption (A3), we can further extend the previous results by giving the cone properties for both $P$ and $Q$. By using on the cone property, we now establish our sufficient maximum principle for Problem $\left(\mathbf{P}^{0,\xi}\right)$. The proof of the following theorem is given in Appendix~\ref{pf:thm:suff}.

\begin{theorem}\label{thm:suff}
    Under Assumptions (A1)-(A4), suppose that the following relations hold:
    \begin{equation}\label{thm:suff:condition}
    \begin{aligned}
        &(i)\quad 2\lambda_v> \frac{L^2 L_2}{\lambda_0}, \\
        &(ii)\quad \left(2\lambda_v- \frac{L^2L_2}{\lambda_0} \right)\left[2\lambda_x+2\lambda_m- \frac{5(l+1)L^2L_0}{\lambda_0} \right]> \frac{4(l+1)L^4L_1^2}{\lambda_0^2}. \qquad\qquad
    \end{aligned}
    \end{equation}
    If FBSDEs with jumps \eqref{FBSDE:MFTC} has a unique solution 
    \begin{align*}
        (Y,P,Q,R,u)\in \left(\sr_\f^2\times\sr_\f^2\times\mr_\f^2\times\kr_{\f,\lambda}^2\times\mr_\f^2\right)(0,T),
    \end{align*}
    then $u$ is the unique optimal control for Problem $\left(\mathbf{P}^{0,\xi}\right)$, and for any control $v\in\mr_\f^2(0,T)$, we have
    \begin{align}\label{thm:suff_0}
        J(v)-J(u) \geq \left( \lambda_v- \frac{L^2L_2}{2\lambda_0}-\frac{2(l+1)L^4L_1^2}{\lambda_0\left[2\lambda_x\lambda_0+2\lambda_m\lambda_0- 5(l+1)L^2L_0\right]}\right) \int_0^T \left\|v_s-u_s\right\|_2^2 dt.
    \end{align}
\end{theorem}

In Condition \eqref{thm:suff:condition}, (i) is the optimal while (ii) is not, due to the  straightforward usage of Young's inequality. Condition \ref{thm:suff:condition} actually states that the convexity of the cost functions is required to be larger than the constants $(L_0,L_1,L_2)$ in Conditions \eqref{generic:condition:b} and \eqref{generic:condition:A} of Assumption (A3)-(i,ii). As a particular case, when $b$ and $\sigma$ are linear, then $L_0=L_1=L_2=0$, and Condition \ref{thm:suff:condition} is simply $\lambda_v>0$, $\lambda_x+\lambda_m>0$. The next proposition shows that the Lagrangian function $L$ (defined in \eqref{Lagrangian}) admits a unique minimizer in the control argument.

\begin{proposition}\label{prop:L}
    Let Assumptions (A1)-(A4) and  Condition \eqref{thm:suff:condition}-(i) be satisfied. Then, for any $(s,x,m,r(\cdot))\in [0,T]\times\brn\times\pr_2(\brn)\times L_\lambda^2(E)$ and $p,q^j\in\brn$ ($d_j>0$), the maps
    \begin{align*}
        &\br^{d_0}\ni v^0\mapsto L\left(s,x,m,v,p,q,r\right),\\
        &\br^{d_j}\ni v^j\mapsto L\left(s,x,m,v,p,q,r\right),\quad 1\le j\le n, \quad d_j>0,
    \end{align*}
    have respective unique minimizers $\varphi^0(s,x,m,p)$ and $\varphi^j\left(s,x,m,q^j\right)$, and we have
    \begin{equation}\label{prop:L:H}
        \begin{aligned}
            &H_p(s,x,m,p,q,r)= B\left(s,x,m,\varphi^0(s,x,m,p)\right);\\
            &H_{q^j}(s,x,m,p,q,r)=  A^j\left(s,x,m,\varphi^j\left(s,x,m,q^j\right)\right); \\
            &H_x(s,x,m,p,q,r)\\
            &\qquad =  L_x \Big(s,x,m,\left(\varphi^0(s,x,m,p),\varphi^1\left(s,x,m,q^1\right),\dots,\varphi^n\left(s,x,m,q^n\right)\right),p,q,r\Big);\\
            &D_y \frac{dH}{d\nu}(s,x,m,p,q,r)(\xi)\\
            &\qquad = D_y \frac{dL}{d\nu}\Big(s,x,m,\left(\varphi^0(s,x,m,p),\varphi^1\left(s,x,m,q^1\right),\dots,\varphi^n\left(s,x,m,q^n\right)\right),p,q,r\Big)(\xi).
        \end{aligned}
    \end{equation}
\end{proposition}

\begin{proof}
Under Assumption (A3), the Lagrangian $L$ also writes for $(s,x,m,v,p,q,r)\in [0,T]\times\brn\times\pr_2(\brn)\times\brd\times\brn\times\brnn\times L_\lambda^2(E)$,
\begin{align*}
    &L(s,x,m,v,p,q,r)\\
    =\ & \sum_{j=0}^n f^j\left(s,x,m,v^j\right)+B\left(s,x,m,v^0\right)^\top p+\sum_{j=1}^n A^j\left(s,x,m,v^j\right)^\top q^j+\int_E \gamma(s,x,m,e)^\top r(e)\lambda(de),
\end{align*}
where $v=\left(v^0,v^1\dots,v^n\right)$, with $v^j\in\br^{d_j}$. Under Assumptions (A1) and (A2), we  have
\begin{align*}
    L_{v^0} (s,x,m,v,p,q,r)=\ & B_{v^0}\left(s,x,m,v^0\right)^\top p + f^0_{v^0} \left(s,x,m,v^0\right),\\
    L_{v^j} (s,x,m,v,p,q,r)=\ & A^j_{v^j}\left(s,x,m,v^j\right)^\top q^j + f^j_{v^j} \left(s,x,m,v^j\right),\quad 1\le j\le n, \quad d_j>0.
\end{align*}
The existence of the minimizer of $L$ in $v^0$ and $v^j$ is a direct result of the convexity of $f$ and Assumption (A4); also see \cite{AB10''} for similar fixed point arguments. We here only prove the uniqueness. For some fixed $\left(s,x,m,q^j\right)$, suppose that $v^j$ and $u^j$ are both minimizer of $L$ in $v^j$, then from the first order condition, we have
\begin{align}
    0=\ &\left(v^j-u^j\right)^\top \left[A^j_{v^j}\left(s,x,m,v^j\right)-A^j_{v^j}\left(s,x,m,u^j\right)\right]^\top q^j \notag \\
    &+ \left[f^j_{v^j}\left(s,x,m,v^j\right)-f^j_{v^j}\left(s,x,m,u^j\right)\right]^\top \left(v^j-u^j\right); \label{add-1}
\end{align}
and similar to Proposition~\ref{prop:cone}, we have the following cone property:
\begin{align}\label{prop:L:pq}
    \left|q^j\right|\le \frac{L^2}{\lambda_0} \left[1+|x|+|m|_1+\left|v^j\right|\right]. %\qquad \left|q^j\right|\le \frac{L^2}{\lambda_0} \left[1+|x|+|m|_1+\left|u^j\right|\right].
\end{align}
From the convexity of $f^j$ in $v^j$ in Assumption (A4), we have
\begin{align}\label{add-3}
    \left[f^j_{v^j}\left(s,x,m,v^j\right)-f^j_{v^j}\left(s,x,m,u^j\right)\right]^\top \left(v^j-u^j\right) \geq 2\lambda_v \left|v^j-u^j\right|^2;
\end{align}
and from Assumption (A3) and the cone property \eqref{prop:L:pq}, we have
\begin{align}\label{add-2}
    &\left|\left(v^j-u^j\right)^\top \left[A^j_{v^j}\left(s,x,m,v^j\right)-A^j_{v^j}\left(s,x,m,u^j\right)\right]^\top q^j \right| \notag\\
    \le\ & \left|v^j-u^j\right|\cdot \frac{L_2\left|v^j-u^j\right|}{1+|x|+|m|_1+\left|v^j\right|} \cdot \frac{L^2}{\lambda_0} \left[1+|x|+|m|_1+\left|v^j\right|\right]\notag \\
    =\ &  \frac{L^2L_2}{\lambda_0}\left|v^j-u^j\right|^2. 
\end{align}
Substituting \eqref{add-3} and \eqref{add-2} into \eqref{add-1}, we know that $\left(2\lambda_v-\frac{L^2L_2}{\lambda_0}\right) \left|v^j-u^j\right|^2\le 0$, so when Condition~\eqref{thm:suff:condition}-(i) holds, we have $v^j=u^j$. That is, the map $\varphi^j\left(s,x,m,q^j\right)$ is well-defined, and so as the map $\varphi^0\left(s,x,m,p\right)$. In a similar manner, it is also easy to check that the minimizing map $\varphi^0$ (resp. $\varphi^j$) is well-defined in the neighbourhood of $(x,m,p)$ (resp. $(x,m,q^j)$). Then, from the first order conditions, we obtain \eqref{prop:L:H}.
\end{proof}

In view of the last proposition and the cone property in Proposition~\ref{prop:cone} for the processes $P$ and $Q^j$ (with $d_j>0$), we know that FBSDEs with jumps \eqref{FBSDE:MFTC} for $\Theta_t$ also reads
\begin{equation}\label{FBSDEs:H}
    \left\{
        \begin{aligned}
            Y_t=\ & \xi+\int_0^t H_p\left(s,Y_{s-},\lr\left(Y_{s-}\right),P_{s-}\right) ds +\sum_{j,\ d_j>0}\int_0^t H_{q^j}\left(s,Y_{s-},\lr\left(Y_{s-}\right),Q^j_s\right)dB^j_s\\
            &+\sum_{j,\ d_j=0} \int_0^t \left(\sigma^j_0(s)+\sigma^j_1(s)Y_{s-}+\sigma^j_2(s)\e\left[Y_{s-}\right] \right) dB^j_s +\int_0^s \int_E \gamma\left(s,Y_{s-},\lr\left(Y_{s-}\right),e\right) \mathring{N}(de,ds),\\
            P_s=\ & g_x\left(Y_T,\lr\left(Y_T\right)\right)+\he\left[D_y\frac{dg}{d\nu}\left(\widehat{Y_T},\lr\left(Y_T\right)\right)\left(Y_T\right)\right]\\
            &+\int_t^T \bigg\{ H_x\left(s,Y_{s-},\lr\left(Y_{s-}\right),P_{s-},Q_s,R_s\right)\\
            &\quad\qquad +\he\bigg[ D_y\frac{dH}{d\nu}\left(s,\widehat{Y_{s-}},\lr\left(Y_{s-}\right),\widehat{P_{s-}},\widehat{Q_s},\widehat{R_s}\right)\left(Y_{s-}\right) \bigg]\bigg\}dr\\
            &-\int_t^T Q_s dB_s-\int_s^T\int_E R_s(e) \mathring{N}(de,ds), \quad t\in[0,T];
        \end{aligned}
    \right.
\end{equation}
%with the cone property \eqref{cone:PQ:without_v}; 
and when FBSDEs with jumps \eqref{FBSDEs:H} has a solution, the optimal control for Problem $\left(\mathbf{P}^{0,\xi}\right)$ is of the following feedback form:
\begin{equation}\label{u_feedback}
    \begin{aligned}
        &u_t^0=\varphi^0\left(t,Y_{t-},\lr\left(Y_{t-}\right),P_{t-}\right);\quad u_t^j=\varphi^j\left(t,Y_{t-},\lr\left(Y_{t-}\right),Q^j_t \right),\quad d_j>0.
    \end{aligned}
\end{equation}

\section{FBSDEs with jumps}\label{sec:FBSDE}

We next give the well-posedness of the system of FBSDEs with jumps \eqref{FBSDE:MFTC} and also the regularity of the solutions with respect to the initial condition $\xi$. For notational convenience, from this section, we denote by $\Theta$ the process $(Y,P,Q,R,u)\in \left(\sr_\f^2\times\sr_\f^2\times\mr_\f^2\times\kr_{\f,\lambda}^2\times\mr_\f^2\right)(0,T)$, and denote by $\theta_t:=(Y_t,\lr(Y_t),u_t)$ for $t\in[0,T]$, and denote by $\mathbb{S}$ the space of processes $\Theta$ with the squared  norm
\begin{equation*}
	\begin{split}
		\|\Theta\|_{\mathbb{S}}^2:&=\e\bigg[\sup_{0\le t\le T}|Y_t|^2+\sup_{0\le t\le T}|P_t|^2+\int_0^T\left(|Q_t|^2+|u_t|^2+\int_E |R_t(e)|^2\lambda(de)\right)dt\bigg]<\infty.
	\end{split}
\end{equation*}

\subsection{Well-posedness, boundedness and continuity}\label{subsec:well-posedness}

We begin by giving the well-posedness result, and the $L^2$-boundedness and continuity of the solution $\Theta$ with respect to $\xi$ under Assumptions (A1)-(A4).

\begin{theorem}\label{main2_thm}
	Under Assumptions (A1)-(A4) and also the validity of Condition \eqref{thm:suff:condition}-(i,ii), FBSDEs with jumps \eqref{FBSDE:MFTC} has a unique solution $\Theta=(Y,P,Q,R,u)\in\mathbb{S}$. There exists a positive constant $C$ depending only on $(l,L,L_0,L_1,L_2,\lambda_0,\lambda_x,\lambda_m,\lambda_v,T)$, such that for any initial conditions $\xi^1,\xi^2\in \lr_{\f_0}^2$, the corresponding solutions $\Theta^1$ and $\Theta^2$ of FBSDEs with jumps \eqref{FBSDE:MFTC} satisfy
    \begin{align}
        \left\|\Theta^1\right\|_{\mathbb{S}}\le C \left( 1+\left\|\xi^1\right\|_2\right), \quad \left\|\Theta^2-\Theta^1 \right\|_{\mathbb{S}} \le C\left\|\xi^2-\xi^1\right\|_2.\label{lem:kappa_0}
    \end{align}   
    Moreover, the processes $P$ and $Q$ satisfy the cone property \eqref{cone:PQ:without_v}.
\end{theorem}

\begin{proof}
The well-posedness of FBSDEs with jumps \eqref{FBSDE:MFTC} can be proven by the method of continuation in the coefficients which was first introduced by Hu and Peng \cite{YH2}, almost exactly same to the proofs of the well-posedness for FBSDEs driven purely by Brownian motions in our previous works \cite{AB11,AB12,AB10}, with the only difference being the appearance of the jump-diffusion term. But it is not a matter in view of the extension for the method of continuation including the jump-diffusion by Wu \cite{zhen1999forward}. Also see \cite{MR3346708,MR4199898} for the application of the method of continuation for mean field FBSDEs with jump-diffusion under different settings. We also refer to \cite{AB10''} for a different kind of proof for the well-posedness for mean field FBSDEs, which follows a the temporal partition approach. We can see that in both approaches, the crucial step is to establish a consistent Lipschitz-continuity of the solution with some parameters, similar as the second estimate in \eqref{lem:kappa_0}. Therefore, we here only give the proof of the second estimate in \eqref{lem:kappa_0}, and the proof for \eqref{lem:kappa_0} is similar and standard.

We denote by $\Theta^1_s=(y_s,p_s,q_s,r_s,v_s)$ and $\Theta^2_s=(Y_s,P_s,Q_s,R_s,u_s)$, and denote by $\de\xi:=\xi^2-\xi^1$, and
\begin{align*}
    &\de\Theta_s:=\Theta^2_s-\Theta^1_s=\left(Y_s-y_s,P_s-p_s,Q_s-q_s,R_s-r_s,u_s-v_s\right)=\left(\de y_s,\de p_s,\de q_s,\de r_s,\de v_s\right).
\end{align*}
From Assumption (A3)-(i,ii), we know that $\de y_t$ satisfies the following SDE with jump
\begin{align}
    \de y_t=\ & \de \xi +\int_0^t \bigg\{ \int_0^1 \bigg[B_x\left(s,Y^h_s,\lr\left(Y^h_s\right),u^{h,0}_s\right)\de y_s  \notag \\
    &\quad\qquad\qquad\qquad + \he\bigg[\left(D_y\frac{dB}{d\nu}\left(s,Y^h_s,\lr\left(Y^h_s\right),u^{h,0}_s\right)\left(\widehat{Y^h_s}\right)\right) \left(\widehat{\de y_s}\right) \bigg]  \notag \\
    &\quad\qquad\qquad\qquad +B_{v^0}\left(s,Y^h_s,\lr\left(Y^h_s\right),u^{h,0}_s\right)\de v^0_s\bigg]dh \bigg\} ds  \notag \\
    &+\sum_{j,\ d_j>0} \int_0^t \bigg\{ \int_0^1 \bigg[A^j_x\left(s,Y^h_s,\lr\left(Y^h_s\right),u^{h,j}_s\right)\de y_s  \notag \\
    &\qquad\qquad\qquad\qquad + \he\bigg[\left(D_y\frac{dA^j}{d\nu}\left(s,Y^h_s,\lr\left(Y^h_s\right),u^{h,j}_s\right)\left(\widehat{Y^h_s}\right)\right) \left(\widehat{\de y_s}\right) \bigg]  \notag \\
    &\qquad\qquad\qquad\qquad +A^j_{v^j}\left(s,Y^h_s,\lr\left(Y^h_s\right),u^{h,j}_s\right)\de v^j_s\bigg]dh \bigg\}dB^j_s \notag \\
    &+\sum_{j,\ d_j=0} \int_0^t\left\{\sigma^j_1(s)\de y_s+\sigma^j_2(s)\e[\de y_s] \right\} dB^j_s \notag \\
    &+\int_0^t \int_E \left\{\gamma_1(s,e)\de y_{s-}+\gamma_2(s,e)\e\left[\de y_{s-}\right] \right\}\mathring{N}(de,ds),\quad t\in[0,T]; \label{lem:kappa_sde}
\end{align}
and $(\de p_t,\de q_t,\de r_t)$ satisfy the following BSDE with jump
\begin{align}
    \de p_t=\ & -\int_t^T \de q_s dB_s-\int_t^T\int_E \de r_s(e) \mathring{N}(de,ds)+ \left[g_x(Y_T,\lr(Y_T))-g_x(y_T,\lr(y_T))\right] \notag \\
    &+\he\left[D_y\frac{dg}{d\nu}\left(\hyT,\lr(Y_T)\right)(Y_T)-D_y\frac{dg}{d\nu}\left(\widehat{y_T},\lr(y_T)\right)(y_T)\right] \notag \\
    &+\int_t^T \bigg\{B_x\left(s,Y_s,\lr(Y_s),u^0_s\right)^\top \de p_s+\left[B_x\left(s,Y_s,\lr(Y_s),u^0_s\right)-B_x\left(s,y_s,\lr(y_s),v^0_s\right)\right]^\top p_s \notag \\
    &\qquad +\he\left[\left(D_y\frac{dB}{d\nu}\left(s,\widehat{Y_s},\lr(Y_s),\widehat{u^0_s}\right)\left(Y_s\right)\right)^\top \widehat{\de p_s}\right] \notag \\
    &\qquad +\he\left[\left(D_y\frac{dB}{d\nu}\left(s,\widehat{Y_s},\lr(Y_s),\widehat{u^0_s}\right)\left(Y_s\right)-D_y\frac{dB}{d\nu}\left(s,\widehat{y_s},\lr(y_s),\widehat{v^0_s}\right)\left(y_s\right)\right)^\top \widehat{p_s}\right] \notag \\
    &\qquad +\sum_{j,d_j>0}^n \bigg\{ A^j_x \left(s,Y_s,\lr(Y_s),u^j_s\right)^\top \de q^j_s+\left[A^j_x \left(s,Y_s,\lr(Y_s),u^j_s\right)-A^j_x \left(s,y_s,\lr(y_s),v^j_s\right)\right]^\top q^j_s \notag \\
    &\quad\qquad\qquad +\he\bigg[\left(D_y\frac{dA^j}{d\nu}\left(s,\widehat{Y_s},\lr(Y_s),\widehat{u^j_s}\right)\left(Y_s\right)\right)^\top \widehat{\de q^j_s}\bigg] \notag \\
    &\quad\qquad\qquad +\he\bigg[\left(D_y\frac{dA^j}{d\nu}\left(s,\widehat{Y_s},\lr(Y_s),\widehat{u^j_s}\right)\left(Y_s\right)-D_y\frac{dA^j}{d\nu}\left(s,\widehat{y_s},\lr(y_s),\widehat{v^j_s}\right)\left(y_s\right)\right)^\top \widehat{q^j_s}\bigg]\bigg\} \notag \\
    &\qquad +\sum_{j,d_j=0}^n \bigg\{\sigma^j_1 (s)^\top \de q^j_s+\sigma^j_2 (s)^\top \e \left[\de q^j_s\right] \bigg\} \notag \\
    &\qquad +\int_E \bigg\{\gamma_1(s,e)^\top \de r_s(e)+\gamma_2\left(s,e\right)^\top\e[\de r_s(e)]\bigg\} \lambda(de) \notag \\
    &\qquad +f_x\left(s,Y_s,\lr(Y_s),u_s\right)-f_x\left(s,y_s,\lr(y_s),v_s\right) \notag \\
    &\qquad +\he\bigg[ D_y\frac{df}{d\nu}\left(s,\widehat{Y_s},\lr(Y_s),\widehat{u_s}\right)\left(Y_s\right) - D_y\frac{df}{d\nu}\left(s,\widehat{y_s},\lr(y_s),\widehat{v_s}\right)\left(y_s\right)\bigg] \bigg\}ds,\quad t\in[0,T], \label{lem:kappa_bsde}
\end{align}
where $Y^h_t=y_t+h(Y_t-y_t)$, $u^{h}_t:=v_t+h\left(u_t-v_t\right)$ and $u^{h,j}_t:=v^j_t+h\left(u^j_t-v^j_t\right)$ for $0\le j\le n$. Also, the following optimality conditions hold:
\begin{align}
    0=\ &B_{v^0}\left(t,Y_{t-},\lr(Y_{t-}),u^0_t\right)^\top P_{t-}- B_{v^0}\left(t,y_{t-},\lr(y_{t-}),v^0_t\right)^\top p_{t-} \notag \\
    &+\left[f^0_{v^0}\left(t,Y_{t-},\lr(Y_{t-}),u^0_t\right) - f^0_{v^0}\left(t,y_{t-},\lr(y_{t-}),v^0_t\right)\right]; \label{lem:kappa_condition_B} \\
    0=\ &A^j_{v^j}\left(t,Y_{t-},\lr(Y_{t-}),u^j_t\right)^\top Q_t^j -A^j_{v^j}\left(t,y_{t-},\lr(y_{t-}),v^j_t\right)^\top q_t^j \notag \\
    &+ \left[f^j_{v^j}\left(t,Y_{t-},\lr(Y_{t-}),u^j_t\right)-f^j_{v^j}\left(t,y_{t-},\lr(y_{t-}),v^j_t\right)\right],\quad 1\le j\le n\quad\text{with}\quad d_j>0. \label{lem:kappa_condition_A}
\end{align}
From Proposition \ref{prop:cone}, we know that the processes $P$, $p$ and $Q^j$, $q^j$ (with $d_j>0$) satisfy the following cone properties:
\begin{align}
    \left|P_{t-}\right|\le\ & \frac{L^2}{\lambda_0} \left[1+|Y_{t-}|+\left|\lr(Y_{t-})\right|_1 +\left|u_t^0\right|\right]; \notag \\
    \left|Q_t^j\right|\le\ & \frac{L^2}{\lambda_0} \left[1+|Y_{t-}|+\left|\lr(Y_{t-})\right|_1+\left|u_t^j\right|\right],\quad 1\le j\le n\quad\text{with}\quad d_j>0; \notag \\
    \left|p_{t-}\right|\le\ &  \frac{L^2}{\lambda_0}\left[1+|y_{t-}|+\left|\lr(y_{t-})\right|_1+\left|v_t^0\right|\right]; \notag \\
    \left|q_t^j\right|\le\ &  \frac{L^2}{\lambda_0}\left[1+|y_{t-}|+\left|\lr(y_{t-})\right|_1+\left|v_t^j\right|\right],\quad 1\le j\le n\quad\text{with}\quad d_j>0. \label{lem:kappa_cone_q}
\end{align}
From the optimality conditions \eqref{lem:kappa_condition_B} and \eqref{lem:kappa_condition_A},  and the cone property \eqref{lem:kappa_cone_q} and Assumption (A3)-(i), we can compute that 
\begin{align*}
    & |\de p_{t-}| \notag \\
    =\ &\Bigg| \left[\left(B_{v^0}\right)\left(\left(B_{v^0}\right)\left(B_{v^0}\right)^\top\right)^{-1}\left(t,Y_{t-},\lr(Y_{t-}),u^0_t\right)\right] \\
    &\quad \times \left[B_{v^0}\left(t,Y_{t-},\lr(Y_{t-}),u^0_t\right)-B_{v^0}\left(t,y_{t-},\lr(y_{t-}),v^0_t\right)\right]^\top p_{t-} \notag \\
    &+\left[\left(B_{v^0}\right)\left(\left(B_{v^0}\right)\left(B_{v^0}\right)^\top\right)^{-1}\left(t,Y_{t-},\lr(Y_{t-}),u^0_t\right)\right]\\
    &\qquad \times \left[f^0_{v^0}\left(t,Y_{t-},\lr(Y_{t-}),u^0_t\right) - f^0_{v^0}\left(t,y_{t-},\lr(y_{t-}),v^0_t\right)\right]\Bigg| \notag \\
    \le\ & \frac{L}{\lambda_0} \cdot \frac{L_1\left|Y_{t-}-y_{t-}\right|+L_1 \left\|Y_{t-}-y_{t-}\right\|_2 +L_2\left|v^0_t-u^0_t\right|}{1+|y_{t-}|+\left|\lr(y_{t-})\right|_1+\left|v_t^0\right|}\cdot  \frac{L^2}{\lambda_0} \left[1+|y_{t-}|+\left|\lr(y_{t-})\right|_1 +\left|v_t^0\right| \right]\\
    &+ \frac{L^2}{\lambda_0}\Big(\left|Y_{t-}-y_{t-}\right|+\left\|Y_{t-}-y_{t-}\right\|_2 + \left|v^0_t-u^0_t\right| \Big)\\
    =\ & \left(\frac{L^2}{\lambda_0}+ \frac{L^3L_1}{\lambda_0^2} \right) \left|Y_{t-}-y_{t-}\right|+\left(\frac{L^2}{\lambda_0}+ \frac{L^3L_1}{\lambda_0^2} \right) \left\|Y_{t-}-y_{t-}\right\|_2 +\left(\frac{L^2}{\lambda_0}+ \frac{L^3L_2}{\lambda_0^2} \right)\left|v^0_t-u^0_t\right|,
\end{align*}
and in a similar way, for $1\le j\le n$ with $d_j>0$, 
\begin{align*}
    \left|\de q^j_t\right| \le\ & \left(\frac{L^2}{\lambda_0}+ \frac{L^3L_1}{\lambda_0^2} \right) \left|Y_{t-}-y_{t-}\right|+\left(\frac{L^2}{\lambda_0}+ \frac{L^3L_1}{\lambda_0^2} \right) \left\|Y_{t-}-y_{t-}\right\|_2 + \left(\frac{L^2}{\lambda_0}+ \frac{L^3L_2}{\lambda_0^2} \right)\left|v^0_t-u^0_t\right|. %\label{lem:kappa_cone_deq}
\end{align*}
By applying It\^o's formula for $\de p_t^\top \de y_t$ and taking expectation, and using the Fubini's theorem (similar to \eqref{fact:Fubini}) and the fact that $Y,y\in\sr^2_\f(0,T)$ (similar to \eqref{fact:llrc}), and also using the optimality conditions \eqref{lem:kappa_condition_B} and \eqref{lem:kappa_condition_A}, we have
\begin{align}
    &\e\left[\de p_T^\top \de y_T-\de p_0^\top \de \xi\right] \notag \\
    =\ & \e\int_0^T \int_0^1 P_s^\top \bigg\{ \left[B_x\left(s,Y^h_s,\lr\left(Y^h_s\right),u^{h,0}_s\right)-B_x\left(s,Y_s,\lr(Y_s),u^0_s\right)\right]\de y_s \notag \\
    &\qquad\qquad + \he\bigg[\left(D_y\frac{dB}{d\nu}\left(s,Y^h_s,\lr\left(Y^h_s\right),u^{h,0}_s\right)\left(\widehat{Y^h_s}\right)-D_y\frac{dB}{d\nu}\left(s,Y_s,\lr(Y_s),u^0_s\right)\left(\hys\right)\right) \widehat{\de y_s} \bigg]  \notag \\
    &\qquad\qquad +\left[B_{v^0}\left(s,Y^h_s,\lr\left(Y^h_s\right),u^{h,0}_s\right)-B_{v^0}\left(t,Y_s,\lr(Y_s),u^0_s\right)\right]\de v^0_s\bigg\} dh ds \notag \\
    &- \e\int_0^T \int_0^1 p_s^\top \bigg\{ \left[B_x\left(s,Y^h_s,\lr\left(Y^h_s\right),u^{h,0}_s\right) - B_x\left(s,y_s,\lr(y_s),v^0_s\right)\right] \de y_s \notag \\
    &\qquad\qquad +\he\bigg[\left(D_y\frac{dB}{d\nu}\left(s,Y^h_s,\lr\left(Y^h_s\right),u^{h,0}_s\right)\left(\widehat{Y^h_s}\right)-D_y\frac{dB}{d\nu}\left(s,y_s,\lr(y_s),v^0_s\right)\left(\widehat{y_s}\right)\right) \widehat{\de y_s} \bigg]  \notag \\
    &\qquad\qquad +\left[B_{v^0}\left(s,Y^h_s,\lr\left(Y^h_s\right),u^{h,0}_s\right)-B_{v^0}\left(s,y_s,\lr(y_s),v^0_s\right)\right]\de v^0_s \bigg\}dh ds  \notag \\
    &+\sum_{j,\ d_j>0}\int_0^T \int_0^1 \left(Q^j_s\right)^\top \bigg\{ \left[A^j_x\left(s,Y^h_s,\lr\left(Y^h_s\right),u^{h,j}_s\right)-A^j_x\left(s,Y_s,\lr(Y_s),u^j_s\right)\right]\de y_s \notag \\
    &\qquad\qquad + \he\bigg[\left(D_y\frac{dA^j}{d\nu}\left(s,Y^h_s,\lr\left(Y^h_s\right),u^{h,j}_s\right)\left(\widehat{Y^h_s}\right)-D_y\frac{dA^j}{d\nu}\left(s,Y_s,\lr(Y_s),u^j_s\right)\left(\hys\right)\right) \widehat{\de y_s} \bigg]  \notag \\
    &\qquad\qquad +\left[A^j_{v^j}\left(s,Y^h_s,\lr\left(Y^h_s\right),u^{h,j}_s\right)-A^j_{v^j}\left(t,Y_s,\lr(Y_s),u^j_s\right)\right]\de v^0_s\bigg\} dh ds \notag \\  
    &-\sum_{j,\ d_j>0} \int_0^T \int_0^1 \left(q^j_s\right)^\top \bigg\{ \left[A^j_x\left(s,Y^h_s,\lr\left(Y^h_s\right),u^{h,j}_s\right) - A^j_x\left(s,y_s,\lr(y_s),v^j_s\right)\right] \de y_s \notag \\
    &\qquad\qquad +\he\bigg[\left(D_y\frac{dA^j}{d\nu}\left(s,Y^h_s,\lr\left(Y^h_s\right),u^{h,j}_s\right)\left(\widehat{Y^h_s}\right)-D_y\frac{dA^j}{d\nu}\left(s,y_s,\lr(y_s),v^j_s\right)\left(\widehat{y_s}\right)\right) \widehat{\de y_s} \bigg]  \notag \\
    &\qquad\qquad +\left[A^j_{v^j}\left(s,Y^h_s,\lr\left(Y^h_s\right),u^{h,j}_s\right)-A^j_{v^j}\left(s,y_s,\lr(y_s),v^j_s\right)\right]\de v^0_s \bigg\} dh ds \notag \\  
    &-\e\int_0^T \bigg\{ \left[f_x\left(s,Y_s,\lr(Y_s),u_s\right)-f_x\left(s,y_s,\lr(y_s),v_s\right)\right]^\top \de y_s  \notag \\
    &\qquad\qquad +\left[f_v\left(s,Y_s,\lr(Y_s),u_s\right)-f_v\left(s,y_s,\lr(y_s),v_s\right)\right]^\top \de v_s  \notag \\
    &\qquad\qquad +\he\bigg[ \left(D_y\frac{df}{d\nu}\left(s,Y_s,\lr(Y_s),u_s\right)\left(\hys\right) - D_y\frac{df}{d\nu}\left(s,y_s,\lr(y_s),v_s\right)\left(\widehat{y_s}\right)\right)^\top \widehat{\de y_s}\bigg] \bigg\} ds. \label{lem:kappa_1}
\end{align}
From the cone property \eqref{lem:kappa_cone_q} and Assumption (A3)-(i), we have the following estimate on the first and second terms of the right hand side of \eqref{lem:kappa_1}:
\begin{align*}
    &\Bigg|\e\int_0^T \int_0^1 P_s^\top \bigg\{ \left[B_x\left(s,Y^h_s,\lr\left(Y^h_s\right),u^{h,0}_s\right)-B_x\left(s,Y_s,\lr(Y_s),u^0_s\right)\right]\de y_s \notag \\
    &\qquad\qquad\qquad + \he\bigg[\left(D_y\frac{dB}{d\nu}\left(s,Y^h_s,\lr\left(Y^h_s\right),u^{h,0}_s\right)\left(\widehat{Y^h_s}\right)-D_y\frac{dB}{d\nu}\left(s,Y_s,\lr(Y_s),u^0_s\right)\left(\hys\right)\right) \widehat{\de y_s} \bigg]  \notag \\
    &\qquad\qquad\qquad +\left[B_{v^0}\left(s,Y^h_s,\lr\left(Y^h_s\right),u^{h,0}_s\right)-B_{v^0}\left(t,Y_s,\lr(Y_s),u^0_s\right)\right]\de v^0_s\bigg\} dh ds \Bigg|\\
    \le\ &  \frac{L^2}{\lambda_0} \e\int_0^T \bigg(  \frac{5}{2} L_0\|\de y_s\|_2^2+L_1\left|\de v^0_s\right|\cdot |\de y_s|+L_1\left|\de v^0_s\right|\cdot \|\de y_s\|_2+\frac{L_2}{2} \left|\de v^0_s\right|^2 \bigg) ds; \\
    &\Bigg|\e\int_0^T \int_0^1 p_s^\top \bigg\{ \left[B_x\left(s,Y^h_s,\lr\left(Y^h_s\right),u^{h,0}_s\right) - B_x\left(s,y_s,\lr(y_s),v^0_s\right)\right] \de y_s \notag \\
    &\qquad\qquad\qquad +\he\bigg[\left(D_y\frac{dB}{d\nu}\left(s,Y^h_s,\lr\left(Y^h_s\right),u^{h,0}_s\right)\left(\widehat{Y^h_s}\right)-D_y\frac{dB}{d\nu}\left(s,y_s,\lr(y_s),v^0_s\right)\left(\widehat{y_s}\right)\right) \widehat{\de y_s} \bigg]  \notag \\
    &\qquad\qquad\qquad +\left[B_{v^0}\left(s,Y^h_s,\lr\left(Y^h_s\right),u^{h,0}_s\right)-B_{v^0}\left(s,y_s,\lr(y_s),v^0_s\right)\right]\de v^0_s \bigg\}dh ds \Bigg|\\
    \le\ & \frac{L^2}{\lambda_0} \e\int_0^T \bigg(  \frac{5}{2} L_0\|\de y_s\|_2^2+L_1\left|\de v^0_s\right|\cdot |\de y_s|+L_1\left|\de v^0_s\right|\cdot \|\de y_s\|_2+\frac{L_2}{2} \left|\de v^0_s\right|^2 \bigg) ds ;
\end{align*}
and similarly, we have the following estimate on the third and fourth terms of the right hand side of \eqref{lem:kappa_1}: for $1\le j\le n$ with $d_j>0$, 
\begin{align*}
    &\Bigg|\e\int_0^T \int_0^1 \left(Q^j_s\right)^\top \bigg\{ \left[A^j_x\left(s,Y^h_s,\lr\left(Y^h_s\right),u^{h,j}_s\right)-A^j_x\left(s,Y_s,\lr(Y_s),u^j_s\right)\right]\de y_s \notag \\
    &\quad\qquad\qquad + \he\bigg[\left(D_y\frac{dA^j}{d\nu}\left(s,Y^h_s,\lr\left(Y^h_s\right),u^{h,j}_s\right)\left(\widehat{Y^h_s}\right)-D_y\frac{dA^j}{d\nu}\left(s,Y_s,\lr(Y_s),u^j_s\right)\left(\hys\right)\right) \widehat{\de y_s} \bigg]  \notag \\
    &\quad\qquad\qquad +\left[A^j_{v^j}\left(s,Y^h_s,\lr\left(Y^h_s\right),u^{h,j}_s\right)-A^j_{v^j}\left(t,Y_s,\lr(Y_s),u^j_s\right)\right]\de v^0_s\bigg\} dh ds \Bigg|\\
    \le\ & \frac{L^2}{\lambda_0} \e\int_0^T \bigg(  \frac{5}{2} L_0\|\de y_s\|_2^2+L_1\left|\de v^j_s\right|\cdot |\de y_s|+L_1\left|\de v^j_s\right|\cdot \|\de y_s\|_2+\frac{L_2}{2} \left|\de v^j_s\right|^2 \bigg) ds; \\
    &\e\Bigg|\int_0^T \int_0^1 \left(q^j_s\right)^\top \bigg\{ \left[A^j_x\left(s,Y^h_s,\lr\left(Y^h_s\right),u^{h,j}_s\right) - A^j_x\left(s,y_s,\lr(y_s),v^j_s\right)\right] \de y_s \notag \\
    &\quad\qquad\qquad +\he\bigg[\left(D_y\frac{dA^j}{d\nu}\left(s,Y^h_s,\lr\left(Y^h_s\right),u^{h,j}_s\right)\left(\widehat{Y^h_s}\right)-D_y\frac{dA^j}{d\nu}\left(s,y_s,\lr(y_s),v^j_s\right)\left(\widehat{y_s}\right)\right) \widehat{\de y_s} \bigg]  \notag \\
    &\quad\qquad\qquad +\left[A^j_{v^j}\left(s,Y^h_s,\lr\left(Y^h_s\right),u^{h,j}_s\right)-A^j_{v^j}\left(s,y_s,\lr(y_s),v^j_s\right)\right]\de v^0_s \bigg\} dh ds \Bigg| \notag \\
    \le\ &  \frac{L^2}{\lambda_0} \e\int_0^T \bigg(  \frac{5}{2} L_0\|\de y_s\|_2^2+L_1\left|\de v^j_s\right|\cdot |\de y_s|+L_1\left|\de v^j_s\right|\cdot \|\de y_s\|_2+\frac{L_2}{2} \left|\de v^j_s\right|^2 \bigg) ds.
\end{align*}
Substituting the last four inequalities back into \eqref{lem:kappa_1}, using the convexity of $f$ in Assumption (A4) and Condition \eqref{thm:suff:condition}-(i,ii), we have
\begin{align}
    &\e\left[\de p_T^\top \de y_T-\de p_0^\top \de \xi\right] \notag \\
    \le\ &-\int_0^T \bigg[\left(2\lambda_x+2\lambda_m- \frac{5(l+1)L^2L_0}{\lambda_0} \right)\|\de y_s\|_2^2 +\left(2\lambda_v- \frac{L^2L_2}{\lambda_0} \right) \|\de v_s\|_2^2 \notag \\
    &\quad\qquad - \frac{4\sqrt{l+1}L^2L_1}{\lambda_0} \left\|\de v_s\right\|_2\cdot \|\de y_s\|_2 \bigg] ds \notag \\
    =\ &-\int_0^T \Bigg[\Bigg(\sqrt{2\lambda_x+2\lambda_m- \frac{5(l+1)L^2L_0}{\lambda_0} }\|\de y_s\|_2 \notag \\
    &\quad\qquad\qquad -\frac{2\sqrt{l+1}L^2L_1}{\sqrt{\lambda_0\left[2\lambda_x\lambda_0+2\lambda_m\lambda_0- 5(l+1)L^2L_0\right]}}\|\de v_s\|_2\Bigg)^2 \notag \\
    &\quad\qquad +\left(2\lambda_v- \frac{L^2L_2}{\lambda_0}- \frac{4(l+1)L^4L_1^2}{\lambda_0\left[2\lambda_x\lambda_0+2\lambda_m\lambda_0-5(l+1)L^2L_0\right]} \right) \|\de v_s\|_2^2 \Bigg] ds \notag \\
    \le\ &- \left(2\lambda_v- \frac{L^2L_2}{\lambda_0}- \frac{4(l+1)L^4L_1^2}{\lambda_0\left[2\lambda_x\lambda_0+2\lambda_m\lambda_0 -5(l+1)L^2L_0\right]} \right) \int_0^T \|\de v_s\|_2^2 ds; \label{lem:kappa_2}
\end{align}
here, the coefficient of the last line of \eqref{lem:kappa_2} is strictly positive due to Condition \eqref{thm:suff:condition}-(ii). Similarly, from Fubini's theorem and Assumption (A4), we can also have
\begin{align*}
    \e\left[\de p_T^\top \de y_T\right]=\ &  \e\bigg\{\left[g_x(Y_T,\lr(Y_T))-g_x(y_T,\lr(y_T))\right]^\top \de y_T \notag \\
    &\quad +\he\bigg[\bigg(D_y\frac{dg}{d\nu}\left(\hyT,\lr(Y_T)\right)\left(Y_T\right)-D_y\frac{dg}{d\nu}\left(\widehat{y_T},\lr(y_T)\right)\left(y_T\right)\bigg)^\top \de y_T\bigg] \bigg\}\\
    =\ &  \e\bigg\{\left[g_x(Y_T,\lr(Y_T))-g_x(y_T,\lr(y_T))\right]^\top \de y_T \notag \\
    &\quad +\he\bigg[\left(D_y\frac{dg}{d\nu}\left(Y_T,\lr(Y_T)\right)\left(\hyT\right)-D_y\frac{dg}{d\nu}\left(y_T,\lr(y_T)\right)\left(\widehat{y_T}\right)\right)^\top \widehat{\de y_T}\bigg] \bigg\}\\
    \geq\ & 0.
\end{align*}
Substituting the last inequality into \eqref{lem:kappa_2}, we have
\begin{align*}
     \int_0^T \|\de v_s\|_2^2 ds \notag \le\ & C_1\e\bigg[\de p_0^\top \de \xi \bigg],
\end{align*}
where $C_1$ is a constant depending only on $(l,L,L_0,L_1,L_2,\lambda_0,\lambda_x,\lambda_m,\lambda_v)$. Then, we know that for any $\epsilon\in(0,1)$, we have
\begin{align}
    & \int_0^T \|\de v_s\|_2^2 ds \le \epsilon \e\left[|\de p_0|^2\right] +\frac{C_1}{4\epsilon} \|\de\xi\|_2^2. \label{lem:kappa_3}
\end{align}
Applying a similar approach as used in the proof of \eqref{lem:control-state-1} to SDE \eqref{lem:kappa_sde}, from Assumption (A1), we have
\begin{align}
    &\e\bigg[\sup_{0\le s\le T} |\de y_s|^2\bigg] \le C(L,T) \e\bigg[|\de\xi|^2+\int_0^T |\de v_s|^2 ds\bigg]. \label{lem:kappa_sde_estimate}
\end{align}
Then, by applying a similar approach as used in the proof of \eqref{adjoint:boundedness} to BSDE \eqref{lem:kappa_bsde}, and using the cone properties in \eqref{lem:kappa_cone_q} and Estimate \eqref{lem:kappa_sde_estimate}, we have 
\begin{align}
        &\e\bigg[\sup_{0\le t\le T} |\de p_s|^2+\int_0^T \left(|\de q_s|^2+\int_E |\de r_s(e)|^2 \lambda(de)\right) ds \bigg] \notag \\
        \le\ & C(L,T,\lambda_0) \e\bigg[|\de y_T|^2+\int_0^T \left(|\de y_s|^2+|\de v_s|^2 \right) ds\bigg] \notag \\
        \le\ & C(L,T,\lambda_0) \left(\|\de\xi\|_2^2+\e\int_0^T|\de v_s|^2 ds\right). \label{lem:kappa_bsde_estimate}
\end{align}
Substituting \eqref{lem:kappa_sde_estimate} and \eqref{lem:kappa_bsde_estimate} back into \eqref{lem:kappa_3}, we have
\begin{align*}
    & \int_0^T \|\de v_s\|_2^2 ds \le \epsilon C(L,T,\lambda_0) \int_0^T \|\de v_s\|_2^2 ds +\left(\epsilon C(L,T,\lambda_0)+\frac{C_1}{4\epsilon} \right)\|\de\xi\|_2^2.
\end{align*}
By choosing $\epsilon:=\frac{1}{2C(L,T,\lambda_0)}$, we have
\begin{align}
     \int_0^T \|\de v_s\|_2^2 ds \le C \|\de\xi\|_2^2,\label{lem:kappa_4}
\end{align}
where $C=2\epsilon C(L,T,\lambda_0)+\frac{C_1}{2\epsilon} $ is a constant depending only on $(l,L,L_0,L_1,L_2,\lambda_0,\lambda_x,\lambda_m,\lambda_v,T)$. By combining \eqref{lem:kappa_sde_estimate}, \eqref{lem:kappa_bsde_estimate} and \eqref{lem:kappa_4}, we finally obtain the second estimate in \eqref{lem:kappa_0}. And the first estimate in \eqref{lem:kappa_0} can be proven similarly.
\end{proof}

As a direct consequence of the well-posedness of FBSDEs \eqref{FBSDE:MFTC} in Theorem \ref{main2_thm} and the sufficient maximum principle in Theorem~\ref{thm:suff}, we can now solve Problem $\left(\mathbf{P}^{0,\xi}\right)$.

\begin{theorem}\label{thm:solvability}
    Under Assumptions (A1)-(A4) and also the validity of Condition \eqref{thm:suff:condition}-(i,ii), $u\in \mr_\f^2(0,T)$ (the solution of FBSDEs with jumps satisfying the optimality condition in \eqref{FBSDE:MFTC}) is the unique optimal control for Problem $\left(\mathbf{P}^{0,\xi}\right)$.
\end{theorem}

Up to now, we can see that we only require the first-order continuous differentiability of the coefficients $(b,\sigma,f,g)$ to solve FBSDEs with jumps \eqref{FBSDE:MFTC} and the mean field type control problem $\left(\mathbf{P}^{0,\xi}\right)$. To further consider the classically solvability of the HJB integro-partial differential equation \eqref{HJB:intro}, we next study the Jacobian flow of FBSDEs with jumps \eqref{FBSDE:MFTC}, which may require the second-order continuous differentiability of the coefficients.

\subsection{Derivatives in initial $\xi%\in L^2(\Omega,\f,\mathbb{P};\brn)
$}\label{sec:Jaco}

Now we study the differentiability with respect to the initial condition $\xi\in L^2(\Omega,\f,\mathbb{P};\brn)$ of the solution $\Theta=(Y,P,Q,R,u)$ of FBSDEs with jumps \eqref{FBSDE:MFTC} under Assumptions (A1)-(A4). We still use the notation $\theta_t:=(Y_t,\lr(Y_t),u_t)$. The following assumptions are the regularity-enhanced version of Assumptions (A1) and (A2). 

{\bf (A1')} The coefficients $b$ and $\sigma$ satisfy (A1). Moreover, the following derivatives exist, and they are continuous in all their arguments:
\begin{align*}
    &b_{xx},\ b_{xv},\ b_{vx},\ b_{vv},\ D_y \frac{db_x}{d\nu},\ D_y \frac{db_v}{d\nu},\ D_{y'}D_y \frac{d^2b}{d\nu^2},\ D_y^2 \frac{db}{d\nu},\\
    &\sigma_{xx},\ \sigma_{xv},\ \sigma_{vx},\ \sigma_{vv},\ D_y \frac{d\sigma_x}{d\nu},\ D_y \frac{d\sigma_v}{d\nu},\ D_{y'}D_y \frac{d^2\sigma}{d\nu^2},\ D_y^2 \frac{d\sigma}{d\nu}.
\end{align*}

\textbf{(A2')} The functionals $f$ and $g$ satisfy (A2). Moreover, the following derivatives exist, and they are continuous in all their arguments and are globally bounded by $L$ in norm:
\begin{align*}
    &f_{xx},\ f_{xv},\ f_{vx},\ f_{vv},\ D_y \frac{df_x}{d\nu},\ D_y \frac{df_v}{d\nu},\ D_{y'}D_y \frac{d^2f}{d\nu^2},\ D_y^2 \frac{df}{d\nu},\ g_{xx},\ D_y \frac{dg_x}{d\nu},\ D_{y'}D_y \frac{d^2g}{d\nu^2},\ D_y^2 \frac{dg}{d\nu}.
\end{align*}

The only difference between Assumptions (A1') and (A1) (resp. Assumptions (A2') and (A2)) is that the former requires one more differentiability of Coefficients $(b,\sigma)$ (resp. Coefficients $(f,g)$) in $(x,m,v)$, which is natural since we are studying the G\^ateaux derivatives of the processes $\Theta$ this section. In contrast, Assumptions (A1) and (A2) in our previous work \cite{AB13,ABmfg1} studying the second order mean field games do not require the second order derivatives of the coefficients in $m$. This difference is because the fact that the mean field game problem is a fixed point problem with the state process depending on the equilibrium law; while in the MFTC problem, the state process depends on the law of the current state, which makes it necessary to differentiate the coefficients with respect to the distribution variable when deriving the Jacobian flow of FBSDEs with jumps \eqref{FBSDE:MFTC}. For a more detailed discussion on the difference the mean field games and MFTC problems, we refer to our previous work \cite{AB12}.

For notational convenience, for any $(x,\xi,v),\ (\de x,\de \xi,\de v)\in\brn\times L^2(\Omega,\f,\mathbb{P};\brn)\times\brd$, we denote by 
\begin{align}
    &Db(t,x,\lr(\xi),v)(\de x,\de \xi,\de v) \notag \\
    :=\ & b_x (t,x,\lr(\xi),v)\de x+\he\left[D_y\frac{db}{d\nu}(t,x,\lr(\xi),v)\left(\widehat{\xi}\right) \widehat{\de \xi} \right]+b_v (t,x,\lr(\xi),v)\de v; \label{def:b_1}
\end{align}
and also,
\begin{align}
    &Db_x(t,x,\lr(\xi),v)(\de x,\de \xi,\de v) \notag \\
    :=\ & b_{xx} (t,x,\lr(\xi),v)\de x+\he\left[D_y\frac{db_x}{d\nu}(t,x,\lr(\xi),v)\left(\widehat{\xi}\right) \widehat{\de \xi} \right]+b_{xv} (t,x,\lr(\xi),v)\de v; \notag \\
    &Db_v(t,x,\lr(\xi),v)(\de x,\de \xi,\de v) \notag \\
    :=\ & b_{vx} (t,x,\lr(\xi),v)\de x+\he\left[D_y\frac{db_v}{d\nu}(t,x,\lr(\xi),v)\left(\widehat{\xi}\right) \widehat{\de \xi} \right]+b_{vv} (t,x,\lr(\xi),v)\de v; \notag \\
    &D\left(D_y\frac{db}{d\nu}\right)(t,x,\lr(\xi),v)(\xi)(\de x,\de \xi,\de v) \notag \\
    :=\ & \left(D_y\frac{db_x}{d\nu}\right)(t,x,\lr(\xi),v)(\xi)\de x+\left(D_y\frac{db_v}{d\nu}\right)(t,x,\lr(\xi),v)(\xi)\de v \notag \\
    & +\widetilde{\e}\left[D_{y'}D_y\frac{d^2b}{d\nu^2}(t,x,\lr(\xi),v)\left(\xi,\widetilde{\xi}\right) \widetilde{\de \xi} \right]+D_y^2\frac{db}{d\nu}(t,x,\lr(\xi),v)(\xi)\de \xi; \label{def:b_2}
\end{align}
and also use the similar notations
\begin{align*}
    D\sigma^j,\ D\sigma^j_x,\ D\sigma^j_v,\ D\bigg(D_y\frac{d\sigma^j}{d\nu}\bigg),\ Df_x,\ Df_v,\ D\left(D_y\frac{df}{d\nu}\right),\ Dg_x,\ D\left(D_y\frac{dg}{d\nu}\right).
\end{align*}
Under Assumptions (A1') and (A3)-(i), the conditions in \eqref{generic:condition:b} also writes (recall in \eqref{def:B:v^0} that $B$ is defined on $\br^{d_0}$ rather than $\brd$), for $t\in[0,T]$, $x\in\brn$, $m\in\pr_2(\brn)$, $v^0\in\br^{d_0}$ and $y\in\brn$,
\begin{align}
    (i)\ &\left|B_{xx}\left(t,x,m,v^0\right)\right|,\  \left|D_y\frac{dB_x}{d\nu}\left(t,x,m,v^0\right)(y) \right|,\ \  \left|D_y^2\frac{dB}{d\nu}\left(t,x,m,v^0\right)(y) \right|, \notag \\
    &\sqrt{\int_\brn \left|D_{y'}D_y\frac{d^2B}{d\nu^2}\left(t,x,m,v^0\right)\left(y,y'\right)\right|^2m(dy')} \le  \frac{L_0}{1+|x|+|m|_1+\left|v^0\right|};\notag\\
    (ii)\ &\left|B_{x v^0}\left(t,x,m,v^0\right)\right|,\ \left|B_{v^0 x}\left(t,x,m,v^0\right)\right|,\  \left|D_y\frac{dB_{v^0}}{d\nu}\left(t,x,m,v^0\right)(y) \right| \notag \\
    &\le \frac{L_1}{1+|x|+|m|_1+\left|v^0\right|};\notag\\
    (iii)\ &\left|B_{v^0 v^0}\left(t,x,m,v^0\right)\right|\le \frac{L_2}{1+|x|+|m|_1+\left|v^0\right|}. \label{B_vv:boundedness}
\end{align}
Therefore, we know that
\begin{align}
    &|Db_x(t,x,\lr(\xi),v)(\de x,\de \xi,\de v))|\le \frac{L_0\left(|\de x|+\|\de \xi\|_2\right)+L_1\left|\de v^0\right|}{1+|x|+|\lr(\xi)|_1+\left|v^0\right|}; \notag \\ 
    &\left| D\left(D_y\frac{db}{d\nu}\right)(t,x,\lr(\xi),v)(\xi)(\de x,\de \xi,\de v)\right| \le \frac{L_0\left(|\de x|+\|\de \xi\|_2+|\de \xi|\right)+L_1\left|\de v^0\right|}{1+|x|+|\lr(\xi)|_1+\left|v^0\right|}; \notag \\
    &|Db_v(t,x,\lr(\xi),v)(\de x,\de \xi,\de v))|\le \frac{L_1\left(|\de x|+\|\de \xi\|_2\right)+L_2\left|\de v^0\right|}{1+|x|+|\lr(\xi)|_1+\left|v^0\right|}, \label{generic:condition:b'}
\end{align}
where $v=(v^0,v^1,\dots,v^n)$, $\de v=(\de v^0,\de v^1,\dots,\de v^n)$, $v^j,\de v^j\in \br^{d_j}$. Similarly, the conditions \eqref{generic:condition:A} also writes, for $1\le j\le n$ with $d_j>0$ and  $v^j\in\br^{d_j}$,
\begin{align}
    (i)\ &\left|A^j_{xx}\left(t,x,m,v^j\right)\right|,\  \left|D_y\frac{dA^j_x}{d\nu}\left(t,x,m,v^j\right)(y) \right|,\ \ \left|D_y^2\frac{dA^j}{d\nu}\left(t,x,m,v^j\right)(y) \right|, \notag \\
    &\sqrt{\int_\brn \left|D_{y'}D_y\frac{d^2A^j}{d\nu^2}\left(t,x,m,v^j\right)\left(y,y'\right)\right|^2m(dy')} \le  \frac{L_0}{1+|x|+|m|_1+\left|v^j\right|};\notag\\
    (ii)\ &\left|A^j_{x v^0}\left(t,x,m,v^j\right)\right|,\ \left|A^j_{v^0 x}\left(t,x,m,v^j\right)\right|,\  \left|D_y\frac{dA^j_{v^0}}{d\nu}\left(t,x,m,v^j\right)(y) \right| \notag\\
    &\le \frac{L_1}{1+|x|+|m|_1+\left|v^j\right|};\notag\\
    (iii)\ &\left|A^j_{v^0 v^j}\left(t,x,m,v^0\right)\right|\le \frac{L_2}{1+|x|+|m|_1+\left|v^j\right|}. \label{A_vv:boundedness}
\end{align}
and therefore,
\begin{align}
    &|D\sigma^j_x(t,x,\lr(\xi),v)(\de x,\de \xi,\de v))|\le  \frac{L_0\left(|\de x|+\|\de \xi\|_2\right)+L_1\left|\de v^j\right|}{1+|x|+|\lr(\xi)|_1+\left|v^j\right|}; \notag \\
    &\left| D\left(D_y\frac{d\sigma^j}{d\nu}\right)(t,x,\lr(\xi),v)(\xi)(\de x,\de \xi,\de v)\right| \le  \frac{L_0\left(|\de x|+\|\de \xi\|_2+|\de \xi|\right)+L_1\left|\de v^j\right|}{1+|x|+|\lr(\xi)|_1+\left|v^j\right|}; \notag \\
    &|D\sigma^j_v(t,x,\lr(\xi),v)(\de x,\de \xi,\de v))|\le \ \frac{L_1\left(|\de x|+\|\de \xi\|_2\right)+L_2\left|\de v^j\right|}{1+|x|+|\lr(\xi)|_1+\left|v^j\right|}. \label{generic:condition:A'}
\end{align}
Under Assumptions (A2'), the first convexity condition of $f$ in (A4) also writes, for any $t\in[0,T]$ and $(x,\xi,v),\ (\de x,\de \xi,\de v)\in\brn\times L^2(\Omega,\f,\mathbb{P};\brn)\times\brd$,
\begin{align}
        &\de f(t,x,\lr(\xi),v)(\de x,\de \xi,\de v) \notag \\
        :=\ & \de x^\top Df_x(t,x,\lr(\xi),v)(\de x,\de \xi,\de v)+  \de v^\top Df_v(t,x,\lr(\xi),v)(\de x,\de \xi,\de v) \notag \\
        &+ \e \left[\de \xi^\top D\left(D_y\frac{df}{d\nu}(t,x,\lr(\xi),v\right)(\xi)(\de x,\de \xi,\de v)\right] \label{def:convex_f}\\
        =\ &\left[\left(
		\begin{array}{cc}
			f_{vv} & f_{xv}\\
			f_{vx} & f_{xx}
		\end{array}
		\right)(t,x,\lr(\xi),v) \right] \left(
		\begin{array}{cc}
			\de v\\
			\de x
		\end{array}
		\right)^{\otimes 2} \notag \\
        &+2\e\bigg[\de \xi^\top D_y\frac{df_x}{d\nu}(t,x,\lr(\xi),v)(\xi)\bigg] \de x+2\e\bigg[\de \xi^\top D_y\frac{df_v}{d\nu}(t,x,\lr(\xi),v)(\xi)\bigg] \de v \notag \\
        &+\e\widetilde{\e}\left[\de \xi^\top \left(D_{y'}D_y\frac{d^2f}{d\nu^2}(t,x,\lr(\xi),v)\left(\xi,\widetilde{\xi}\right)\right)\widetilde{\de \xi}\right]+\e\left[\de \xi^\top \left(D_y^2\frac{df}{d\nu}(t,x,\lr(\xi),v)\left(\xi\right)\right)\de \xi \right] \notag \\
        \geq\ &  2\lambda_v |\de v|^2+2\lambda_x |\de x|^2+2\lambda_m\|\de \xi\|_2^2, \label{convex_f}
\end{align}
where the second equality uses the \textit{Schur complement} and the last inequality uses the convexity of $f$; similarly, the convexity of $g$ in (A4) also gives
\begin{align}
    &\de g(x,\lr(\xi))(\de x,\de \xi) \notag \\
        :=\ & \de x^\top Dg_x(x,\lr(\xi))(\de x,\de \xi)+ \e \left[\de \xi^\top D\left(D_y\frac{dg}{d\nu}(x,\lr(\xi)\right)(\xi)(\de x,\de \xi)\right] \label{def:g_1} \\
        =\ & \de x^\top \left[g_{xx}(x,\lr(\xi))\right] \de x +2\e\bigg[\de \xi^\top D_y\frac{dg_x}{d\nu}(x,\lr(\xi))(\xi)\bigg] \de x\notag \\
        &+\e\widetilde{\e}\left[\de \xi^\top \left(D_{y'}D_y\frac{d^2g}{d\nu^2}(x,\lr(\xi))\left(\xi,\widetilde{\xi}\right)\right)\widetilde{\de \xi}\right]+\e\left[\de \xi^\top\left(D_y^2\frac{dg}{d\nu}(x,\lr(\xi))\left(\xi\right)\right)\de \xi\right] \notag \\
        \geq\ &  0. \label{convex_g}
\end{align}

With these notations, the G\^ateaux derivatives of the processes $\Theta=(Y,P,Q,R,u)$ in the initial condition $\xi\in\lr^2_{\f_0}$ along the direction $\eta\in\lr^2_{\f_0}$ can be characterized by the following system of FBSDEs with jumps:
\begin{align}
        &\dr_\eta Y_t=\eta+\int_0^t Db\left(s,\theta_s\right)\left(\dr_\eta Y_s,\dr_\eta Y_s,\dr_\eta u_s\right) ds +\int_0^t D\sigma\left(s,\theta_s\right)\left(\dr_\eta Y_s,\dr_\eta Y_s,\dr_\eta u_s\right) dB_s  \notag \\
        &\quad\qquad +\int_0^t \int_E \bigg\{\gamma_1\left(s,e\right)\dr_\eta Y_{s-} + \gamma_2(s,e)\e\left[\dr_\eta Y_{s-} \right] \bigg\}\mathring{N}(de,ds),  \notag \\
        &\dr_\eta P_t= Dg_{x}(Y_T,\lr(Y_T)) \left( \dr_\eta Y_T,\dr_\eta Y_T\right)+\he\bigg[D\left(D_y\frac{dg}{d\nu}\right)\left(\hyT,\lr(Y_T)\right)(Y_T) \left(\widehat{\dr_\eta Y_T},\dr_\eta Y_T\right)\bigg]  \notag  \\
        &\quad\qquad +\int_t^T \Bigg\{ b_x\left(s,\theta_s\right)^\top \dr_\eta P_s+\sum_{j=1}^n \sigma^j_x \left(s,\theta_s\right)^\top \dr_\eta Q^j_s +\int_E \gamma_1\left(s,e\right)^\top \dr_\eta R_s(e)\lambda(de)  \notag  \\
        &\qquad\qquad\qquad +\he\bigg[\left(D_y\frac{db}{d\nu}\left(s,\widehat{\theta_s}\right)\left(Y_s\right)\right)^\top \widehat{\dr_\eta P_s}+\sum_{j=1}^n \left(D_y\frac{d\sigma^j}{d\nu}\left(s,\widehat{\theta_s}\right)\left(Y_s\right)\right)^\top \widehat{\dr_\eta Q^j_s} \bigg]  \notag  \\
        &\qquad\qquad\qquad +\int_E \gamma_2\left(s,e\right)^\top \e\left[\dr_\eta R_s(e)\right] \lambda(de) +\left[Db_{x}\left(s,\theta_s\right)\left(\dr_\eta Y_s,\dr_\eta Y_s,\dr_\eta u_s\right)\right]^\top P_s  \notag  \\
        &\qquad\qquad\qquad  + \sum_{j=1}^n \bigg[D\sigma^j_{x} \left(s,\theta_s\right) \left(\dr_\eta Y_s,\dr_\eta Y_s,\dr_\eta u_s\right)\bigg]^\top Q^j_s  \notag  \\
        &\qquad\qquad\qquad +\he\bigg\{\bigg[D\left(D_y\frac{db_x}{d\nu}\right)\left(s,\widehat{\theta_s}\right)(Y_s)\left(\widehat{\dr_\eta Y_s},\dr_\eta Y_s,\widehat{\dr_\eta u_s}\right)\bigg]^\top \hps  \notag \\
        &\qquad\qquad\qquad\qquad +\sum_{j=1}^n \bigg[D\bigg(D_y\frac{d\sigma^j_x}{d\nu}\bigg)\left(s,\widehat{\theta_s}\right)\left(Y_s\right)\left(\widehat{\dr_\eta Y_s},\dr_\eta Y_s,\widehat{\dr_\eta u_s}\right)\bigg]^\top \hqsj \bigg\}  \notag  \\
        &\qquad\qquad\qquad +Df_{x}\left(s,\theta_s\right)\left(\dr_\eta Y_s,\dr_\eta Y_s,\dr_\eta u_s\right)  \notag  \\
        &\qquad\qquad\qquad + \he\bigg[D\left(D_y\frac{df}{d\nu}\right)\left(s,\widehat{\theta_s}\right)\left(Y_s\right)\left(\widehat{\dr_\eta Y_s},\dr_\eta Y_s,\widehat{\dr_\eta u_s}\right)\bigg]\Bigg\}ds  \notag  \\
        &\quad\qquad -\int_t^T \dr_\eta Q_s dB_s-\int_t^T\int_E \dr_\eta R_s(e) \mathring{N}(de,ds),  \notag  \\
        &b_v\left(t,\theta_{t-}\right)^\top \dr_\eta P_{t-} +\sum_{j=1}^n \sigma^j_v\left(t,\theta_{t-}\right)^\top \dr_\eta Q_t^j +\left[Db_v\left(t,\theta_{t-}\right)\left(\dr_\eta Y_{t-},\dr_\eta Y_{t-},\dr_\eta u_{t}\right)\right]^\top P_{t-}  \notag \\
        &\quad\qquad +\sum_{j=1}^n \left[D\sigma^j_v\left(t,\theta_{t-}\right)\left(\dr_\eta Y_{t-},\dr_\eta Y_{t-},\dr_\eta u_{t}\right)\right]^\top Q_t^j \notag \\
        &\quad\qquad +Df_v(t,\theta_{t-})\left(\dr_\eta Y_{t-},\dr_\eta Y_{t-},\dr_\eta u_{t}\right)=0, \quad t\in[0,T]. \label{FBSDE:Jaco}
\end{align}

We have the following result, whose proof is given in Appendix~\ref{pf:thm:Gateaux}.

\begin{theorem}\label{thm:Gateaux}
    Under Assumptions (A1'), (A2'), (A3) and (A4) and the validity of Condition \eqref{thm:suff:condition}-(i,ii), for any $\xi,\eta\in \lr_{\f_0}^2$, FBSDEs with jumps \eqref{FBSDE:Jaco} has a unique solution $\dr_\eta \Theta:=(\dr_\eta Y,\dr_\eta P,\dr_\eta Q,\dr_\eta R,\dr_\eta u)\in\mathbb{S}$, and it satisfies
    \begin{equation}\label{thm:Gateaux_boundedness}
        \|\dr_\eta \Theta\|_{\mathbb{S}}\le C \|\eta\|_2,
    \end{equation}
    where $C>0$ is a constant depending only on $(l,L,L_0,L_1,L_2,\lambda_0,\lambda_x,\lambda_m,\lambda_v,T)$. For $\epsilon\in(0,1)$, let $\Theta^\epsilon$ be the solution of FBSDEs with jumps \eqref{FBSDE:MFTC} corresponding to the initial condition $\xi^\epsilon:=\xi+\epsilon\eta$, then, 
    \begin{align}\label{thm:Gateaux_appro}
        \lim_{\epsilon\to 0} \left\|\frac{\Theta^\epsilon-\Theta}{\epsilon} -\dr_\eta \Theta\right\|_{\mathbb{S}}=0.
    \end{align}
    That is, $\dr_\eta\Theta$ is the G\^ateaux derivative of $\Theta$ with respect to $\xi$ along the direction $\eta$, so we write $\dr_\eta\Theta$ as the official $D_\eta\Theta$ in the rest of our article. Moreover, the G\^ateaux derivative is linear in $\eta$ and continuous in $\xi$. 
\end{theorem}

\subsection{Derivatives in initial $y\in\brd$}\label{subsec:D_y}

In this section, we still assume that Assumptions (A1'), (A2'), (A3) and (A4) and Condition \eqref{thm:suff:condition} are satisfied. In Theorem~\ref{thm:Gateaux}, we have shown that the processes $D_\eta\Theta$ is the G\^ateaux derivative of $\Theta$ in the initial condition $\xi$. This G\^ateaux derivative can be seen as the derivatives on the ``Lifting" of the distribution variable, and $D_\eta\Theta$ in some sense represent the variation of dependence of $\Theta$ in the distribution $\lr(\xi)=\mu$. However, we should note the fact that the processes $\Theta$  depends on $\xi$ not only on its law, but also on its exact state. With this in mind, it is natural to study the following system of FBSDEs with jump-diffusion, which will be used in the study of the linear functional-derivative of the value function $V$ for the HJB equation in Section~\ref{sec:HJB} (so the reader may skip this subsection at the first read). %We would like to emphasize that it is that our $\sigma$ can depend on the control argument $v$, so in proving the solvability of the HJB integro-partial differential equation \eqref{HJB:intro}, we need to give the boundedness estimate for the derivative $D_y^2 \frac{dV}{d\nu}(t,\mu)(y)$ (see the proof of Theorem~\ref{thm:main:HJB}), and therefore we need the results in this subsection to give the boundedness for $D_y^2 \frac{dV}{d\nu}(t,\mu)(y)$ (see Proposition~\ref{thm:value_regu_mu'}). 

For $\mu\in\pr_2(\brn)$, we choose any $\xi\in\lr_{\f_0}^2$ with $\lr(\xi)=\mu$, and denote by $\Theta^\xi:=\left(Y^\xi,P^\xi,Q^\xi,R^\xi,u^\xi\right)$ the corresponding solution to FBSDEs \eqref{FBSDE:MFTC} with the initial condition $\xi$. Then, we consider the following system of FBSDEs with jumps for $\Theta^{y,\mu}:=\left(Y^{y,\mu},P^{y,\mu},Q^{y,\mu},R^{y,\mu},u^{y,\mu}\right)$ with initial condition $y\in\brn$, which also depends on the measure of the processes $\Theta^\xi$: 
\begin{equation}\label{FBSDE:MFTC_y}
    \left\{
        \begin{aligned}
            &Y^{y,\mu}_t=y+\int_0^t b\left(s,\theta^{y,\mu}_s\right) ds +\int_0^t \sigma \left(s,\theta^{y,\mu}_s\right)dB_s\\
            &\quad\qquad +\int_0^t \int_E \gamma\left(s,Y^{y,\mu}_{s-},\lr\left(Y^{\xi}_{s-}\right),e\right) \mathring{N}(de,ds),\\
            &P^{y,\mu}_t= g_x\left(Y^{y,\mu}_T,\lr\left(Y^{\xi}_T\right)\right)+\he\left[D_y\frac{dg}{d\nu}\left(\widehat{Y^{\xi}_T},\lr\left(Y^{\xi}_T\right)\right)\left(Y^{y,\mu}_T\right)\right]\\
            &\quad\qquad +\int_t^T \bigg\{ b_x\left(s,\theta^{y,\mu}_s\right)^\top P^{y,\mu}_s+\sum_{j=1}^n \sigma^j_x \left(s,\theta^{y,\mu}_s\right)^\top Q^{y,\mu,j}_s\\
            &\qquad\qquad\qquad +\int_E \gamma_x\left(s,Y^{y,\mu}_s,\lr\left(Y^{\xi}_s\right),e\right)^\top R^{y,\mu}_s(e)\lambda(de)+f_x\left(s,\theta^{y,\mu}_s\right)\\
            &\qquad\qquad\qquad +\he\bigg[\left(D_y\frac{db}{d\nu}\left(s,\widehat{\theta^\xi_s}\right)\left(Y^{y,\mu}_s\right)\right)^\top \widehat{P^{\xi}_s} +\sum_{j=1}^n \left(D_y\frac{d\sigma^j}{d\nu}\left(s,\widehat{\theta^\xi_s}\right)\left(Y^{y,\mu}_s\right)\right)^\top \widehat{Q^{\xi,j}} \\
            &\qquad\qquad\qquad\qquad +\int_E \left(D_y\frac{d\gamma}{d\nu}\left(s,\widehat{Y^{\xi}_s},\lr\left(Y^{\xi}_s\right),e\right)\left(Y^{y,\mu}_s\right)\right)^\top \widehat{R^{\xi}_s}(e) \lambda(de)\\
            &\qquad\qquad\qquad\qquad + D_y\frac{df}{d\nu}\left(s,\widehat{\theta^\xi_s}\right)\left(Y^{y,\mu}_s\right) \bigg]\bigg\}ds\\
            &\quad\qquad -\int_t^T Q^{y,\mu}_s dB_s-\int_t^T\int_E R^{y,\mu}_s(e) \mathring{N}(de,ds), \\
            &b_v\left(t,\theta^{y,\mu}_{t-}\right)^\top P^{y,\mu}_{t-} +\sum_{j=1}^n \sigma^j_v\left(t,\theta^{y,\mu}_{t-}\right)^\top Q_t^{y,\mu,j} +f_v\left(t,\theta^{y,\mu}_{t-}\right)=0, \quad t\in[0,T].
        \end{aligned}
    \right.
\end{equation}
where $\theta^{y,\mu}_t:=\left(Y^{y,\mu}_t,\lr\left(Y^{\xi}_t\right),u^{y,\mu}_t\right)$ and $\theta^\xi_t:=\left(Y^\xi_t,\lr\left(Y^{\xi}_t\right),u^\xi_t\right)$ for $t\in[0,T]$. Here, note that the system of FBSDEs \eqref{FBSDE:MFTC_y} depends on $\xi$ only through its law $\lr(\xi)=\mu$, therefore, it is reasonable to use the superscript $\mu$ in $\Theta^{y,\mu}$. The existence, uniqueness of the processes $\Theta^{y,\mu}$ are similar to that of $\Theta^{\xi}$ in Theorem~\ref{thm:solvability}, which is omitted here; similar to the proof of \eqref{lem:kappa_0}, we can also have the following boundedness and continuity, and we omit the proof:
\begin{align}
    \left\|\Theta^{y,\mu}\right\|_{\mathbb{S}}\le C \left( 1+|y|+|\mu|_2\right), \quad \left\|\Theta^{y',\mu'}-\Theta^{y,\mu} \right\|_{\mathbb{S}} \le C\left(|y'-y|+W_2(\mu,\mu')\right),\label{add-4}
\end{align}  
where $C$ depends only on $(l,L,L_0,L_1,L_2,\lambda_0,\lambda_x,\lambda_m,\lambda_v,T)$. Under Assumption (A3), note that in FBSDEs with jumps \eqref{FBSDE:MFTC_y} the following relation holds:
\begin{equation*}
\begin{aligned}
    &B_{v^0} \left(t,Y^{y,\mu}_{t-},\lr\left(Y^{\xi}_{t-}\right),u^{y,\mu,0}_t\right)^\top P^{y,\mu}_{t-}+f^0_{v^0} \left(t,Y^{y,\mu}_{t-},\lr\left(Y^{\xi}_{t-}\right),u^{y,\mu,0}_t\right)=0,\\
    &A_{v^j} \left(s,Y^{y,\mu}_{t-},\lr\left(Y^{\xi}_{t-}\right),u^{y,\mu,j}_t\right)^\top Q^{y,\mu,j}_t+f^0_{v^j} \left(t,Y^{y,\mu}_{t-},\lr\left(Y^{\xi}_{t-}\right),u^{y,\mu,j}_t\right)=0,
\end{aligned}
\end{equation*}
where $u^{y,\mu}_t=\left(u^{y,\mu,0}_t,u^{y,\mu,1}_t,\dots,u^{y,\mu,n}_t \right)$. Similar as in Proposition~\ref{prop:cone}, we also have the following cone property for $P^{y,\mu}$ and $Q^{y,\mu,j}$ for $1\le j\le n$ with $d_j>0$:
\begin{equation}\label{cone:P^y_t'}
    \begin{aligned}
        \left|P_{t-}^{y,\mu}\right| \le\ & \frac{L^2}{\lambda_0} \left[1+\left|Y^{y,\mu}_{t-}\right|+\left|\lr\left(Y^{\xi}_{t-}\right)\right|_1+\left|u^{y,\mu,0}_t\right| \right],\\
        \left|Q_t^{y,\mu,j}\right| \le\ & \frac{L^2}{\lambda_0} \left[1+\left|Y^{y,\mu}_{t-}\right|+ \left|\lr\left(Y^{\xi}_{t-}\right)\right|_1 +\left|u^{y,\mu,j}_t\right| \right];
    \end{aligned}
\end{equation}
From the uniqueness result for FBSDEs with jumps \eqref{FBSDE:MFTC}, we know that 
\begin{align*}
    \Theta^{y,\mu}_t\big|_{y=\xi}=\Theta^{\xi}_t.
\end{align*}

We also give the derivatives of the processes $\Theta^{y,\mu}$ with respect to $y\in\brn$, which will be shown in the next section to be the linear functional-derivative of $V$. Consider the following system of FBSDEs with jumps for processes $D_y \Theta^{y,\mu}:=\left(D_y Y^{y,\mu},D_y P^{y,\mu},D_y Q^{y,\mu},D_y R^{y,\mu},D_y u^{y,\mu}\right)$:
\begin{align}
    &D_y Y^{y,\mu}_t=I+\int_0^t \left[b_x\left(s,\theta^{y,\mu}_s\right)D_yY^{y,\mu}_s+b_v\left(s,\theta^{y,\mu}_s\right)D_y u^{y,\mu}_s \right] ds \notag \\
    &\quad\qquad\qquad +\int_0^t \left[\sigma_x\left(s,\theta^{y,\mu}_s\right)D_yY^{y,\mu}_s+\sigma_v\left(s,\theta^{y,\mu}_s\right)D_y u^{y,\mu}_s \right]dB_s +\int_0^t \int_E \gamma_1\left(s,e\right)D_y Y^{y,\mu}_{r-} \mathring{N}(de,ds),\notag \\
    &D_y P^{y,\mu}_t= g_{xx}\left(Y_T^{y,\mu},\lr\left(Y_T^{\xi}\right)\right)D_yY^{y,\mu}_T +\he\left[D_y^2\frac{dg}{d\nu}\left(\widehat{Y_T^{\xi}},\lr\left(Y^{\xi}_T\right)\right)\left(Y^{y,\mu}_T\right)D_yY^{y,\mu}_T\right] \notag \\
    &\quad\qquad\qquad +\int_t^T \Bigg\{ b_{x}\left(s,\theta^{y,\mu}_s\right)^\top D_y P^{y,\mu}_s +\sum_{j=1}^n \sigma^j_x \left(s,\theta^{y,\mu}_s\right)^\top D_y Q^{y,\mu,j}_s +\int_E \gamma_1\left(s,e\right)^\top D_y R^{y,\mu}_s(e)\lambda(de) \notag \\
    &\qquad\qquad\qquad\qquad +\left[b_{xx}\left(s,\theta^{y,\mu}_s\right) D_y Y^{y,\mu}_s+b_{xv}\left(s,\theta^{y,\mu}_s\right) D_y u^{y,\mu}_s\right]^\top P^{y,\mu}_s \notag \\
    &\qquad\qquad\qquad\qquad +\sum_{j=1}^n \left[\sigma^j_{xx}\left(s,\theta^{y,\mu}_s\right) D_y Y^{y,\mu}_s+\sigma^j_{xv}\left(s,\theta^{y,\mu}_s\right) D_y u^{y,\mu}_s\right]^\top Q^{y,\mu,j}_s \notag \\      
    &\qquad\qquad\qquad\qquad +f_{xx}\left(s,\theta^{y,\mu}_s\right) D_y Y^{y,\mu}_s +f_{xv}\left(s,\theta^{y,\mu}_s\right) D_y u^{y,\mu}_s \notag  \\
    &\qquad\qquad\qquad\qquad +\he\bigg\{\bigg[D_y^2\frac{db}{d\nu}\left(s,\widehat{\theta^{\xi}_s}\right)\left(Y_s^{y,\mu}\right) D_y Y^{y,\mu}_s\bigg]^\top \widehat{P^{\xi}_s} \notag \\
    &\quad\qquad\qquad\qquad\qquad\qquad +\sum_{j=1}^n \bigg[D_y^2\frac{d\sigma^j}{d\nu}\left(s,\widehat{\theta^{\xi}_s}\right)\left(Y_s^{y,\mu}\right) D_y Y^{y,\mu}_s\bigg]^\top \widehat{Q^{\xi,j}_s}  \notag \\
    &\quad\qquad\qquad\qquad\qquad\qquad +D_y^2\frac{df}{d\nu}\left(s,\widehat{\theta^{\xi}_s}\right)\left(Y_s^{y,\mu}\right)D_y Y^{y,\mu}_t\bigg\}\Bigg\}ds \notag \\
    &\quad\qquad\qquad -\int_t^T D_y Q^{y,\mu}_s dB_s-\int_t^T\int_E D_y R^{y,\mu}_s(e) \mathring{N}(de,ds), \quad s\in[t,T], \label{FBSDE:D_yTheta}
\end{align}
with the following condition
\begin{align}
    0=\ & b_v\left(t,\theta^{y,\mu}_{t-}\right)^\top D_y P^{y,\mu}_{t-} +\sum_{j=1}^n \sigma^j_v\left(t,\theta^{y,\mu}_{t-}\right)^\top D_y Q_t^{y,\mu,j} \notag \\
    &+ \left[b_{vx}\left(t,\theta^{y,\mu}_{t-}\right)D_y Y^{y,\mu}_{t-}+b_{vv}\left(t,\theta^{y,\mu}_{t-}\right)D_y u^{y,\mu}_t\right]^\top P^{y,\mu}_{t-} \notag \\
    &+\sum_{j=1}^n \left[\sigma^j_{vx}\left(t,\theta^{y,\mu}_{t-}\right)D_y Y^{y,\mu}_{t-}+\sigma^j_{vv}\left(t,\theta^{y,\mu}_{t-}\right)D_y u^{y,\mu}_t\right]^\top Q^{y,\mu,j}_t \notag \\
    &+f_{vx}\left(t,\theta^{y,\mu}_{t-}\right) D_y Y^{y,\mu}_{t-}+f_{vv}\left(t,\theta^{y,\mu}_{t-}\right) D_y u^{y,\mu}_t, \label{thm:D^2dV_1}
\end{align}
where $\theta^{y,\mu}_t=\left(Y^{y,\mu}_t,\lr\left(Y^{\xi}_t\right),u^{y,\mu}_t\right)$ and $\widehat{\theta^{y,\mu}_t}=\left(\widehat{Y^{y,\mu}_t},\lr\left(Y^{\xi}_t\right),\widehat{u^{y,\mu}_t}\right)$. The next result shows that the system \eqref{FBSDE:D_yTheta}-\eqref{thm:D^2dV_1} has a unique solution $D_y \Theta^{y,\mu}$, and the solution gives the derivative of $\Theta^{y,\mu}$ with respect to $y\in\brn$.

\begin{theorem}\label{thm:D^2dV}
    Under Assumptions (A1'), (A2'), (A3) and (A4), and the validity of Condition \eqref{thm:suff:condition}-(i,ii), for any $(y,\mu)\in\brn\times\pr_2(\brn)$, FBSDEs with jumps \eqref{FBSDE:D_yTheta}-\eqref{thm:D^2dV_1} has a unique solution $D_y \Theta^{y,\mu}\in\mathbb{S}$, and $D_y\Theta^{y,\mu}$ is the G\^ateaux derivative of $\Theta^{y,\mu}$ with respect to $y\in\brn$. Moreover, the G\^ateaux derivative satisfies
    \begin{equation}\label{add-5}
        \left\|D_y \Theta^{y,\mu}\right\|_{\mathbb{S}}\le C,
    \end{equation}
    where $C>0$ is a constant depending only on $(l,L,L_0,L_1,L_2,\lambda_0,\lambda_x,\lambda_m,\lambda_v,T)$, and it is continuous in $y$ and $\mu$.
\end{theorem}

\begin{proof}
The proof of the well-posedness of the above system of FBSDEs with jumps is similar to that for the well-posedness of FBSDEs with jumps \eqref{FBSDE:MFTC}, which is omitted; and similar to Theorem~\ref{thm:Gateaux}, we can also prove the continuity of $D_y \Theta^{y,\mu}$ and the following estimate
\begin{align*}
    \lim_{\epsilon\to 0}\left\|\frac{\Theta^{y+\epsilon\Tilde{y},\mu}-\Theta^{y,\mu}}{\epsilon}-D_y \Theta^{y,\mu} \right\|_{\mathbb{S}}=0,
\end{align*}
which shows that the components of $D_y \Theta^{y,\mu}$ are the G\^ateaux derivatives of the processes $\Theta^{y,\mu}$ in $y\in\brn$. 
\end{proof}

Here, we emphasize that the processes $D_y\Theta^{y,\mu}$ in \eqref{FBSDE:D_yTheta} are different from the processes $D_\eta \Theta$ in \eqref{FBSDE:Jaco}; indeed, $D_\eta \Theta$ is the total derivative of $\Theta$ in the initial, while $D_y\Theta^{y,\mu}$ is the partial derivative with respect to the state-only of the initial, not including the derivative in law with respect to the initial. In the rest of this article, we use the subscript in $D$ to distinguish them.

\section{Regularity of the value function}\label{sec:value}

In this section, we consider the MFTC problem $\left(\mathbf{P}^{t,\xi}\right)$ with any initial time $t\in[0,T]$ and initial condition $\xi\in\lr_{\f_t}^2$ independent of the Brownian motion $B$ and the Poisson jump process $N$. The state process corresponding to a control $v\in\mr^2_\f(t,T)$ is denoted by $X^{t,\xi,v}\in\sr_\f^2(t,T)$ and satisfies the following SDE:
\begin{equation*}
\begin{aligned}
    X_s^{t,\xi,v}=\xi&+\int_t^s b\left(r,X_r^{t,\xi,v},\lr\left(X_r^{t,\xi,v}\right),v_r\right) dr +\int_t^s \sigma \left(r,X_r^{t,\xi,v},\lr\left(X_r^{t,\xi,v}\right),v_r\right)dB_r\\
    &+\int_t^s \int_E \gamma\left(r,X_{r-}^{t,\xi,v},\lr\left(X_{r-}^{t,\xi,v}\right),v_r,e\right) \mathring{N}(de,dr),\quad s\in[t,T],
\end{aligned}
\end{equation*}
and the corresponding cost, which is denoted by $J_{t,\xi}(v)$, satisfies 
\begin{equation*}
    \begin{aligned}
        J_{t,\xi}(v):=\e\left[\int_t^T f\left(s,X_s^{t,\xi,v},\lr\left(X_s^{t,\xi,v}\right),v_s\right) ds + g\left(X_T^{t,\xi,v},\lr\left(X_T^{t,\xi,v}\right)\right)\right].
    \end{aligned}
\end{equation*}
The value function is then defined as 
\begin{align}\label{def:value}
    V(t,\lr(\xi)):= \inf_{v\in\mr_\f^2(t,T)} J_{t,\xi}(v).
\end{align}
From the results in Sections~\ref{sec:maxi} and \ref{sec:FBSDE}, we know that under Assumptions (A1)-(A4) and also the validity of Condition \eqref{thm:suff:condition}-(i,ii),
\begin{align}\label{value_equation}
    V(t,\lr(\xi))=J_{t,\xi}\left(u^{t,\xi}\right)=\e\left[\int_t^T f\left(s,Y_s^{t,\xi},\lr\left(Y_s^{t,\xi}\right),u^{t,\xi}_s\right) ds + g\left(Y_T^{t,\xi},\lr\left(Y_T^{t,\xi}\right)\right)\right].
\end{align}
Here,  $\Theta^{t,\xi}:=\left(Y^{t,\xi},P^{t,\xi},Q^{t,\xi},R^{t,\xi},u^{t,\xi}\right)\in \left(\sr_\f^2\times\sr_\f^2\times\mr_\f^2\times\kr_{\f,\lambda}^2\times\mr_\f^2\right)(t,T)$ is the unique solution of the following system of FBSDEs with jumps:
\begin{equation}\label{FBSDE:t,xi}
    \left\{
        \begin{aligned}
            &Y^{t,\xi}_s= \xi+\int_t^s H_p\left(r,Y^{t,\xi}_{r-},\lr\left(Y^{t,\xi}_{r-}\right),P^{t,\xi}_{r-}\right) dr +\sum_{j=1}^n \int_t^s H_{q^j} \left(r,Y^{t,\xi}_{r-},\lr\left(Y^{t,\xi}_{r-}\right),Q^{t,\xi,j}_{r-}\right)dB^j_r\\
            &\quad\qquad +\int_t^s \int_E \gamma\left(r,Y^{t,\xi}_{r-},\lr\left(Y^{t,\xi}_{r-}\right),e\right) \mathring{N}(de,dr),\\
            &P^{t,\xi}_s= g_x\left(Y_T^{t,\xi},\lr\left(Y_T^{t,\xi}\right)\right)+\he\left[D_y\frac{dg}{d\nu}\left(\widehat{Y_T^{t,\xi}},\lr\left(Y^{t,\xi}_T\right)\right)\left(Y^{t,\xi}_T\right)\right]\\
            &\quad\qquad +\int_s^T \bigg\{ H_x\left(r,Y^{t,\xi}_{r-},\lr\left(Y^{t,\xi}_{r-}\right),P^{t,\xi}_{r-},Q^{t,\xi}_r,R^{t,\xi}_r\right)\\
            &\qquad\qquad\qquad +\he\bigg[ D_y\frac{dH}{d\nu}\left(r,\widehat{Y^{t,\xi}_{r-}},\lr\left(Y^{t,\xi}_{r-}\right),\widehat{P^{t,\xi}_{r-}},\widehat{Q^{t,\xi}_r},\widehat{R^{t,\xi}_r}\right)\left(Y^{t,\xi}_{r-}\right) \bigg]\bigg\}dr\\
            &\quad\qquad -\int_s^T Q^{t,\xi}_r dB_r-\int_s^T\int_E R^{t,\xi}_r(e) \mathring{N}(de,dr); \\
            &b_v\left(s,\theta^{t,\xi}_{s-}\right)^\top P^{t,\xi}_{s-} +\sum_{j=1}^n \sigma^j_v\left(s,\theta^{t,\xi}_{s-}\right)^\top Q_s^{t,\xi,j} +f_v\left(s,\theta^{t,\xi}_{s-}\right)=0, \quad s\in[t,T],
        \end{aligned}
    \right.
\end{equation}
with the process $u^{t,\xi}_s=\left(u^{t,\xi,0}_s,u^{t,\xi,1}_s,\dots,u^{t,\xi,n}_s\right)$ being defined by
\begin{align*}
    &B_{v^0} \left(s,Y^{t,\xi}_{s-},\lr\left(Y^{t,\xi}_{s-}\right),u^{t,\xi,0}_s\right)^\top P^{t,\xi}_{s-}+f^0_{v^0} \left(s,Y^{t,\xi}_{s-},\lr\left(Y^{t,\xi}_{s-}\right),u^{t,\xi,0}_s\right)=0,\\
    &A_{v^j} \left(s,Y^{t,\xi}_{s-},\lr\left(Y^{t,\xi}_{s-}\right),u^{t,\xi,j}_s\right)^\top Q^{t,\xi,j}_s+f^0_{v^j} \left(s,Y^{t,\xi}_{s-},\lr\left(Y^{t,\xi}_{s-}\right),u^{t,\xi,j}_s\right)=0,\quad 1\le j\le n,\quad d_j>0.
\end{align*}
The well-posedness of System \eqref{FBSDE:t,xi} and $u^{t,\xi}$ can be obtained by the same way as the results for initial time $0$ in Subsection~\ref{subsec:well-posedness}, under the same Assumptions (A1)-(A4) and also the validity of Condition \eqref{thm:suff:condition}-(i,ii). From \eqref{value_equation} we know that $V$ depends on $\xi$ only through its law $\lr(\xi)$, therefore, the value function $V:[0,T]\times\pr_2(\brn)\to\br$ in \eqref{def:value} is well-defined. From results in Subsection~\ref{sec:Jaco}, under additional assumptions (A1') and (A2'), we denote by $D_\eta \Theta^{t,\xi}:=\left(D_\eta Y^{t,\xi},D_\eta P^{t,\xi},D_\eta Q^{t,\xi},D_\eta R^{t,\xi},D_\eta u^{t,\xi}\right)\in \mathbb{S}$ the G\^ateaux derivative of $\Theta^{t,\xi}$ with respect to $\xi$ along the direction $\eta$, which is the unique solution of FBSDEs with jumps \eqref{FBSDE:Jaco} corresponding to the initial time $t$ and initial condition $\xi$, and the initial direction $\eta$. Then, we can give the regularity of the map $\xi\mapsto V(t,\lr(\xi))$.

\begin{theorem}\label{thm:value_regu_mu}
    Under Assumptions (A1'), (A2'), (A3) and (A4), and the validity of Condition \eqref{thm:suff:condition}, the map $\xi\mapsto V (t,\lr(\xi))$ is twice G\^ateaux differentiable with the derivatives
	\begin{align}
		D_\xi  V(t,\lr(\xi))=P^{t,\xi}_t,\quad  D_\xi ^2 V(t,\lr(\xi))(\eta)= D_{\eta}P^{t,\xi}_t,\label{DV_P}
	\end{align}
	and they satisfy the growth conditions	
	\begin{align}
		|V(t,\lr(\xi))|\le& C\left(1+\|\xi\|_2^2\right),\label{V:growth}\\
		\left\|D_\xi  V(t,\lr(\xi))\right\|_2\le& C(1+\|\xi\|_2),\label{DV:growth}\\
		\left\|D_\xi ^2 V(t,\lr(\xi))(\eta)\right\|_2\le& C\|\eta\|_2,\label{DDV:growth}
	\end{align}
    where $C>0$ is a constant depending only on $(l,L,L_0,L_1,L_2,\lambda_0,\lambda_x,\lambda_m,\lambda_v,T)$; and $D_\xi^2 V$ is linear in $\eta$ and continuous in $\xi$. Moreover, the map $\pr_2(\brn)\ni\mu\mapsto V(t,\mu)$ is twice linearly functional-differentiable, and the linear functional-derivative satisfies the relations
    \begin{equation}\label{relation:DV-dV}
    \begin{aligned}
        D_\xi  V(t,\lr(\xi))&=D_y \frac{dV}{d\nu}(t,\lr(\xi))(\xi),\\
	    D_\xi ^2 V(t,\lr(\xi))\left(\eta\right)&=\left(D_y^2 \frac{d V}{d\nu}(t,\lr(\xi))(\xi)\right)^\top \eta +\he\left[\left(D_{y'}D_y \frac{d^2 V}{d\nu^2}(t,\lr(\xi))\left(\xi,\widehat{\xi}\right)\right)^\top \widehat{\eta} \right],
    \end{aligned}
    \end{equation}   
    and the derivatives $D_y \frac{dV}{d\nu}(t,\lr(\xi))(\xi),\ D_y^2 \frac{dV}{d\nu}(t,\lr(\xi))(\xi)$ are continuous.
\end{theorem}

The proof for Theorem~\ref{thm:value_regu_mu} is similar to that in \cite{AB10,AB13}, we provide here for complete, which is given in Appendix~\ref{pf:thm:value_regu_mu}. In the rest of this section, we aim to give the regularity of $V$ in $t$. We need the following additional assumption, which requires the H\"older continuity of the coefficients in the derivatives $b_v$, $\sigma^j_v$ and $f_v$.

{\bf (A3')} The coefficients $b$, $\sigma$ and $f$ satisfy (A3). The diffusion coefficient $\sigma$ is independent of the control argument (that is $l=0$, $b=B$, $f=f^0$). Moreover, for any $0\le t,\;t'\le T$ and $(x,m,v)\in\brn\times\pr_2(\brn)\times\brd$,
\begin{align*}
    \left|b_{v}\left(t',x,m,v^0\right)- b_{v}\left(t,x,m,v^0\right)\right| \le\ & \frac{L|t'-t|^{\frac{1}{2}}}{1+|x|+ \left|v\right|+|m|_1},\\
    \left|f_{v}\left(t',x,m,v\right)- f_{v}\left(t,x,m,v\right)\right|\le\ & L |t'-t|^{\frac{1}{2}}.
\end{align*}

For the sake of convenience, the control-independent assumption in (A3') is imposed here as it is much conveniently to establish the H\"older continuity of $u^{t,\xi}_s$ in time $s$ and henceforth the classically solvability of the HJB integro-particle differential equation; and (A3') is not necessary for the solvability of the MFTC problem and well-posedness of the FBSDEs with jumps. Before giving the regularity of $V$ in $t$, we need the following continuity of $\Theta^{t,\xi}_s$ in $s$, whose proof is given in Appendix~\ref{pf:lem:Theta_regu_s}.

\begin{lemma}\label{lem:Theta_regu_s}
    Under Assumptions (A1), (A2), (A3') and (A4), and the validity of Condition \eqref{thm:suff:condition}-(i,ii), we have for a.e. $s,s' \in[t,T]$, 
    \begin{align*}
        \left\|Y^{t,\xi}_{s'}-Y^{t,\xi}_s\right\|^2_2,\ \left\|P^{t,\xi}_{s'}-P^{t,\xi}_s\right\|^2_2,\ \left\|u^{t,\xi}_{s'}-u^{t,\xi}_s\right\|^2_2\le\ & C\left(1+\|\xi\|_2^2\right)|s'-s|,
    \end{align*}
    where $C>0$ is a constant depending only on $(l,L,L_0,L_1,L_2,\lambda_0,\lambda_x,\lambda_m,\lambda_v,T)$.
\end{lemma}

To give the differentiability of $V$ in $t$, we also need the following It\^o's formula for measure-dependent functionals and SDE with jump diffusion. We also refer to \cite{agram2023stochastic,guo2023ito} for similar results; and also refer to \cite{AB5} for It\^o's formula in mean field theory against Brownian motion only. Due to the page limit, we here only give a sketch of proof of this It\^o's formula (which is put in Appendix~\ref{pf:lem:Ito}); for detailed discussion, we refer to \cite[Corollary 3.6]{guo2023ito} for instance. 

\begin{lemma}\label{lem:Ito}
    Suppose that $X\in\sr_\f^2(0,T)$ is of the form
    \begin{align}\label{Ito:X}
        dX_s=b_s ds + \sigma_s dB_s + \int_E \gamma_s (e) \mathring{N}(de,ds),\quad s\in[0,T],
    \end{align}
    where $b,\sigma\in \mr_\f^2(0,T)$ and $\gamma\in\kr_{\f,\lambda}^2(0,T)$. Suppose that the functional $F:[0,T]\times\pr_2(\brn)\ni(t,\mu)\mapsto F(t,\mu)\in\br$ is $C^1$ in $t$ and linearly functional-differentiable in $\mu$ with the continuous derivative $D_y^2\frac{dF}{d\nu}$. Then,
    \begin{align}\label{Ito_0}
        &\frac{d}{ds}F(s,\lr(X_s)) \notag \\
        =\ & F_s(s,\lr(X_s))+ \e\Bigg\{\left(D_y\frac{dF}{d\nu}(s,\lr(X_s))(X_s)\right)^\top b_s +\frac{1}{2}\text{Tr}\left[\left(\sigma_s\sigma_s^\top\right) D_y^2\frac{dF}{d\nu}(s,\lr(X_s))(X_s)\right] \notag \\
        &+\int_E \bigg[\frac{dF}{d\nu}(s,\lr(X_{s}))(X_{s-}+\gamma_s(e))-\frac{dF}{d\nu}(s,\lr(X_{s}))(X_{s-}) \notag \\
        &\quad\qquad -\left(D_y\frac{dF}{d\nu}(s,\lr(X_s))(X_{s-})\right)^\top \gamma_s(e) \bigg] \lambda(de)\Bigg\}.
    \end{align}
\end{lemma}

We now establish the temporal regularity of $V$ based on the  It\^o's formula \eqref{Ito_0} and the continuity of $\Theta^{t,\xi}_s$ in $s$ in Lemma~\ref{lem:Theta_regu_s}. 

\begin{theorem}\label{thm:V_t}
    Under Assumptions (A1'), (A2'), (A3') and (A4), and the validity of Condition \eqref{thm:suff:condition}-(i,ii), the value functional $V$ is $C^1$ in $t$ with the temporal derivative
    \begin{align}
        \frac{\dd V}{\dd t}(t,\mu)=\ & -\e \Bigg\{\frac{1}{2}\text{Tr}\left[\left(\sigma\sigma^\top\right) \left(t,\xi,\mu,u^{t,\xi}_t\right) D_y^2\frac{dV}{d\nu}(t,\mu)(\xi)\right] \notag \\
        &\qquad +\left(D_y\frac{dV}{d\nu}(t,\mu)(\xi)\right)^\top b\left(t,\xi,\mu,u^{t,\xi}_t\right)  +f\left(t,\xi,\mu,u^{t,\xi}_t\right)\notag \\
        &\qquad +\int_E \bigg[\frac{dV}{d\nu}(t,\mu)\left(\xi+\gamma\left(t,\xi,\mu,e\right)\right)-\frac{dV}{d\nu}(t,\mu)(\xi) \notag \\
        &\quad\qquad\qquad -\left(D_y\frac{dV}{d\nu}(t,\mu)(\xi)\right)^\top \gamma\left(t,\xi,\mu,e\right) \bigg] \lambda(de) \Bigg\}, \label{thm:V_t_0}
    \end{align}
    where $\xi\in\lr_{\f_t}^2$ satisfying $\lr(\xi)=\mu$. 
\end{theorem}

\begin{proof}
For any $\mu\in\pr_2(\brn)$, we choose $\xi\in\lr_{\f_t}^2$ such that $\lr(\xi)=\mu$ which is independent of the Brownian motion and the jump. By the usual dynamic programming principle for McKean–Vlasov control problem with jump (see \cite[Section 4.1]{guo2023ito} for instance), for any $\epsilon\in[0,T-t]$,
\begin{equation*}
	V(t,\mu)=\e\left[\int_t^{t+\epsilon}f\left(s,\theta^{t,\xi}_s\right)ds\right]+V\left(t+\epsilon,\lr\left(Y^{t,\xi}_{t+\epsilon}\right)\right),
\end{equation*}
so we have
\begin{align}
		\frac{1}{\epsilon}\left[V(t+\epsilon,\mu)-V(t,\mu)\right] =\ & \frac{1}{\epsilon} \e\left[V(t+\epsilon,\mu)-V\left(t+\epsilon,\lr\left(Y^{t,\xi}_{t+\epsilon}\right)\right)\right] -\frac{1}{\epsilon} \e\left[\int_t^{t+\epsilon}f\left(s,\theta^{t,\mu}_s \right)ds\right]. \label{thm:V_t_1}
\end{align}	
From Theorem~\ref{thm:value_regu_mu}, Lemma~\ref{lem:Theta_regu_s} and the It\^o's formula in Lemma~\ref{lem:Ito}, we know that
\begin{align}
    &\lim_{\epsilon\to 0} \frac{1}{\epsilon} \e\left[V(t+\epsilon,\mu)-V\left(t+\epsilon,\lr\left(Y^{t,\xi}_{t+\epsilon}\right)\right)\right] \notag \\
    =\ & -\e \Bigg\{\left(D_y\frac{dV}{d\nu}(t,\mu)(\xi)\right)^\top b\left(t,\xi,\mu,u^{t,\xi}_t\right)  +\frac{1}{2}\text{Tr}\left[\left(\sigma\sigma^\top\right) \left(t,\xi,\mu,u^{t,\xi}_t\right) D_y^2\frac{dV}{d\nu}(t,\mu)(\xi)\right]\notag \\
    &\qquad +\int_E \left[\frac{dV}{d\nu}(t,\mu)\left(\xi+\gamma\left(r,\xi,\mu,e\right)\right)-\frac{dV}{d\nu}(t,\mu)(\xi)-\left(D_y\frac{dV}{d\nu}(t,\mu)(\xi)\right)^\top \gamma\left(t,\xi,\mu,e\right) \right] \lambda(de) \Bigg\}. \label{thm:V_t_2}
\end{align}
Again from Lemma~\ref{lem:Theta_regu_s}, we also have 
\begin{align*}
    \lim_{\epsilon\to 0} \frac{1}{\epsilon} \e\left[\int_t^{t+\epsilon}f\left(s,\theta^{t,\mu}_s \right)ds\right]=\e\left[f\left(s,\xi,\mu,u^{t,\xi}_t\right)\right].
\end{align*}
From \eqref{thm:V_t_1}, \eqref{thm:V_t_2} and the last equation, we obtain \eqref{thm:V_t_0}.
\end{proof}

\section{HJB integro-partial differential equation}\label{sec:HJB}

We define the map $\hr:[0,T]\times\brn\times\pr_2(\brn)\times\brd\times\brnn\times L_\lambda^2(E)\to\br$ as
\begin{align}
    \hr(t,x,\mu,v,p,\Gamma,k(\cdot)):= \ & \frac{1}{2}\text{Tr}\left[\left(\sigma\sigma^\top\right) \left(t,x,\mu\right) \Gamma\right] +p^\top \left[b\left(t,x,\mu,v\right)-\int_E \gamma(t,x,\mu,e)\lambda(de)\right]\notag \\
    & +f\left(t,x,\mu,v\right) +\int_E \left[k\left(x+\gamma\left(t,x,\mu,e\right)\right)-k(x) \right] \lambda(de). \label{def:hr}
\end{align}
Then, the derivative of $V$ in $t$ in Equation \eqref{thm:V_t_0} also writes:
\begin{align}
    \frac{\dd V}{\dd t}(t,\mu)=\ & -\e \left[ \hr \left(t,\xi,\mu,u^{t,\xi}_t,D_y\frac{dV}{d\nu}(t,\mu)(\xi),D_y^2\frac{dV}{d\nu}(t,\mu)(\xi),\frac{dV}{d\nu}(t,\mu)(\cdot)\right) \right]. \label{thm:V_t_0'}
\end{align}
Equation \eqref{thm:V_t_0'} inspires us to write down the following HJB integro-partial differential equation for MFTC problem with jump-diffusion: for $(t,\mu)\in[0,T]\times\pr_2(\brn)$, 
\begin{equation}\label{HJB:inf}
\left\{
    \begin{aligned}
        &\frac{\dd V}{\dd t} (t,\mu)+ \int_\brn \inf_{v\in\brd} \hr \left(t,x,\mu,v,D_y\frac{dV}{d\nu}(t,\mu)(x),D_y^2\frac{dV}{d\nu}(t,\mu)(x),\frac{dV}{d\nu}(t,\mu)(\cdot)\right) \mu(dx)=0, \\
        &V(T,\mu)= \int_\brn g(x,\mu)\mu(dx).
    \end{aligned}
\right.
\end{equation}
The formulation of our HJB equation \eqref{HJB:inf} is standard in the literature; see \cite[Equation (4.6)]{guo2023ito} and \cite[Section 4]{agram2023stochastic} for similar formulation. The aim of this section is to show that $V$ is the unique classical solution of Equation \eqref{HJB:inf}. We define the function $\mathbf{H}$ as follows: for any $(t,x,\mu)\in[0,T]\times\brn\times\pr_2(\brn)$ and $v\in\brd$,
\begin{align}\label{def:bfH}
    \mathbf{H}(t,x,\mu,v):=\hr\left(t,x,\mu,v,D_y\frac{dV}{d\nu}(t,\mu)(x),D_y^2\frac{dV}{d\nu}(t,\mu)(x),\frac{dV}{d\nu}(t,\mu)(\cdot)\right).
\end{align}
%and the main task of this section is to study the uniquely existence of the minimizer of the map $v\mapsto \mathbf{H}(t,x,\mu,v)$. 
The function $\mathbf{H}$ is well-defined due to the regularity of the function $V$ given in Theorem~\ref{thm:value_regu_mu}. Note in \eqref{def:bfH} that the function $H$ depends on the derivatives $D_y\frac{dV}{d\nu}$ and $D_y^2\frac{dV}{d\nu}$, and to give the characterization of these two derivatives, 
%which together with the first relation of \eqref{DV_P} and \eqref{relation:DV-dV} of Theorem~\ref{thm:value_regu_mu} inspires us to 
we consider the system \eqref{FBSDE:MFTC_y} in Subsection~\ref{subsec:D_y} corresponding to initial $(t,y,\mu)$, which also reads 
\begin{equation}\label{FBSDE:d_nu}
\left\{
\begin{aligned}
    Y^{t,y,\mu}_s=\ & y+\int_t^s H_p\left(r,Y^{t,y,\mu}_{r-},\lr\left(Y^{t,\xi}_{r-}\right),P^{t,y,\mu}_{r-}\right) dr \\
    &+\sum_{j=1}^n\int_t^s H_{q^j}\left(r,Y^{t,y,\mu}_{r-},\lr\left(Y^{t,\xi}_{r-}\right),Q^{t,y,\mu,j}_r\right)dB^j_r\\
    &+\int_t^s \int_E \gamma\left(r,Y^{t,y,\mu}_{r-},\lr\left(Y^{t,\xi}_{r-}\right),e\right) \mathring{N}(de,dr),\\
    P^{t,y,\mu}_s=\ & g_x\left(Y^{t,y,\mu}_T,\lr\left(Y^{t,\xi}_T\right)\right)+\he\left[D_y\frac{dg}{d\nu}\left(\widehat{Y^{t,\xi}_T},\lr\left(Y^{t,\xi}_T\right)\right)\left(Y^{t,y,\mu}_T\right)\right]\\
    &+\int_s^T \bigg\{ H_x\left(r,Y^{t,y,\mu}_{r-},\lr\left(Y^{t,\xi}_{r-}\right),P^{t,y,\mu}_{r-},Q^{t,y,\mu}_r,R^{t,y,\mu}_r\right)\\
    &\quad\qquad +\he\bigg[ D_y\frac{dH}{d\nu}\left(s,\widehat{Y^{t,\xi}_{r-}},\lr\left(Y^{t,\xi}_{r-}\right),\widehat{P^{t,\xi}_{r-}},\widehat{Q^{t,\xi}_r},\widehat{R^{t,\xi}_r}\right)\left(Y^{t,y,\mu}_{r-}\right) \bigg]\bigg\}dr\\
    &-\int_s^T Q^{t,y,\mu}_r dB_r-\int_s^T\int_E R^{t,y,\mu}_r(e) \mathring{N}(de,dr), \quad s\in[t,T].
\end{aligned}
\right.
\end{equation}
Under Assumptions (A3') and (A4) and Condition \eqref{thm:suff:condition}, we define the process $u^{t,y,\mu}$ as 
\begin{equation}\label{opti_condi}
\begin{aligned}
    &b_{v} \left(s,Y^{t,y,\mu}_{s-},\lr\left(Y^{t,\xi}_{s-}\right),u^{t,y,\mu}_s\right)^\top P^{t,y,\mu}_{s-}+f_{v} \left(s,Y^{t,y,\mu}_{s-},\lr\left(Y^{t,\xi}_{s-}\right),u^{t,y,\mu}_s\right)=0.
\end{aligned}
\end{equation}
We denote by the solution $\Theta^{t,y,\mu}:=\left(Y^{t,y,\mu},P^{t,y,\mu},Q^{t,y,\mu},R^{t,y,\mu},u^{t,y,\mu}\right) \in \mathbb{S}$. From Theorem~\ref{thm:D^2dV}, we know that $\Theta^{t,y,\mu}$ is G\^ateaux differentiable in the initial $y\in\brn$, and we denote by $D_y \Theta^{t,y,\mu}:=\left(D_y Y^{t,y,\mu},D_y P^{t,y,\mu},D_y Q^{t,y,\mu},D_y R^{t,y,\mu},D_y u^{t,y,\mu}\right)\in\mathbb{S}$ the G\^ateaux derivative of $\Theta^{t,y,\mu}$ with respect to $y$, which is the unique solution of FBSDEs with jumps \eqref{FBSDE:D_yTheta}-\eqref{thm:D^2dV_1} corresponding to the initial time $t$. 

We now give the boundedness estimate of the linear functional-derivatives $D_y\frac{dV}{d\nu}$ and $D_y^2\frac{dV}{d\nu}$, and also provide the characterization of the processes $Q^{t,y,\mu}$ and also $R^{t,y,\mu}$ with the linear functional-derivative of the value function $V$. 
%The former will be used to show the convexity of the function $\mathbf{H}$ later, and the latter will be used to show that $\mathbf{H}_{v}\left(t,x,\mu,u^{t,x,\mu}_t\right)=0$.
The similar characterization was also given in \cite{MR1176785} for quasilinear BSDE; in our article, the characterization of $Q^{t,y,\mu}$ and $R^{t,y,\mu}$ by the value function $V$ do not appear in the HJB equation \eqref{HJB:inf}, and we give it as a property, since it can facilitate any possible  numerical studies in future works. 
 
\begin{proposition}\label{thm:value_regu_mu'}
    Under Assumptions (A1'), (A2'), (A3) and (A4), and the validity of Condition \eqref{thm:suff:condition}-(i,ii), the linear functional-derivatives of $V$ satisfy the relations
    \begin{align}\label{relation:DV}
        &D_y\frac{dV}{d\nu}(t,\mu)(y)=P^{t,y,\mu}_t,\quad D_y^2\frac{dV}{d\nu}(t,\mu)(y)=D_y P^{t,y,\mu}_t,
    \end{align}
    with the derivatives satisfying 
    \begin{align}
        &\left|D_y\frac{dV}{d\nu}(t,\mu)(y)\right|\le  C \left[1+|y|+|\mu|_2 \right], \label{boundedness:DdV}\\
        &\left|D_y^2\frac{dV}{d\nu}(t,\mu)(y)\right|\le C. \label{thm:D^2dV_02}
    \end{align}
    and the derivative $D_y^2 \frac{dV}{d\nu}(t,\mu)(y)$ is continuous. Moreover, under Assumption (A3'), we have
    \begin{equation}\label{lem:Q&V_00}
    \begin{aligned}
        &Q^{t,y,\mu,j}_t=D_y^2\frac{dV}{d\nu}\left(t,\mu\right)(y)\sigma^j\left(t,y,\mu\right), \\
        &R^{t,y,\mu}_t(e)= D_y\frac{dV}{d\nu}\left(t,\mu\right)\left(\xi+\gamma\left(t,y,\mu,e\right)\right)-D_y\frac{dV}{d\nu}\left(t,\mu\right)\left(y\right).
    \end{aligned}
    \end{equation}
\end{proposition}

%{Remark: Compare with \cite[Page 8]{MR1176785}}

%{Peskir: It\^o's lemma for non-differentiable function for $C^1$ setting for mean field setting}

The proof of Proposition is given in Appendix~\ref{thm:value_regu_mu'}. We now give the main result of this section, which shows that $V$ is the unique classical solution of the HJB integro-partial differential equation \eqref{HJB:inf}. 

\begin{theorem}\label{thm:main:HJB}
    Under Assumptions (A1'), (A2'), (A3') and (A4), and the validity of Condition \eqref{thm:suff:condition}-(i,ii), the value function $V$ is the unique classical solution of the HJB integro-partial differential equation \eqref{HJB:inf} (satisfying Conditions \eqref{boundedness:DdV} and \eqref{thm:D^2dV_02}).  
\end{theorem}

\begin{proof}
In view of Equation \eqref{thm:V_t_0'}, to prove the existence result, we only need to show that for any $(t,x,\mu)\in[0,T]\times\brn\times\pr_2(\brn)$, we have $u^{t,x,\mu}_t=\argmin_{v\in\brd} \mathbf{H}\left(t,x,\mu,v\right)$. Indeed, from \eqref{def:hr}, \eqref{def:bfH} and Assumption (A3), we can compute that
\begin{equation*}\label{H_v^0}
\begin{aligned}
    \mathbf{H}_{v}(t,x,\mu,v)=\ & \left(b_v\left(t,x,\mu,v\right)\right)^\top D_y\frac{dV}{d\nu}(t,\mu)(x)+f_v \left(t,x,\mu,v\right),
\end{aligned}
\end{equation*}
then, from \eqref{opti_condi} and \eqref{relation:DV}, we know that
\begin{align*}
    \mathbf{H}_{v}\left(t,x,\mu,u^{t,x,\mu}_t\right)=\ & \left(b_{v}\left(t,x,\mu,u^{t,x,\mu}_t\right)\right)^\top P^{t,x,\mu}_t+f_{v} \left(t,x,\mu,u^{t,x,\mu}_t\right)=0.
\end{align*}
By following a similar approach as in Proposition~\ref{prop:L}, we know that for any $(t,x,\mu)$, there is a unique $\hv$ satisfying $\mathbf{H}_{v}\left(t,x,\mu,\hv\right)=0$, which together with the convexity assumption on $f$ imply that $u^{t,x,\mu}_t$ is the unique minimizer of the map $v\mapsto\mathbf{H}\left(t,x,\mu,v\right)=0$, which completes the proof for the existence. 

We next prove the uniqueness result. Suppose that $U$ is another classical solution of the HJB integro-partial differential equation \eqref{HJB:inf}. We come back to the Problem $\left(\mathbf{P}^{t,\xi}\right)$ (for some $\xi\in\lr_{\f_t}^2$ with $\lr(\xi)=\mu$). For any admissible control $v$ for Problem $\left(\mathbf{P}^{t,\xi}\right)$, we denote by $X^{t,\xi,v}$ the corresponding controlled state process. Since $U$ also satisfies Conditions \eqref{boundedness:DdV} and \eqref{thm:D^2dV_02}, by applying the mean field version of It\^o's formula in Lemma~\ref{lem:Ito} on $U\left(s,\lr\left(X^{t,\xi,v}_s\right)\right)$ and using the definition of the functional $\hr$ in \eqref{def:hr},
\begin{align*}
    &U\left(T,\lr\left(X^{t,\xi,v}_T\right)\right)-U(t,\mu))\\
    =\ & \int_t^T \e\Bigg\{\frac{\dd U}{\dd s}\left(s,\lr\left(X^{t,\xi,v}_s\right)\right)+ \left(D_y\frac{dU}{d\nu}\left(s,\lr\left(X^{t,\xi,v}_s\right)\right)\left(X^{t,\xi,v}_s\right)\right)^\top b\left(s,X^{t,\xi,v}_s,\lr\left(X^{t,\xi,v}_s\right),v_s\right) \\
    &\quad\qquad +\frac{1}{2}\text{Tr}\left[\left(\sigma\sigma^\top\right)\left(s,X^{t,\xi,v}_s,\lr\left(X^{t,\xi,v}_s\right),v_s\right) D_y^2\frac{dU}{d\nu}\left(s,\lr\left(X^{t,\xi,v}_s\right)\right)\left(X^{t,\xi,v}_s\right)\right] \notag \\
    &\quad\qquad +\int_E \bigg[\frac{dU}{d\nu}\left(s,\lr\left(X^{t,\xi,v}_s\right)\right)\left(\left(X^{t,\xi,v}_{s-}\right)+\gamma\left(s,X^{t,\xi,v}_s,\lr\left(X^{t,\xi,v}_s\right),e\right)\right) \\
    &\qquad\qquad\qquad -\frac{dU}{d\nu}\left(s,\lr\left(X^{t,\xi,v}_s\right)\right)\left(X^{t,\xi,v}_{s-}\right) \\
    &\qquad\qquad\qquad -\left(D_y\frac{dU}{d\nu}\left(s,\lr\left(X^{t,\xi,v}_s\right)\right)\left(X^{t,\xi,v}_{s-}\right)\right)^\top \gamma\left(s,X^{t,\xi,v}_s,\lr\left(X^{t,\xi,v}_s\right),e\right) \bigg] \lambda(de)\Bigg\} ds \notag\\
    =\ & \int_t^T \e\Bigg[\frac{\dd U}{\dd s}\left(s,\lr\left(X^{t,\xi,v}_s\right)\right)-f\left(s,X^{t,\xi,v}_s,\lr\left(X^{t,\xi,v}_s\right),v_s\right) \\
    &\quad\qquad +\hr\bigg(s,X^{t,\xi,v}_s,\lr\left(X^{t,\xi,v}_s\right),v_s,D_y\frac{dU}{d\nu}\left(s,\lr\left(X^{t,\xi,v}_s\right)\right)\left(X^{t,\xi,v}_s\right), \\
    &\qquad\qquad\qquad D_y^2\frac{dU}{d\nu}\left(s,\lr\left(X^{t,\xi,v}_s\right)\right)\left(X^{t,\xi,v}_s\right),\frac{dU}{d\nu}\left(s,\lr\left(X^{t,\xi,v}_s\right)\right)(\cdot)\bigg)\Bigg] ds.
\end{align*}
Since $U$ satisfies the HJB equation \eqref{HJB:inf}, we know that for $s\in[t,T]$,
\begin{align*}
    \frac{\dd U}{\dd s}\left(s,\lr\left(X^{t,\xi,v}_s\right)\right)=\ & - \e\bigg[ \inf_{v\in\brd} \hr \bigg(s,X^{t,\xi,v}_s,\lr\left(X^{t,\xi,v}_s\right),v,D_y\frac{dU}{d\nu}\left(s,\lr\left(X^{t,\xi,v}_s\right)\right)\left(X^{t,\xi,v}_s\right),\\
    &\qquad\qquad\qquad D_y^2\frac{dU}{d\nu}\left(s,\lr\left(X^{t,\xi,v}_s\right)\right)\left(X^{t,\xi,v}_s\right),\frac{dU}{d\nu}\left(s,\lr\left(X^{t,\xi,v}_s\right)\right)(\cdot)\bigg) \bigg], \\
    U\left(T,\lr\left(X^{t,\xi,v}_T\right)\right)=\ & \e\left[g\left(X^{t,\xi,v}_T,\lr\left(X^{t,\xi,v}_T\right)\right)\right].
\end{align*}
Therefore, we can deduce that
\begin{align}
    &J_{t,\xi}(v) - U(t,\mu) \notag\\
    =\ & \int_t^T \e\Bigg[\hr\bigg(s,X^{t,\xi,v}_s,\lr\left(X^{t,\xi,v}_s\right),v_s,D_y\frac{dU}{d\nu}\left(s,\lr\left(X^{t,\xi,v}_s\right)\right)\left(X^{t,\xi,v}_s\right),  \notag \\
    &\quad\qquad\qquad D_y^2\frac{dU}{d\nu}\left(s,\lr\left(X^{t,\xi,v}_s\right)\right)\left(X^{t,\xi,v}_s\right),\frac{dU}{d\nu}\left(s,\lr\left(X^{t,\xi,v}_s\right)\right)(\cdot)\bigg)  \notag \\
    &\quad\qquad - \inf_{v\in\brd} \hr \bigg(s,X^{t,\xi,v}_s,\lr\left(X^{t,\xi,v}_s\right),v,D_y\frac{dU}{d\nu}\left(s,\lr\left(X^{t,\xi,v}_s\right)\right)\left(X^{t,\xi,v}_s\right), \notag \\
    &\qquad\qquad\qquad\qquad D_y^2\frac{dU}{d\nu}\left(s,\lr\left(X^{t,\xi,v}_s\right)\right)\left(X^{t,\xi,v}_s\right),\frac{dU}{d\nu}\left(s,\lr\left(X^{t,\xi,v}_s\right)\right)(\cdot)\bigg) \Bigg] ds \label{thm:main:HJB_3} \\
    \geq\ & 0. \notag
\end{align}
From the arbitrariness of the control $v$, we then know that
\begin{align}\label{thm:main:HJB_4}
    U(t,\mu)\le \inf_{v\in\mr_\f^2(t,T)}J_{t,\xi}(v).
\end{align}
Similar as in \eqref{def:bfH}, we define the function
\begin{align*}
    \mathbf{I}(t,x,\mu,v):=\hr \left(t,x,\mu,v,D_y\frac{dU}{d\nu}(t,\mu)(x),D_y^2\frac{dU}{d\nu}(t,\mu)(x),\frac{dU}{d\nu}(t,\mu)(\cdot)\right).
\end{align*}
By a similar approach as in Proposition~\ref{prop:L}, we can also know that for any $(t,x,\mu)$, there is a unique minimizer of $\mathbf{I}(t,x,\mu,v)$, denoted by $\widehat{u}(t,x,\mu)$. Now we choose $\widehat{u}_s:=\widehat{u}\left(s,X_s,\lr(X_s) \right)$ as this feedback control for Problem $\left(\mathbf{P}^{t,\xi}\right)$, i.e., the corresponding state process is 
\begin{align*}
    X_s=\xi & +\int_t^s b\left(r,X_r,\lr(X_r),\widehat{u}(r,X_r,\lr(X_r))\right) dr +\int_t^s \sigma\left(r,X_r,\lr(X_r)\right) dB_r\\
    &+\int_t^s\int_E \gamma\left(r,X_{r-},\lr(X_{r-}),e\right) \mathring{N}(de,dr),
\end{align*}
then, in \eqref{thm:main:HJB_3}, we know that 
\begin{align*}
    J_{t,\xi}\left(\hv\right) - U(t,\mu) =0.
\end{align*}
The last equation and \eqref{thm:main:HJB_4} imply that
\begin{align*}
    U(t,\mu)= \inf_{v\in\mr_\f^2(t,T)}J_{t,\xi}(v)=V(t,\mu),
\end{align*}
which completes the proof.
\end{proof}

\section{Summary and future work on mean field games}

So far, we studied the MFTC problems \eqref{problem} driven by jump-diffusions. We first gave the corresponding sufficient and necessary maximum principle in Section \eqref{sec:maxi}, and we also gave the well-posedness for the system of FBSDEs with jumps \eqref{FBSDE:MFTC} arsing from the maximum principle in Section~\ref{sec:FBSDE}. The study for the G\^ateaux derivatives with respect to the initial condition for the solution of the FBSDEs are also given in Section~\ref{sec:FBSDE}, which are used to study in Section~\ref{sec:value} of the regularity for the value function $V$. Finally, we show that $V$ is the unique classical solution for the HJB integro-partial differential equation \eqref{HJB:inf} in Section~\ref{sec:HJB}.

Our current solution scheme for MFTC problem \eqref{problem} can also be applied for the study on the MFGs driven by jump-diffusions, which is the following fixed point problem: for initial $(t,\mu)\in [0,T]\times\pr_2(\brn)$ and $\xi\in\lr_{\f_t}^2$ with $\lr(\xi)=\mu$,  given a distribution flow $m=\{m_s\in\pr_2(\brn):\ t\le s\le T\}$ with $m_t=\mu$,  consider the following control problem which depends on $m$
\begin{equation}\label{problem_mfg:intro}
    \left\{
    \begin{aligned}
        &v^m\in\argmin_{v\in \mr_\f^2(t,T)}J(v|m):=\e\left[\int_t^T f\left(s,X_s^{m,v},m_s,v_s\right) ds + g\left(X_T^{m,v},m_T\right)\right],\\
        &\text{such that } X_s^{m,v}=\xi +\int_t^s b\left(r,X_r^{m,v},m_r,v_r\right)dr + \int_t^s \sigma\left(r,X_r^{m,v},m_r,v_r\right)dB_r \\
        &\qquad\qquad\qquad\qquad +\int_t^s \int_E \gamma\left(r,X_{r-}^{m,v},m_{r-},v_r,e\right) \mathring{N}(de,dr),\quad s\in[t,T],
    \end{aligned}
    \right.
\end{equation}
if the following consistency condition 
\begin{align}\label{consistence_condition}
    \lr\left(X^{m,v^m}_s\right)=m_s,\quad s\in[t,T]
\end{align}
holds for some distribution flow $m^{t,\mu}$, then $m^{t,\mu}$ is an equilibria for the MFG with the initial $(t,\mu)$. The major difference between MFG and MFTC problem is that, in a MFTC problem, the state process depends in the law of the current state; while in a MFG, the fixed point distribution flow $m$ is the law of the state corresponding to the solution of the control problem \eqref{problem_mfg:intro} with the parameter $m$. Therefore, the related system of FBSDEs with jumps for MFG \eqref{problem_mfg:intro}-\eqref{consistence_condition} arsing from the maximum principle is different from \eqref{FBSDE:MFTC}, which is
\begin{equation}\label{FBSDE:MFG}
    \left\{
        \begin{aligned}
            &Y_s=\xi+\int_t^s b\left(s,\theta_s\right) ds +\int_t^s \sigma \left(s,\theta_s\right)dB_s+\int_t^s \int_E \gamma\left(s,\theta_{s-},e\right) \mathring{N}(de,ds),\\
            &P_t= g_x(Y_T,\lr(Y_T))+\int_t^T L_x\left(s,\theta_s,P_s,Q_s,R_s\right)ds -\int_t^T Q_s dB_s-\int_t^T\int_E R_s(e) \mathring{N}(de,ds), \\
            &L_v\left(s,\theta_{s-}, P_{s-}, Q_s, R_s\right)=0, \quad s\in[t,T],
        \end{aligned}
    \right.
\end{equation}
where $\theta_s:=\left( Y_s,\lr(Y_s),u_s\right)$, and the Lagrangian function $L$ is defined in \eqref{Lagrangian}, same as that for MFTC problem. Different from MFTC problem \eqref{problem}, for the well-posedness of \eqref{FBSDE:MFG} and the solvability of the MFG \eqref{problem_mfg:intro}-\eqref{consistence_condition}, we do not require the differentiability of the coefficients with respect to the distribution variable, instead, we need some monotonicity conditions, such as the commonly used small mean-field effect condition or $\beta$-monotonicity in \cite{AB11,AB13,ABmfg1}; and for a discussion on the relation between the  $\beta$-monotonicity and various monotonicity conditions including displacement monotonicity and Lions monotonicity, we refer to \cite{AB12}. We denote by $\theta^{t,\xi}_s=\left( Y^{t,\xi}_s,\lr\left(Y^{t,\xi}_s\right),u^{t,\xi}_s\right)$ the solution of MFG \eqref{problem_mfg:intro}-\eqref{consistence_condition}, then, given the equilibrium distribution flow $m^{t,\mu}_s:=\lr\left(Y^{t,\xi}_s\right)$ and any initial state $x\in\brn$, we also consider the following stochastic optimal control problem: 
\begin{equation}\label{problem_mfg:intro'}
    \left\{
    \begin{aligned}
        &u^{t,x,\mu}\in\argmin_{v\in \mr_\f^2(t,T)}J_{t,x,\mu}(v):=\e\left[\int_t^T f\left(s,X_s^{v},m^{t,\mu}_s,v_s\right) ds + g\left(X_T^{v},m^{t,\mu}_T\right)\right],\\
        &\text{such that } X_s^{v}=x +\int_t^s b\left(r,X_r^{v},m^{t,\mu}_r,v_r\right)dr + \int_t^s \sigma\left(r,X_r^{v},m^{t,\mu}_r,v_r\right)dB_r \\
        &\quad\qquad\qquad\qquad +\int_t^s \int_E \gamma\left(r,X_{r-}^{v},m^{t,\mu}_{r-},v_r,e\right) \mathring{N}(de,dr),\quad s\in[t,T],
    \end{aligned}
    \right.
\end{equation}
and define the value function as
\begin{align}\label{value:mfg}
    V(t,x,\mu):= \min_{v\in \mr_\f^2(t,T)}J_{t,x,\mu}(v)=\ & J_{t,x,\mu}\left(u^{t,x,\mu}\right)\notag\\
    =\ & \e\left[\int_t^T f\left(s,Y_s^{t,x,\mu},m^{t,\mu}_s,u_s^{t,x,\mu}\right) ds + g\left(Y_T^{t,x,\mu},m^{t,\mu}_T\right)\right].
\end{align}
where $Y^{t,x,\mu}$ is the state process corresponding to the optimal control $u^{t,x,\mu}$; also see \cite{BR} for the purpose of studying two mean field SDEs with respective initials $\xi$ and $x$. Following the similar approach as for MFTC problem \eqref{problem}, by studying of the regularity for the processes $\left(Y^{t,\xi}, u^{t,\xi}\right)$ and $\left(Y^{t,x,\mu}, u^{t,x,\mu}\right)$ (and also the related adjoint processes), the value function $V$ defined in \eqref{value:mfg} is expected to be smooth enough under appropriate assumptions similar as in \cite{AB13} on the coefficients. Then, when $\sigma$ and $\gamma$ are independent of control argument, by applying the mean field version of It\^o's formula in Lemma~\ref{lem:Ito} for measure-dependent functionals and SDE with jump diffusion on $V\left(s,Y^{t,x,\mu},m^{t,\mu}_s\right)$, the value function $V$ defined in \eqref{value:mfg} is expected to be shown as a classical solution of the following integro-partial differential equation:
\begin{equation}\label{master}
    \left\{
    \begin{aligned}
        &\frac{\dd V}{\dd t} (t,x,\mu)= -\frac{1}{2}\text{Tr}\left[\left(\sigma\sigma^\top\right) \left(t,x,\mu\right) D_x^2 V(t,x,\mu)\right]\\
        &\quad\qquad -b\left(t,x,\mu,u^{t,x,\mu}_t\right)^\top D_x V(t,x,\mu)-f\left(t,x,\mu,u^{t,x,\mu}_t\right)\\
        &\quad\qquad -\int_E \left[V(t,x+\gamma(t,x,\mu,e),\mu)-V(t,x,\mu)-D_x V(t,x,\mu)^\top \gamma\left(t,x,\mu,e\right) \right] \lambda(de)\\
        &\quad\qquad -\e \Bigg\{\frac{1}{2}\text{Tr}\left[\left(\sigma\sigma^\top\right) \left(t,\xi,\mu\right) D_y^2\frac{dV}{d\nu}(t,x,\mu)(\xi)\right] + b\left(t,\xi,\mu,u^{t,\xi}_t\right)^\top D_y\frac{dV}{d\nu}(t,x,\mu)(\xi)   \\
        &\quad\qquad \qquad +\int_E \bigg[\frac{dV}{d\nu}(t,x,\mu)\left(\xi+\gamma\left(t,\xi,\mu,e\right)\right)-\frac{dV}{d\nu}(t,x,\mu)(\xi)\\
        &\qquad\qquad\qquad\qquad  -\left(D_y\frac{dV}{d\nu}(t,x,\mu)(\xi)\right)^\top \gamma\left(t,\xi,\mu,e\right) \bigg] \lambda(de) \Bigg\}, \quad t\in[0,T),\\
        &V(T,x,\mu)= g(x,\mu),\quad (x,\mu)\in\brn\times\pr_2(\brn).
    \end{aligned}
    \right.
\end{equation}
Equation \eqref{master} is the master equation for MFG with jump-diffusion, which is also a decoupling field for the HJB-FP system for MFG; we refer to \cite{AB11,AB13} for the study for the classical solution of the master equation for MFG without diffusion only. The integral terms in the master equation \eqref{master} make it difficult to be solved by the classical analytical method, while our probabilistic approach can provide a positive definite answer. We expect that, when the coefficient functions are $C^1$ in both spatial and control variables and are continuous in the distribution argument, MFG \eqref{problem_mfg:intro} and Problem \eqref{problem_mfg:intro'} can be warranted with a unique solution; when the coefficient functions are $C^2$ in both spatial and control variables and are also $C^1$ in the distribution argument, the value functional in \eqref{value:mfg} can be shown to be $C^2$ in spatial variable and $C^1$ in distribution variable; furthermore, when the coefficient functions are $C^3$ in both spatial and control variables and $C^2$ in the distribution argument, the master integro-partial differential equation \eqref{master} can be shown to have a unique classical solution. 

\section*{Acknowledgement}

Alain Bensoussan is supported by the National Science Foundation under grant NSF-DMS-2204795. Ziyu Huang acknowledges the financial supports as a postdoctoral fellow from Department of Statistics of The Chinese University of Hong Kong. Shanjian Tang is supported by the National Natural Science Foundation of China under grant nos. 11631004 and  12031009. Phillip Yam acknowledges the financial supports from HKGRF-14301321 with the project title ``General Theory for Infinite Dimensional Stochastic Control: Mean field and Some Classical Problems'', and HKGRF-14300123 with the project title ``Well-posedness of Some Poisson-driven Mean Field Learning Models and their Applications''. The work described in this article was also supported by a grant from the Germany/Hong Kong Joint Research Scheme sponsored by the Research Grants Council of Hong Kong and the German Academic Exchange Service of Germany (Reference No. G-CUHK411/23). He also thanks The University of Texas at Dallas for the kind invitation to be a Visiting Professor in Naveen Jindal School of Management.

%%%%%%%%%%%%%%%%%%%%%%%%%%%%%%%%%%%%%%%%%%%%%%%%%%%%%%%%%%%%%%%%%%%%

\newpage

\normalsize 
\appendix
\appendixpage
\addappheadtotoc

In all the following proofs, we mean by $C(\alpha_1,\dots, \alpha_k)$  a constant depending only on parameters $(\alpha_1,\dots, \alpha_k)$, which may be different in different lines.

\section{Proof of statements in Section~\ref{sec:maxi}}\label{sec:pf:maxi}

\subsection{Proof of Lemma~\ref{lem:control-state}}\label{pf:lem:control-state}

For the existence and uniqueness of the $\f_t$-progressively-measurable c\`adl\`ag solution of SDE \eqref{SDE^v}, we refer to \cite[Theorem 1.19]{Oksendal-Sulem}; here, we only prove \eqref{lem:control-state-1}. From the It\^o's formula for $|X^v_t|^2$ (see \cite[Theorem 1.16]{Oksendal-Sulem} for instance), we know that
\begin{align*}
    \left|X^v_t\right|^2=\ & |\xi|^2+\int_0^t \bigg[ 2\left(X^v_s\right)^\top b\left(s,X_s^v,\lr\left(X_s^v\right),v_s\right)+\sum_{j=1}^n \left|\sigma^j\left(s,X_s^v,\lr\left(X_s^v\right),v_s\right)\right|^2\\
    &\quad\qquad\qquad +\int_E \left|\gamma\left(s,X_{s-}^v,\lr\left(X_{s-}^v\right),v_s,e\right)\right|^2\lambda(de)\bigg]ds \\
    &+2 \int_0^t \left(X^v_s\right)^\top\sigma \left(s,X_s^v,\lr\left(X_s^v\right),v_s\right)dB_s\\
    &+\int_E \left[2\left(X^v_{s-}\right)^\top \gamma\left(s,X_{s-}^v,\lr\left(X_{s-}^v\right),v_s,e\right)+\left|\gamma\left(s,X_{s-}^v,\lr\left(X_{s-}^v\right),v_s,e\right)\right|^2 \right]\mathring{N}(de,ds).
\end{align*}
Therefore, from Assumption (A1) and the Cauchy's inequality, we know that 
\begin{align*}
    \e\left[\sup_{0\le t\le T}\left|X^v_t\right|^2\right]\le \e\left[|\xi|^2+C(L)\int_0^T  \left(1+\left|X_t^v\right|^2+\left|\lr\left(X_t^v\right)\right|_2^2+|v_t|^2\right)dt\right].
\end{align*}
By applying Gr\"onwall's inequality and noting that $\left|\lr\left(X_t^v\right)\right|_2^2= \e\left[\left|X_t^v\right|^2\right]$, we obtain \eqref{lem:control-state-1}. 

\subsection{Proof of Lemma~\ref{lem:rhoX}}\label{pf:lem:rhoX}

For the well-posedness of SDE \eqref{SDE:deltaX}, we also refer to \cite[Theorem 1.19]{Oksendal-Sulem}; and the proof of Estimate \eqref{deltaX:boundedness} is similar to that of \eqref{lem:control-state-1} under Assumption (A1), which is also omitted. Here, we only prove \eqref{Delta_delta_distance}. From SDE \eqref{SDE^v}, we know that the process $\frac{\de X^\epsilon}{\epsilon}$ satisfies the following SDE: for $t\in[0,T]$,
\begin{align*}
    \frac{\de X^\epsilon_t}{\epsilon}=\ &\int_0^t \int_0^1 \bigg\{ b_x\left(s,\theta^{\epsilon,h}_s\right) \frac{\de X^\epsilon_s}{\epsilon}+b_v\left(s,\theta^{\epsilon,h}_s\right) \de v_s + \he\bigg[\left(D_y\frac{db}{d\nu}\left(s,\theta^{\epsilon,h}_s\right)\left(\widehat{X_s^{\epsilon,h}}\right)\right) \frac{\widehat{\de X_s^\epsilon}}{\epsilon}\bigg] \bigg\} dh ds \\
    &+\int_0^t \int_0^1 \bigg\{ \sigma_x\left(s,\theta^{\epsilon,h}_s\right) \frac{\de X_s^\epsilon}{\epsilon} +\sigma_v\left(s,\theta^{\epsilon,h}_s\right) \de v_s \\
    &\quad\qquad\qquad + \he\bigg[\left(D_y\frac{d\sigma}{d\nu}\left(s,\theta^{\epsilon,h}_s\right)\left(\widehat{X_s^{\epsilon,h}}\right)\right) \frac{\widehat{\de X_s^\epsilon}}{\epsilon}\bigg] \bigg\} dhdB_s\\
    &+\int_0^t\int_E \int_0^1 \bigg\{ \gamma_x\left(s,X_{s-}^{\epsilon,h},\lr\left(X_{s-}^{\epsilon,h}\right),v^{\epsilon,h}_s,e\right) \frac{\de X_{s-}^\epsilon}{\epsilon} +\gamma_v\left(s,X_{s-}^{\epsilon,h},\lr\left(X_{s-}^{\epsilon,h}\right),v^{\epsilon,h}_s,e\right) \de v_s \\
    &\qquad\qquad\qquad + \he\bigg[\left(D_y\frac{d\gamma}{d\nu}\left(s,X_{s-}^{\epsilon,h},\lr\left(X_{s-}^{\epsilon,h}\right),v^{\epsilon,h}_s,e\right)\left(\widehat{X_{s-}^{\epsilon,h}}\right)\right) \frac{\widehat{\de X_{s-}^\epsilon}}{\epsilon}\bigg] \bigg\} dh \mathring{N}(de,ds),
\end{align*}
where 
\begin{align*}
    \theta^{\epsilon,h}_t:=\left(X_t^{\epsilon,h},\lr\left(X_t^{\epsilon,h}\right),v^{\epsilon,h}_t\right),\quad X^{\epsilon,h}_t:=Y_t+h\epsilon\de X_t^\epsilon,\quad v^{\epsilon,h}_t:=u_t+h\epsilon\de v_t.
\end{align*}
Following a similar approach as the proof of \eqref{lem:control-state-1}, from Assumption (A1), we have
\begin{align}\label{DeX:boundedness}
    &\e\left[\sup_{0\le t\le T}\left|\frac{\de X^\epsilon_t}{\epsilon}\right|^2\right]\le C(L,T)\e\left[\int_0^T |\de v_t|^2 dt \right].
\end{align}
We denote by $\chi_t^\epsilon:= \frac{\de X_t^\epsilon}{\epsilon}-\delta X_t$ for $\epsilon\in(0,1)$ and $t\in[0,T]$. Then, from SDE \eqref{SDE:deltaX}, we know that the process $\chi^\epsilon$ satisfies the following SDE:
\begin{align}
    \chi^\epsilon_t=\ &\int_0^t \bigg\{b_x\left(s,\theta_s\right) \chi^\epsilon_s+\he\bigg[\left(D_y\frac{db}{d\nu}\left(s,\theta_s\right)\left(\hys\right)\right) \widehat{\chi^\epsilon_s}\bigg]+\alpha^{b,\epsilon}_s \bigg\}ds \notag \\
    &+\int_0^t \bigg\{\sigma_x\left(s,\theta_s\right) \chi^\epsilon_s+\he\bigg[\left(D_y\frac{d\sigma}{d\nu}\left(s,\theta_s\right)\left(\hys\right)\right) \widehat{\chi^\epsilon_s}\bigg]+\alpha^{\sigma,\epsilon}_s \bigg\}dB_s \notag \\
    &+\int_0^t\int_E \bigg\{\gamma_x\left(s,Y_{s-},\lr\left(Y_{s-}\right),u_s,e\right) \chi^\epsilon_{s-}+\alpha^{\gamma,\epsilon}(s,s-,e) \notag \\
    &\quad\qquad\qquad +\he\bigg[\left(D_y\frac{d\gamma}{d\nu}\left(s,Y_{s-},\lr\left(Y_{s-}\right),u_s,e\right)\left(\widehat{Y_{s-}}\right)\right) \widehat{\chi^\epsilon_{s-}}\bigg] \bigg\} \mathring{N}(de,ds), \label{SDE:rhoX}
\end{align}
where
\begin{align*}
    &\alpha^{b,\epsilon}_s:= \int_0^1 \bigg\{ \left[b_x\left(s,\theta^{\epsilon,h}_s\right)-b_x\left(s,\theta_s\right)\right] \frac{\de X^\epsilon_s}{\epsilon} +\left[b_v\left(s,\theta^{\epsilon,h}_s\right)-b_v\left(s,\theta_s\right) \right] \de v_s \\
    &\quad\qquad\qquad + \he\bigg[\left(D_y\frac{db}{d\nu}\left(s,\theta^{\epsilon,h}_s\right)\left(\widehat{X_s^{\epsilon,h}}\right)- D_y\frac{db}{d\nu}\left(s,\theta_s\right)\left(\hys\right)\right) \frac{\widehat{\de X_s^\epsilon}}{\epsilon}\bigg] \bigg\} dh;\\
    &\alpha^{\sigma,\epsilon}_s:= \int_0^1 \bigg\{ \left[\sigma_x\left(s,\theta^{\epsilon,h}_s\right)-\sigma_x\left(s,\theta_s\right)\right] \frac{\de X^\epsilon_s}{\epsilon}+\left[\sigma_v\left(s,\theta^{\epsilon,h}_s\right)-\sigma_v\left(s,\theta_s\right) \right] \de v_s \\
    &\quad\qquad\qquad + \he\bigg[\left(D_y\frac{d\sigma}{d\nu}\left(s,\theta^{\epsilon,h}_s\right)\left(\widehat{X_s^{\epsilon,h}}\right)- D_y\frac{d\sigma}{d\nu}\left(s,\theta_s\right)\left(\hys\right)\right) \frac{\widehat{\de X_s^\epsilon}}{\epsilon}\bigg] \bigg\} dh;\\
    &\alpha^{\gamma,\epsilon}(s,s-,e):= \int_0^1 \bigg\{ \left[\gamma_x\left(s,X_{s-}^{\epsilon,h},\lr\left(X_{s-}^{\epsilon,h}\right),v^{\epsilon,h}_s,e\right)-\gamma_x\left(s,Y_{s-},\lr\left(Y_{s-}\right),u_s,e\right)\right] \frac{\de X^\epsilon_{s-}}{\epsilon}\\
    &\quad\qquad\qquad\qquad\qquad +\left[\gamma_v\left(s,X_{s-}^{\epsilon,h},\lr\left(X_{s-}^{\epsilon,h}\right),v^{\epsilon,h}_s,e\right)-\gamma_v\left(s,Y_{s-},\lr\left(Y_{s-}\right),u_s,e\right)\right] \de v_s \\
    &\quad\qquad\qquad\qquad\qquad + \he\bigg[\bigg(D_y\frac{d\gamma}{d\nu}\left(s,X_{s-}^{\epsilon,h},\lr\left(X_{s-}^{\epsilon,h}\right),v^{\epsilon,h}_s,e\right)\left(\widehat{X_{s-}^{\epsilon,h}}\right)\\
    &\qquad\qquad\qquad\qquad\qquad\qquad - D_y\frac{d\gamma}{d\nu}\left(s,Y_{s-},\lr\left(Y_{s-}\right),u_s,e\right)\left({\widehat{Y_{s-}}}\right)\bigg) \frac{\widehat{\de X_{s-}^\epsilon}}{\epsilon}\bigg] \bigg\} dh.
\end{align*}
Applying a similar approach as the proof of \eqref{lem:control-state-1} for SDE \eqref{SDE:rhoX}, from Assumption (A1), by Gr\"onwall inequality, we have
\begin{align*}
    &\e\left[\sup_{0\le t\le T}\left|\chi^\epsilon_t\right|^2\right]\le C(L,T)\e\left[\int_0^T \left(\left|\alpha^{b,\epsilon}_t\right|^2+\left|\alpha^{\sigma,\epsilon}_t\right|^2 +\int_E \left|\alpha^{\gamma,\epsilon}(t,t-,e) \right|^2 \lambda(de) \right) dt \right].
\end{align*}
Then, from the boundedness estimate \eqref{DeX:boundedness} and the continuity of the derivatives $b_x$, $b_v$, $D_y\frac{db}{d\nu}$, $\sigma_x$, $\sigma_v$, $D_y\frac{d\sigma}{d\nu}$, $\gamma_x$, $\gamma_v$ and $D_y\frac{d\gamma}{d\nu}$ in Assumption (A1), we obtain \eqref{Delta_delta_distance}.

\subsection{Proof of Lemma~\ref{lem:adjoint}}\label{pf:lem:adjoint}

For the existence and uniqueness of the solution of BSDE \eqref{adjoint}, we refer to \cite[Theorem 4.5]{Oksendal-Sulem}; here, we only prove \eqref{adjoint:boundedness}. By taking conditional expectation and using the fact that $t\mapsto\int_t^T Q_s dB_s$ and $t\mapsto\int_t^T\int_E R_s(e) \mathring{N}(de,ds)$ are martingales, we know that
\begin{align*}
    P_t=\e\Bigg[&g_x(Y_T,\lr(Y_T))+\he\left[D_y\frac{dg}{d\nu}\left(\hyT,\lr(Y_T)\right)(Y_T)\right]\\
    &+\int_t^T \bigg\{ b_x\left(s,\theta_s\right)^\top P_s+\sum_{j=1}^n \sigma^j_x \left(s,\theta_s\right)^\top Q^j_s+\int_E \gamma_x\left(s,\theta_s,e\right)^\top R_s(e)\lambda(de)+f_x\left(s,\theta_s\right)\\
    &\qquad\qquad +\he\bigg[\left(D_y\frac{db}{d\nu}\left(s,\widehat{\theta_s}\right)\left(Y_s\right)\right)^\top \hps+\sum_{j=1}^n \left(D_y\frac{d\sigma^j}{d\nu}\left(s,\widehat{\theta_s}\right)\left(Y_s\right)\right)^\top \hqsj \\
    &\qquad\qquad\qquad +\int_E \left(D_y\frac{d\gamma}{d\nu}\left(s,\widehat{\theta_s},e\right)\left(Y_s\right)\right)^\top \hrs(e) \lambda(de)+ D_y\frac{df}{d\nu}\left(s,\widehat{\theta_s}\right)\left(Y_s\right) \bigg]\bigg\}ds \ \Bigg|\  \f_t\Bigg].
\end{align*}
Therefore, by using the Burkholder-Davis-Gundy inequality (see \cite{dellacherie2011probabilities}), Assumptions (A1) and (A2), we have
\begin{align}\label{lem:adjoint_1}
    \e\bigg[\sup_{0\le t\le T}|P_t|^2\bigg]\le C(L,T)\e\left[1+|Y_T|^2+\int_0^T \left(|Y_t|^2+|u_t|^2+|P_t|^2+|Q_t|^2 +\int_E |R_t(e)|^2\lambda(de)\right) dt\right].
\end{align}
By applying It\^o's formula \cite[Theorem 1.16]{Oksendal-Sulem} for $|P_t|^2$ and taking expectation, we have
\begin{align*}
    &\e\left[|P_T|^2-|P_0|^2\right]\\
    =\ &\e\bigg[\int_0^T -2P_t^\top\bigg\{ b_x\left(t,\theta_t\right)^\top P_t+\sum_{j=1}^n \sigma^j_x \left(t,\theta_t\right)^\top Q^j_t+\int_E \gamma_x\left(t,\theta_t,e\right)^\top R_t(e)\lambda(de)+f_x\left(t,\theta_t\right)\\
    &\qquad\qquad\qquad +\he\bigg[\left(D_y\frac{db}{d\nu}\left(t,\widehat{\theta_t}\right)\left(Y_t\right)\right)^\top \widehat{P_t}+\sum_{j=1}^n \left(D_y\frac{d\sigma^j}{d\nu}\left(t,\widehat{\theta_t}\right)\left(Y_t\right)\right)^\top \widehat{Q^j_t} \\
    &\qquad\qquad\qquad +\int_E \left(D_y\frac{d\gamma}{d\nu}\left(t,\widehat{\theta_t},e\right)\left(Y_t\right)\right)^\top \widehat{R_t}(e) \lambda(de)+ D_y\frac{df}{d\nu}\left(t,\widehat{\theta_t}\right)\left(Y_t\right) \bigg]\bigg\} dt \\
    &\quad +\int_0^T \bigg(|Q_t|^2 +\int_E |R_t(e)|^2 \lambda(de) \bigg) dt\bigg].
\end{align*}
Therefore, from Assumptions (A1) and (A2), and using the standard Young's inequality, we have
\begin{align*}
    & \e\bigg[\int_0^T \bigg(|Q_t|^2 +\int_E |R_t(e)|^2 \lambda(de) \bigg) dt\bigg]\\
    \le\ & C(L)\bigg\{1+\|Y_T\|_2^2+\int_0^T \left[\|Y_t\|_2^2+\|u_t\|_2^2+\|P_t\|_2^2+ \|P_t\|_2\cdot\left(\|Q_t\|_2+\|R_t(\cdot)\|_{L^2(\Omega;L_\lambda^2)}\right) \right] dt\bigg\}\\
    \le\ & C(L)\bigg[1+\|Y_T\|_2^2+\int_0^T \left(\|Y_t\|_2^2+\|u_t\|_2^2+2\|P_t\|_2^2+ \frac{1}{2}\|Q_t\|_2+\frac{1}{2}\|R_t(\cdot)\|_{L^2(\Omega;L_\lambda^2)}\right) dt\bigg],
\end{align*}
from which we deduce that
\begin{align}\label{lem:adjoint_2}
    & \e\bigg[\int_0^T \bigg(|Q_t|^2 +\int_E |R_t(e)|^2 \lambda(de) \bigg) dt\bigg]\le C(L)\bigg[1+\|Y_T\|_2^2+\int_0^T \left(\|Y_t\|_2^2+\|u_t\|_2^2+\|P_t\|_2^2\right) dt\bigg].
\end{align}
Substituting \eqref{lem:adjoint_2} into \eqref{lem:adjoint_1}, we have
\begin{align*}
    \e\bigg[\sup_{0\le t\le T}|P_t|^2\bigg]\le C(L,T)\e\left[1+|Y_T|^2+\int_0^T \left(|Y_t|^2+|u_t|^2+|P_t|^2\right) dt\right],
\end{align*}
then, by using Gr\"onwall's inequality, we have
\begin{align}\label{lem:adjoint_3}
    \e\bigg[\sup_{0\le t\le T}|P_t|^2\bigg]\le C(L,T)\e\left[1+|Y_T|^2+\int_0^T \left(|Y_t|^2+|u_t|^2\right) dt\right].
\end{align}
Substituting \eqref{lem:adjoint_3} back into \eqref{lem:adjoint_2}, we deduce that 
\begin{align}\label{lem:adjoint_4}
    & \e\bigg[\int_0^T \bigg(|Q_t|^2 +\int_E |R_t(e)|^2 \lambda(de) \bigg) dt\bigg]\le C(L,T)\e\left[1+|Y_T|^2+\int_0^T \left(|Y_t|^2+|u_t|^2\right) dt\right].
\end{align}
Finally, by substituting Estimates \eqref{lem:control-state-1} into \eqref{lem:adjoint_3} and \eqref{lem:adjoint_4}, we obtain \eqref{adjoint:boundedness}.

\subsection{Proof of Lemma~\ref{lem:dJ}}\label{pf:lem:dJ}

From Assumption (A2) and Lemma~\ref{lem:rhoX}, we can compute that 
\begin{align}
    &\frac{d}{d\epsilon}J(u+\epsilon\de v)\bigg|_{\epsilon=0} \notag \\
    =\ & \frac{d}{d\epsilon}\bigg\{\e\bigg[\int_0^T f\left(t,X^\epsilon_t,\lr(X^\epsilon_t),v^\epsilon_t\right)dt+g\left(X^\epsilon_T,\lr(X^\epsilon_T)\right) \bigg]\bigg\}\bigg|_{\epsilon=0} \notag \\
    =\ & \e\bigg\{\int_0^T \bigg[ f_x(t,\theta_t)^\top \delta X_t+f_v(t,\theta_t)^\top \de v_t+\he\bigg[\left(D_y \frac{df}{d\nu}(t,\theta_t)\left(\widehat{Y_t}\right)\right)^\top \widehat{\delta X_t} \bigg] \bigg] dt \notag \\
    &\qquad + g_x(Y_T,\lr(Y_T))^\top \delta X_T+\he\bigg[\left(D_y \frac{dg}{d\nu}\left(Y_T,\lr(Y_T)\right)\left(\widehat{Y_T}\right)\right)^\top \widehat{\delta X_T} \bigg] \bigg\}. \label{lem:dJ_1}
\end{align}
By applying It\^o's formula \cite[Theorem 1.16]{Oksendal-Sulem} for $P_t^\top \delta X_t$ and taking expectation, we have
\begin{align}
    &\e\left[P_T^\top \delta X_T\right] \notag \\
    =\ & \e \int_0^T \bigg\{ P_s^\top \bigg\{b_x\left(s,\theta_s\right) \delta X_s+b_v\left(s,\theta_s\right) \de v_s +\he\bigg[\left(D_y\frac{db}{d\nu}\left(s,\theta_s\right)\left(\hys\right)\right) \widehat{\delta X_s}\bigg] \bigg\} \notag \\
    &\quad\qquad -\bigg\{ b_x\left(s,\theta_s\right)^\top P_s+\sum_{j=1}^n \sigma^j_x \left(s,\theta_s\right)^\top Q^j_s+\int_E \gamma_x\left(s,\theta_s,e\right)^\top R_s(e)\lambda(de)+f_x\left(s,\theta_s\right) \notag \\
    &\quad\qquad\qquad +\he\bigg[\left(D_y\frac{db}{d\nu}\left(s,\widehat{\theta_s}\right)\left(Y_s\right)\right)^\top \hps+\sum_{j=1}^n \left(D_y\frac{d\sigma^j}{d\nu}\left(s,\widehat{\theta_s}\right)\left(Y_s\right)\right)^\top \hqsj \notag \\
    &\quad\qquad\qquad\qquad +\int_E \left(D_y\frac{d\gamma}{d\nu}\left(s,\widehat{\theta_s},e\right)\left(Y_s\right)\right)^\top \hrs(e) \lambda(de)+ D_y\frac{df}{d\nu}\left(s,\widehat{\theta_s}\right)\left(Y_s\right) \bigg]\bigg\}^\top \delta X_s \notag \\
    &\quad\qquad+\sum_{j=1}^n \left(Q_s^j\right)^\top \bigg\{\sigma^j_x\left(s,\theta_s\right) \delta X_s+\sigma^j_v\left(s,\theta_s\right) \de v_s +\he\bigg[\left(D_y\frac{d\sigma^j}{d\nu}\left(s,\theta_s\right)\left(\hys\right)\right) \widehat{\delta X_s}\bigg] \bigg\} \notag \\
    &\quad\qquad+\int_E \left(R_s(e)\right)^\top \bigg\{\gamma_x\left(s,Y_{s-},\lr\left(Y_{s-}\right),u_s,e\right) \delta X_{s-}+\gamma_v\left(s,Y_{s-},\lr\left(Y_{s-}\right),u_s,e\right) \de v_s \notag \\
    &\qquad\qquad\qquad\qquad\qquad +\he\bigg[\left(D_y\frac{d\gamma}{d\nu}\left(s,Y_{s-},\lr\left(Y_{s-}\right),u_s,e\right)\left(\widehat{Y_{s-}}\right)\right) \widehat{\delta X_{s-}}\bigg] \bigg\}\lambda(de) \bigg\}ds \notag \\
    =\ & \e\int_0^T \bigg[P_s^\top b_v\left(s,\theta_s\right) +\sum_{j=1}^n \left(Q_s^j\right)^\top \sigma^j_v\left(s,\theta_s\right) +\int_E \left(R_s(e)\right)^\top \gamma_v\left(s,\theta_s,e\right) \lambda(de)\bigg] \de v_s \notag \\
    &\quad\qquad - f_x\left(s,\theta_s\right)^\top \delta X_s -\he\bigg[ \left(D_y\frac{df}{d\nu}\left(s,\theta_s\right)\left(\hys\right)\right)^\top \widehat{\delta X_s} \bigg] ds. \label{lem:dJ_2}
\end{align}
Here, we use the fact that $Y,\delta X\in\sr^2_\f(0,T)$ and therefore
\begin{align}
    &\e\int_0^T \int_E \left(R_s(e)\right)^\top \gamma_x\left(s,Y_{s-},\lr\left(Y_{s-}\right),u_s,e\right) \lambda(de) \delta X_{s-} ds \notag \\
    =\ & \e\int_0^T \int_E \left(R_s(e)\right)^\top \gamma_x\left(s,\theta_s,e\right) \lambda(de) \delta X_{s-} ds, \notag \\
    &\e\int_0^T \int_E \left(R_s(e)\right)^\top \gamma_v\left(s,Y_{s-},\lr\left(Y_{s-}\right),u_s,e\right) \lambda(de) \de v_s ds \notag \\
    =\ & \e\int_0^T \int_E \left(R_s(e)\right)^\top \gamma_v\left(s,\theta_s,e\right) \lambda(de) \de v_s ds, \notag \\
    &\e\int_0^T \int_E \left(R_s(e)\right)^\top \he\bigg[\left(D_y\frac{d\gamma}{d\nu}\left(s,Y_{s-},\lr\left(Y_{s-}\right),u_s,e\right)\left(\widehat{Y_{s-}}\right)\right) \widehat{\delta X_{s-}}\bigg] \lambda(de) ds \notag \\
    =\ & \e\int_0^T \int_E \left(R_s(e)\right)^\top \he\bigg[\left(D_y\frac{d\gamma}{d\nu}\left(s,\theta_s,e\right)\left(\hys\right)\right) \widehat{\delta X_{s}}\bigg] \lambda(de) ds; \label{fact:llrc}
\end{align}
and also use the following equalities due to the Fubini's theorem:
\begin{align}
    &\e\bigg\{ P_s^\top \he\bigg[\left(D_y\frac{db}{d\nu}\left(s,\theta_s\right)\left(\hys\right)\right) \widehat{\delta X_s}\bigg] \bigg\}=\e\bigg\{\he\bigg[\hps^\top \left(D_y\frac{db}{d\nu}\left(s,\widehat{\theta_s}\right)\left(Y_s\right)\right) \bigg] \delta X_t\bigg\}, \notag \\
    &\e\bigg\{\left(Q_s^j\right)^\top \he\bigg[\left(D_y\frac{d\sigma^j}{d\nu}\left(s,\theta_s\right)\left(\hys\right)\right) \widehat{\delta X_s}\bigg] \bigg\}=\e\bigg\{\he\bigg[\hqsj^\top\left(D_y\frac{d\sigma^j}{d\nu}\left(s,\widehat{\theta_s}\right)\left(Y_s\right)\right)\bigg] \delta X_t\bigg\}, \notag \\
    &\e \int_E \left(R_s(e)\right)^\top \he\bigg[\left(D_y\frac{d\gamma}{d\nu}\left(s,\theta_s,e\right)\left(\hys\right)\right) \widehat{\delta X_{s}}\bigg] \lambda(de) \notag\\
    =\ & \e\bigg\{\he\bigg[\int_E \left(\hrs(e)\right)^\top \left(D_y\frac{d\gamma}{d\nu}\left(s,\widehat{\theta_s},e\right)\left(Y_s\right)\right) \lambda(de) \bigg] \delta X_s\bigg\}, \notag \\
    &\e\he\bigg[ \left(D_y\frac{df}{d\nu}\left(s,\widehat{\theta_s}\right)\left(Y_s\right)\right)^\top \delta X_s \bigg]=\e\he\bigg[ \left(D_y\frac{df}{d\nu}\left(s,\theta_s\right)\left(\hys\right)\right)^\top \widehat{\delta X_s} \bigg]. \label{fact:Fubini}
\end{align}
From the terminal condition for $P_T$ and again the Fubini's theorem, we have 
\begin{align*}
    \e\left[P_T^\top \delta X_T\right]=\ & \e\bigg\{g_x(Y_T,\lr(Y_T))^\top \delta X_T+\he\bigg[\left(D_y\frac{dg}{d\nu}\left(\hyT,\lr(Y_T)\right)(Y_T)\right)^\top \delta X_T\bigg] \bigg\}\\
    =\ & \e\bigg\{g_x(Y_T,\lr(Y_T))^\top \delta X_T+\he\bigg[\left(D_y\frac{dg}{d\nu}\left(Y_T,\lr(Y_T)\right)\left(\hyT\right)\right)^\top \widehat{\delta X_T}\bigg] \bigg\}
\end{align*}
From the last equality and \eqref{lem:dJ_2}, we have
\begin{align}
    &\e\bigg\{g_x(Y_T,\lr(Y_T))^\top \delta X_T+\he\bigg[\left(D_y\frac{dg}{d\nu}\left(Y_T,\lr(Y_T)\right)\left(\hyT\right)\right)^\top \widehat{\delta X_T}\bigg] \notag \\
    &\qquad + \int_0^T \bigg[f_x\left(s,\theta_s\right)^\top \delta X_s +\he\bigg[ \left(D_y\frac{df}{d\nu}\left(s,\theta_s\right)\left(\hys\right)\right)^\top \widehat{\delta X_s} \bigg] \bigg] ds \bigg\} \notag \\
    =\ &\e\int_0^T \bigg[P_s^\top b_v\left(s,\theta_s\right) +\sum_{j=1}^n \left(Q_s^j\right)^\top \sigma^j_v\left(s,\theta_s\right) +\int_E \left(R_s(e)\right)^\top \gamma_v\left(s,\theta_s,e\right) \lambda(de)\bigg] \de v_s ds. \label{lem:dJ_3}
\end{align}
Combining \eqref{lem:dJ_3} and \eqref{lem:dJ_1}, we obtain \eqref{lem:dJ_0}.

\subsection{Proof of Theorem~\ref{thm:suff}}\label{pf:thm:suff}

For any control $v\in\mr_\f^2(0,T)$, we denote by $X^v$ the associated controlled state process. Then, the process $X^v_t-Y_t$ satisfies the following SDE:
\begin{align*}
    X_t^v-Y_t=\ &\int_0^t \int_0^1 \bigg\{ b_x\left(s,\theta^{h}_s\right) \left(X_s^v-Y_s\right)+b_v\left(s,\theta^{h}_s\right) \left(v_s-u_s\right) \\
    &\qquad\qquad + \he\bigg[\left(D_y\frac{db}{d\nu}\left(s,\theta^{h}_s\right)\left(\widehat{X_s^{h}}\right)\right) \left(\widehat{X_s^v}-\widehat{Y_s}\right) \bigg] \bigg\} dh ds \\
    &+\int_0^t \int_0^1 \bigg\{ \sigma_x\left(s,\theta^{h}_s\right) \left(X_s^v-Y_s\right)+\sigma_v\left(s,\theta^{h}_s\right) \left(v_s-u_s\right) \\
    &\qquad\qquad + \he\bigg[\left(D_y\frac{d\sigma}{d\nu}\left(s,\theta^{h}_s\right)\left(\widehat{X_s^{h}}\right)\right) \left(\widehat{X_s^v}-\widehat{Y_s}\right) \bigg] \bigg\} dhdB_s\\
    &+\int_0^t\int_E \int_0^1 \bigg\{ \gamma_x\left(s,\theta_{s-}^{h},e\right) \left(v_{s-}-u_{s-}\right)+\gamma_v\left(s,\theta_{s-}^{h},e\right) \left(v_s-u_s\right) \\
    &\qquad\qquad\qquad + \he\bigg[\left(D_y\frac{d\gamma}{d\nu}\left(s,\theta_{s-}^{h},e\right)\left(\widehat{X_{s-}^{h}}\right)\right) \left(\widehat{X_{s-}^v}-\widehat{Y_{s-}}\right) \bigg] \bigg\} dh \mathring{N}(de,ds),
\end{align*}
where 
\begin{align*}
    \theta^{h}_s:=\left(X_s^{h},\lr\left(X_s^{h}\right),v^{h}_s\right),\quad X^{h}_s:=Y_s+h\left(X_s^v-Y_s\right),\quad v^{h}_s:=u_s+h\left(v_s-u_s\right).
\end{align*}
Then, by applying It\^o's formula \cite[Theorem 1.16]{Oksendal-Sulem} for $P_t^\top \left(X^v_s-Y_s\right)$ and taking expectation, and using the Fubini's theorem (similar as in \eqref{fact:Fubini}) and the fact that $Y,X^v\in\sr^2_\f(0,T)$ (similar as in \eqref{fact:llrc}), we have
\begin{align*}
    &\e\left[P_T^\top \left(X^v_T-Y_T\right)\right]\\
    =\ & \e \int_0^T \Bigg\{ P_s^\top \int_0^1 \bigg\{ \left[b_x\left(s,\theta^{h}_s\right)-b_x\left(s,\theta_s\right)\right] \left(X_s^v-Y_s\right)+b_v\left(s,\theta^{h}_s\right) \left(v_s-u_s\right) \\
    &\qquad\qquad\qquad + \he\bigg[\left(D_y\frac{db}{d\nu}\left(s,\theta^{h}_s\right)\left(\widehat{X_s^{h}}\right)-D_y\frac{db}{d\nu}\left(s,\theta_s\right)\left(\hys\right)\right) \left(\widehat{X_s^v}-\widehat{Y_s}\right) \bigg] \bigg\} dh \notag \\
    &\qquad+\sum_{j=1}^n \left(Q_s^j\right)^\top \int_0^1 \bigg\{ \left[\sigma^j_x\left(s,\theta^{h}_s\right)-\sigma^j_x \left(s,\theta_s\right)\right] \left(X_s^v-Y_s\right)+\sigma^j_v\left(s,\theta^{h}_s\right) \left(v_s-u_s\right) \\
    &\qquad\qquad\qquad\qquad + \he\bigg[\left(D_y\frac{d\sigma^j}{d\nu}\left(s,\theta^{h}_s\right)\left(\widehat{X_s^{h}}\right)-D_y\frac{d\sigma^j}{d\nu}\left(s,\theta_s\right)\left(\hys\right)\right) \left(\widehat{X_s^v}-\widehat{Y_s}\right) \bigg] \bigg\} dh \notag \\
    &\qquad+\int_E \left(R_s(e)\right)^\top \int_0^1 \bigg\{ \left[\gamma_x\left(s,\theta_s^{h},e\right)- \gamma_x\left(s,\theta_s,e\right) \right] \left(X^v_s-Y_s\right) \\
    &\qquad\qquad\qquad\qquad + \he\bigg[\bigg(D_y\frac{d\gamma}{d\nu}\left(s,\theta_s^{h},e\right)\left(\widehat{X_s^{h}}\right) -D_y\frac{d\gamma}{d\nu}\left(s,\theta_s,e\right)\left(\hys\right)\bigg) \left(\widehat{X_s^v}-\widehat{Y_s}\right) \bigg] \bigg\} dh \lambda(de) \\
    &\qquad -f_x\left(s,\theta_s\right)^\top \left(X^v_s-Y_s\right)-\he\bigg[\left(D_y\frac{df}{d\nu}\left(s,\theta_s\right)\left(\hys\right)\right)^\top\left(\widehat{X_s^v}-\widehat{Y_s}\right) \bigg] \Bigg\}ds.
\end{align*}
Substituting the optimality conditions \eqref{optimal_condition_split} into the last equality, noting the fact that for the terms corresponding to $\gamma$ and $\sigma^j$ for $1\le j\le n$ with $d_j=0$ are all cancelled out due to their linear condition in Assumption (A3)-(ii,iii), we have
\begin{align}
    &\e\left[P_T^\top \left(X^v_T-Y_T\right)\right] \notag \\
    =\ & \e \int_0^T \Bigg\{ P_s^\top \int_0^1 \bigg\{ \left[B_x\left(s,X^{h}_s,\lr\left(X^{h}_s\right),v^{h,0}_s\right)-B_x\left(s,Y_s,\lr(Y_s),u^0_s\right)\right] \left(X_s^v-Y_s\right) \notag \\
    &\quad\qquad\qquad\qquad +\left[B_{v^0}\left(s,X^{h}_s,\lr\left(X^{h}_s\right),v^{h,0}_s\right)-B_{v^0}\left(s,Y_s,\lr(Y_s),u^0_s\right)\right] \left(v^0_s-u^0_s\right) \notag\\
    &\quad\qquad\qquad\qquad + \he\bigg[\bigg(D_y\frac{dB}{d\nu}\left(s,X^{h}_s,\lr\left(X^{h}_s\right),v^{h,0}_s\right)\left(\widehat{X_s^{h}}\right) \notag\\
    &\qquad\qquad\qquad\qquad\qquad -D_y\frac{dB}{d\nu}\left(s,Y_s,\lr(Y_s),u_s\right)\left(\hys\right)\bigg) \left(\widehat{X_s^v}-\widehat{Y_s}\right) \bigg] \bigg\} dh \notag \\
    &\qquad+\sum_{1\le j\le n,\ d_j>0} \left(Q_s^j\right)^\top \int_0^1 \bigg\{ \left[A^j_x\left(s,X^{h}_s,\lr\left(X^{h}_s\right),v^{h,j}_s\right)-A^j_x \left(s,Y_s,\lr(Y_s),u^j_s\right)\right] \left(X_s^v-Y_s\right) \notag\\
    &\quad\qquad\qquad\qquad +\left[A^j_{v^j}\left(s,X^{h}_s,\lr\left(X^{h}_s\right),v^{h,j}_s\right)-A^j_{v^j}\left(s,Y_s,\lr\left(Y_s\right),u^j_s\right) \right] \left(v^j_s-u^j_s\right) \notag\\
    &\quad\qquad\qquad\qquad + \he\bigg[\bigg(D_y\frac{dA^j}{d\nu}\left(s,X^{h}_s,\lr\left(X^{h}_s\right),v^{h,j}_s\right)\left(\widehat{X_s^{h}}\right) \notag\\
    &\qquad\qquad\qquad\qquad\qquad -D_y\frac{dA^j}{d\nu}\left(s,Y_s,\lr(Y_s),u^j_s\right)\left(\hys\right)\bigg) \left(\widehat{X_s^v}-\widehat{Y_s}\right) \bigg] \bigg\} dh \notag \\
    &\qquad -f_x\left(s,\theta_s\right)^\top \left(X^v_s-Y_s\right)-\he\bigg[\left(D_y\frac{df}{d\nu}\left(s,\theta_s\right)\left(\hys\right)\right)^\top\left(\widehat{X_s^v}-\widehat{Y_s}\right) \bigg] \notag\\
    &\qquad- f^0_{v^0}\left(s,Y_s,\lr(Y_s),u^0_s\right)^\top \left(v^0_s-u^0_s\right)-\sum_{1\le j\le n,\ d_j>0} f^j_{v^j}\left(s,Y_s,\lr(Y_s),u^j_s\right)^\top \left(v^j_s-u^j_s\right) \Bigg\}ds. \label{thm:suff_1}
\end{align}
From Assumptions (A3)-(iv) and (A4), we know that 
\begin{align}
    &\e\int_0^T \bigg\{ f_x\left(s,\theta_s\right)^\top \left(X^v_s-Y_s\right)+\he\bigg[\left(D_y\frac{df}{d\nu}\left(s,\theta_s\right)\left(\hys\right)\right)^\top\left(\widehat{X_s^v}-\widehat{Y_s}\right) \bigg] \notag\\
    &\qquad\quad + f^0_{v^0}\left(s,Y_s,\lr(Y_s),u^0_s\right)^\top \left(v^0_s-u^0_s\right)+\sum_{1\le j\le n,\ d_j>0} f^j_{v^j}\left(s,Y_s,\lr(Y_s),u^j_s\right)^\top \left(v^j_s-u^j_s\right)\bigg\} ds \notag\\
    =\ & \e\int_0^T \bigg\{ f_x\left(s,\theta_s\right)^\top \left(X^v_s-Y_s\right)+f_v\left(s,\theta_s\right)^\top \left(v_s-u_s\right) + \he\bigg[\left(D_y\frac{df}{d\nu}\left(s,\theta_s\right)\left(\hys\right)\right)^\top\left(\widehat{X_s^v}-\widehat{Y_s}\right) \bigg]\bigg\} ds \notag \\
    \le\ & \e\int_0^T \left[f\left(s,X^v_s,\lr(X^v_s),v_s\right)-f(s,\theta_s)- \lambda_v|v_s-u_s|^2-\lambda_x\left|X^v_s-Y_s\right|^2-\lambda_m\left\|X^v_s-Y_s\right\|^2_2\right] ds. \label{thm:suff_2}
\end{align}
Similarly, from the terminal condition for $P_T$ and the Fubini's theorem, and Assumption (A4), we can compute that 
\begin{align}
    &\e\left[P_T^\top \left(X^v_T-Y_T\right)\right] \notag\\
    =\ & \e \left\{g_x(Y_T,\lr(Y_T))^\top \left(X^v_T-Y_T\right)+\he\left[D_y\frac{dg}{d\nu}\left(\hyT,\lr(Y_T)\right)(Y_T) \right]^\top \left(X^v_T-Y_T\right)\right\} \notag\\
    =\ & \e \left\{g_x(Y_T,\lr(Y_T))^\top \left(X^v_T-Y_T\right)+\he\left[\left(D_y\frac{dg}{d\nu}\left(Y_T,\lr(Y_T)\right)\left(\hyT\right)\right)^\top \left(\widehat{X^v_T}-\hyT\right) \right]\right\} \notag\\
    \le\ & \e \left[g\left(X^v_T,\lr(X^v_T)\right)-g(Y_T,\lr(Y_T)) \right]. \label{thm:suff_3}
\end{align}
Combining \eqref{thm:suff_1}, \eqref{thm:suff_2} and \eqref{thm:suff_3}, we have 
\begin{align}
    &J(v)-J(u) \notag \\
    =\ & \e\left[\int_0^T \left(f\left(s,X^v_s,\lr(X^v_s),v_s\right)-f(s,\theta_s)\right) ds + g\left(X^v_T,\lr(X^v_T)\right)-g(Y_T,\lr(Y_T)) \right]\notag \\
    \geq\ & \e\int_0^T \bigg\{ f_x\left(s,\theta_s\right)^\top \left(X^v_s-Y_s\right)+\he\bigg[\left(D_y\frac{df}{d\nu}\left(s,\theta_s\right)\left(\hys\right)\right)^\top\left(\widehat{X_s^v}-\widehat{Y_s}\right) \bigg] \notag\\
    &\qquad\quad + f^0_{v^0}\left(s,Y_s,\lr(Y_s),u^0_s\right)^\top \left(v^0_s-u^0_s\right)+\sum_{1\le j\le n,\ d_j>0} f^j_{v^j}\left(s,Y_s,\lr(Y_s),u^j_s\right)^\top \left(v^j_s-u^j_s\right) \notag \\
    &\qquad\quad + \lambda_v|v_s-u_s|^2+\lambda_x\left|X^v_s-Y_s\right|^2+\lambda_m\left\|X^v_s-Y_s\right\|^2_2\bigg\} ds + \e\left[P_T^\top \left(X^v_T-Y_T\right)\right] \notag \\
    =\ & \e \int_0^T \Bigg\{ P_s^\top \int_0^1 \bigg\{ \left[B_x\left(s,X^{h}_s,\lr\left(X^{h}_s\right),v^{h,0}_s\right)-B_x\left(s,Y_s,\lr(Y_s),u^0_s\right)\right] \left(X_s^v-Y_s\right) \notag \\
    &\quad\qquad\qquad\qquad +\left[B_{v^0}\left(s,X^{h}_s,\lr\left(X^{h}_s\right),v^{h,0}_s\right)-B_{v^0}\left(s,Y_s,\lr(Y_s),u^0_s\right)\right] \left(v^0_s-u^0_s\right) \notag\\
    &\quad\qquad\qquad\qquad + \he\bigg[\bigg(D_y\frac{dB}{d\nu}\left(s,X^{h}_s,\lr\left(X^{h}_s\right),v^{h,0}_s\right)\left(\widehat{X_s^{h}}\right) \notag\\
    &\qquad\qquad\qquad\qquad\qquad -D_y\frac{dB}{d\nu}\left(s,Y_s,\lr(Y_s),u_s\right)\left(\hys\right)\bigg) \left(\widehat{X_s^v}-\widehat{Y_s}\right) \bigg] \bigg\} dh \notag \\
    &\qquad+\sum_{1\le j\le n,\ d_j>0} \left(Q_s^j\right)^\top \int_0^1 \bigg\{ \left[A^j_x\left(s,X^{h}_s,\lr\left(X^{h}_s\right),v^{h,j}_s\right)-A^j_x \left(s,Y_s,\lr(Y_s),u^j_s\right)\right] \left(X_s^v-Y_s\right) \notag\\
    &\quad\qquad\qquad\qquad +\left[A^j_{v^j}\left(s,X^{h}_s,\lr\left(X^{h}_s\right),v^{h,j}_s\right)-A^j_{v^j}\left(s,Y_s,\lr\left(Y_s\right),u^j_s\right) \right] \left(v^j_s-u^j_s\right) \notag\\
    &\quad\qquad\qquad\qquad + \he\bigg[\bigg(D_y\frac{dA^j}{d\nu}\left(s,X^{h}_s,\lr\left(X^{h}_s\right),v^{h,j}_s\right)\left(\widehat{X_s^{h}}\right) \notag \\
    &\qquad\qquad\qquad\qquad\qquad -D_y\frac{dA^j}{d\nu}\left(s,Y_s,\lr(Y_s),u^j_s\right)\left(\hys\right)\bigg) \left(\widehat{X_s^v}-\widehat{Y_s}\right) \bigg] \bigg\} dh \notag\\
    &\qquad + \lambda_v|v_s-u_s|^2+\lambda_x\left|X^v_s-Y_s\right|^2+\lambda_m\left\|X^v_s-Y_s\right\|^2_2 \Bigg\}ds, \label{thm:suff_4}
\end{align}
where the first inequality uses \eqref{thm:suff_2} and \eqref{thm:suff_3}, and the last equality uses \eqref{thm:suff_1}. We now give the boundedness for the second term of the right hand side of \eqref{thm:suff_4} (and that for the first term can be obtained similarly). From Condition \eqref{generic:condition:A} in Assumption (A3)-(ii) and the cone property \eqref{cone:PQ:without_v}, we know that for $1\le j\le n$ with $d_j>0$, 
\begin{align*}
    &\e\int_0^T \bigg|\left[A^j_x\left(s,X^{h}_s,\lr\left(X^{h}_s\right),v^{h,j}_s\right)-A^j_x \left(s,Y_s,\lr(Y_s),u^j_s\right)\right] \left(X_s^v-Y_s\right) \notag\\    
    &\quad\qquad +\left[A^j_{v^j}\left(s,X^{h}_s,\lr\left(X^{h}_s\right),v^{h,j}_s\right)-A^j_{v^j}\left(s,Y_s,\lr\left(Y_s\right),u^j_s\right) \right] \left(v^j_s-u^j_s\right) \notag\\
    &\quad\qquad + \he\bigg[\bigg(D_y\frac{dA^j}{d\nu}\left(s,X^{h}_s,\lr\left(X^{h}_s\right),v^{h,j}_s\right)\left(\widehat{X_s^{h}}\right)\\
    &\qquad\qquad\qquad -D_y\frac{dA^j}{d\nu}\left(s,Y_s,\lr(Y_s),u^j_s\right)\left(\hys\right)\bigg) \left(\widehat{X_s^v}-\widehat{Y_s}\right) \bigg] \bigg| \cdot \left|Q_s^j\right| ds\\
    \le\ &  \e\int_0^T \bigg[\frac{L_0 h\left|X^v_s-Y_s\right|+L_0 h\left\|X^v_s-Y_s\right\|_2+L_1 h \left|v^j_s-u^j_s\right|}{1+|Y_s|+|\lr(Y_s)|_1+ \left|u^j_s\right|}\cdot \left|X^v_s-Y_s\right|\\
    & \quad\qquad +\frac{L_0 h\left|X^v_s-Y_s\right|+2L_0 h\left\|X^v_s-Y_s\right\|_2+L_1 h \left|v^j_s-u^j_s\right|}{1+|Y_s|+|\lr(Y_s)|_1 + \left|u^j_s\right|}\cdot \left\|X^v_s-Y_s\right\|_2\\
    & \quad\qquad +\frac{L_1 h\left|X^v_s-Y_s\right|+L_1 h\left\|X^v_s-Y_s\right\|_2+L_2 h \left|v^j_s-u^j_s\right|}{1+|Y_s|+ |\lr(Y_s)|_1 + \left|u^j_s\right|}\cdot \left|v^j_s-u^j_s\right|]\bigg]\\
    & \qquad \cdot \frac{L^2}{\lambda_0}\left[1+\left|Y_s\right|+ \left|\lr\left(Y_s\right)\right|_1+ \left|u^j_s\right| \right] ds \\
    =\ &  \frac{L^2h}{\lambda_0}\int_0^T \bigg[5L_0 \left\|X^v_s-Y_s\right\|_2^2+ L_2 \left\|v^j_s-u^j_s\right\|_2^2 + 4L_1 \left\|X^v_s-Y_s\right\|_2\cdot \left\|v^j_s-u^j_s\right\|_2\bigg]ds.
\end{align*}
Substituting the last inequality (and the similar estimate for the $B$ part) back into \eqref{thm:suff_4} and using the Condition \eqref{thm:suff:condition}-(i,ii), we have
\begin{align*}
    &J(v)-J(u) \notag \\
    \geq\ & \e \int_0^T \bigg\{\lambda_v\left|v_s-u_s\right|^2+ \lambda_x\left|X^v_s-Y_s\right|^2+\lambda_m\left\|X^v_s-Y_s\right\|^2_2 - \frac{5L^2L_0}{2\lambda_0} \left\|X^v_s-Y_s\right\|_2^2 \notag \\
    &\quad\qquad - \frac{2L^2L_1}{\lambda_0} \left\|X^v_s-Y_s\right\|_2\cdot \left\|v^0_s-u^0_s\right\|_2-  \frac{L^2L_2}{2\lambda_0} \left\|v^0_s-u^0_s\right\|_2^2 \notag \\
    &\quad\qquad -\sum_{1\le j\le n,\ d_j>0} \bigg[  \frac{5L^2L_0}{2\lambda_0} \left\|X^v_s-Y_s\right\|_2^2 + \frac{L^2L_2}{2\lambda_0} \left|v^j_s-u^j_s\right|^2 \\
    &\qquad\qquad\qquad\qquad\qquad +  \frac{2L^2L_1}{\lambda_0} \left\|X^v_s-Y_s\right\|_2 \cdot \left|v^j_s-u^j_s\right| \bigg]\bigg\}ds \notag \\
    \geq\ & \int_0^T \bigg[ \left(\lambda_x+\lambda_m-  \frac{5(l+1)L^2L_0}{2\lambda_0} \right)\left\|X^v_s-Y_s\right\|_2^2 + \left(\lambda_v-\frac{L^2L_2}{2\lambda_0} \right) \left\|v_s-u_s\right\|_2^2 \notag \\
    &\quad\qquad - \frac{2\sqrt{l+1}L^2L_1}{\lambda_0} \left\|X^v_s-Y_s\right\|_2\cdot \left\|v_s-u_s\right\|_2 \bigg] ds \notag \\
    =\ & \int_0^T \Bigg[\Bigg(\sqrt{ \lambda_x+\lambda_m- \frac{5(l+1)L^2L_0}{2\lambda_0}} \left\|X^v_s-Y_s\right\|_2\\
    &\qquad\qquad - \frac{ \sqrt{l+1}L^2L_1}{\sqrt{ \lambda_0\left[\lambda_x\lambda_0+\lambda_m\lambda_0- \frac{5(l+1)L^2L_0}{2}\right]}}\left\|v_s-u_s\right\|_2\Bigg)^2 \notag \\
    &\quad\qquad + \left( \lambda_v- \frac{L^2L_2}{2\lambda_0}-\frac{2(l+1)L64L_1^2}{\lambda_0\left[2\lambda_x\lambda_0+2\lambda_m\lambda_0- 5(l+1)L^2L_0\right]}\right) \left\|v_s-u_s\right\|_2^2 \Bigg] ds \notag \\
    \geq\ & \left( \lambda_v- \frac{L^2L_2}{2\lambda_0}-\frac{2(l+1)L^4 L_1^2}{\lambda_0\left[2\lambda_x\lambda_0+2\lambda_m\lambda_0- 5(l+1)L^2L_0\right]} \right) \int_0^T \left\|v_s-u_s\right\|_2^2 dt,
\end{align*}
from which we obtain \eqref{thm:suff_0}.

\section{Proof of Theorem~\ref{thm:Gateaux}}\label{pf:thm:Gateaux}

The proof of the well-posedness of FBSDEs with jumps \eqref{FBSDE:Jaco} is similar to that for the well-posedness of FBSDEs with jumps \eqref{FBSDE:MFTC} (and is also similar to the proofs for the Jacobian flows and Hessian flows in our previous work \cite{AB11,AB10,AB12,ABmfg1,AB5}), which is omitted. Here, we only give the proof of the estimate \eqref{thm:Gateaux_boundedness}; and \eqref{thm:Gateaux_appro} and the continuity of $\dr_\eta\Theta$ can be proven in a similar way. 

By applying It\^o's formula for $\left(\dr_\eta P_t\right)^\top \dr_\eta Y_t$ and taking expectation, and using the Fubini's theorem (similar as in \eqref{fact:Fubini}) and the fact that $Y,\dr_\eta Y\in\sr^2_\f(0,T)$ (similar as in \eqref{fact:llrc}), and also using the third equation of FBSDEs \eqref{FBSDE:Jaco}, we have
\begin{align}
    &\e\left[\left(\dr_\eta P_T\right)^\top \dr_\eta Y_T-\left(\dr_\eta P_0\right)^\top \eta\right] \notag \\
    =\ & \e\int_0^T\Bigg\{ \left(\dr_\eta P_s\right)^\top  b_v\left(s,\theta_s\right)\dr_\eta u_s + \sum_{j=1}^n \left(\dr_\eta Q^j_s\right)^\top \sigma^j_v\left(s,\theta_s\right)\dr_\eta u_s \notag \\
    &\quad\qquad - \bigg\{ \left[Db_{x}\left(s,\theta_s\right)\left(\dr_\eta Y_s,\dr_\eta Y_s,\dr_\eta u_s\right)\right]^\top P_s + \sum_{j=1}^n \bigg[D\sigma^j_{x} \left(s,\theta_s\right) \left(\dr_\eta Y_s,\dr_\eta Y_s,\dr_\eta u_s\right)\bigg]^\top Q^j_s \notag \\
    &\quad\qquad\qquad +\he\bigg\{\bigg[D\left(D_y\frac{db_x}{d\nu}\right)\left(s,\widehat{\theta_s}\right)(Y_s)\left(\widehat{\dr_\eta Y_s},\dr_\eta Y_s,\widehat{\dr_\eta u_s}\right)\bigg]^\top \hps  \notag \\
    &\quad\qquad\qquad\qquad +\sum_{j=1}^n \bigg[D\bigg(D_y\frac{d\sigma^j_x}{d\nu}\bigg)\left(s,\widehat{\theta_s}\right)\left(Y_s\right)\left(\widehat{\dr_\eta Y_s},\dr_\eta Y_s,\widehat{\dr_\eta u_s}\right)\bigg]^\top \hqsj \bigg\}  \notag \\
    &\quad\qquad\qquad +Df_{x}\left(s,\theta_s\right)\left(\dr_\eta Y_s,\dr_\eta Y_s,\dr_\eta u_s\right) \notag\\
    &\quad\qquad\qquad  + \he\bigg[D\left(D_y\frac{df}{d\nu}\right)\left(s,\widehat{\theta_s}\right)\left(Y_s\right)\left(\widehat{\dr_\eta Y_s},\dr_\eta Y_s,\widehat{\dr_\eta u_s}\right)\bigg]\bigg\}^\top \dr_\eta Y_s \Bigg\}ds \notag \\
    =\ & -\e\int_0^T\Bigg\{  P_s^\top \Big\{ \left[Db_{v}\left(s,\theta_s\right)\left(\dr_\eta Y_s,\dr_\eta Y_s,\dr_\eta u_s\right)\right] \dr_\eta u_s + \left[Db_{x}\left(s,\theta_s\right)\left(\dr_\eta Y_s,\dr_\eta Y_s,\dr_\eta u_s\right)\right] \dr_\eta Y_s \Big\}  \notag \\
    &\qquad\qquad  + \sum_{j=1}^n \left(Q^j_s\right)^\top \bigg\{\bigg[D\sigma^j_{v} \left(s,\theta_s\right) \left(\dr_\eta Y_s,\dr_\eta Y_s,\dr_\eta u_s\right)\bigg] \dr_\eta u_s \notag \\
    &\qquad\qquad\qquad\qquad\qquad + \bigg[D\sigma^j_{x} \left(s,\theta_s\right) \left(\dr_\eta Y_s,\dr_\eta Y_s,\dr_\eta u_s\right)\bigg] \dr_\eta Y_s \bigg\} \notag \\
    &\qquad\qquad + \he\bigg\{P_s^\top\bigg[   D\left(D_y\frac{db_x}{d\nu}\right)\left(s,\theta_s\right)\left(\hys\right)\left(\dr_\eta Y_s,\widehat{\dr_\eta Y_s},\dr_\eta u_s\right) \bigg]\widehat{\dr_\eta Y_s}  \notag \\
    &\qquad\qquad\qquad +\sum_{j=1}^n \left(Q^j_s\right)^\top\bigg[D\bigg(D_y\frac{d\sigma^j_x}{d\nu}\bigg)\left(s,\theta_s\right)\left(\hys\right)\left(\dr_\eta Y_s,\widehat{\dr_\eta Y_s},\dr_\eta u_s\right)\bigg] \widehat{\dr_\eta Y_s} \bigg\}  \notag \\
    &\qquad\qquad  + \de f\left(s,\theta_s\right)\left(\dr_\eta Y_s,\dr_\eta Y_s,\dr_\eta u_s\right) \Bigg\}ds; \label{thm:Gateaux_1}
\end{align}
here, we use the notations defined in \eqref{def:b_1}, \eqref{def:b_2} and \eqref{def:convex_f}. Similarly, by using the Fubini's theorem and notation in \eqref{def:g_1}, we have
\begin{align}
    &\e\left[\left(\dr_\eta P_T\right)^\top \dr_\eta Y_T \right]=\e\left[\left(\dr_\eta Y_T\right)^\top \dr_\eta P_T \right] \notag \\
    =\ & \e\bigg\{\left(\dr_\eta Y_T\right)^\top Dg_{x}(Y_T,\lr(Y_T)) \left( \dr_\eta Y_T,\dr_\eta Y_T\right)\notag\\
    &\quad +\he\bigg[\left(\dr_\eta Y_T\right)^\top D\left(D_y\frac{dg}{d\nu}\right)\left(\hyT,\lr(Y_T)\right)(Y_T) \left(\widehat{\dr_\eta Y_T},\dr_\eta Y_T\right)\bigg]\bigg\} \notag \\
    =\ & \e\bigg\{\left(\dr_\eta Y_T\right)^\top Dg_{x}(Y_T,\lr(Y_T)) \left( \dr_\eta Y_T,\dr_\eta Y_T\right) \notag\\
    &\quad +\he\bigg[\left(\widehat{\dr_\eta Y_T}\right)^\top D\left(D_y\frac{dg}{d\nu}\right)\left(Y_T,\lr(Y_T)\right)\left(\hyT\right) \left(\dr_\eta Y_T,\widehat{\dr_\eta Y_T}\right)\bigg]\bigg\} \notag \\
    =\ & \e\left[\de g (Y_T,\lr(Y_T)) \left( \dr_\eta Y_T,\dr_\eta Y_T\right)\right]. \label{thm:Gateaux_2}
\end{align}
From \eqref{thm:Gateaux_1}, \eqref{thm:Gateaux_2} and the convexity conditions \eqref{convex_f} and \eqref{convex_g}, we have
\begin{align}
    & \int_0^T 2\lambda_v\left\|\dr_\eta u_s\right\|_2^2+2(\lambda_x+\lambda_m)\left\|\dr_\eta Y_s\right\|_2^2 ds \notag \\
    \le\ & \e\Bigg\{\int_0^T |P_s|\cdot \bigg\{ \left|Db_{v}\left(s,\theta_s\right)\left(\dr_\eta Y_s,\dr_\eta Y_s,\dr_\eta u_s\right)\right|\cdot \left|\dr_\eta u^0_s \right| \notag \\
    &\qquad\qquad\qquad + \left|Db_{x}\left(s,\theta_s\right)\left(\dr_\eta Y_s,\dr_\eta Y_s,\dr_\eta u_s\right)\right|\cdot \left|\dr_\eta Y_s\right| \notag \\
    &\qquad\qquad\qquad +\he\bigg[   D\left(D_y\frac{db_x}{d\nu}\right)\left(s,\theta_s\right)\left(\hys\right)\left(\dr_\eta Y_s,\widehat{\dr_\eta Y_s},\dr_\eta u_s\right) \widehat{\dr_\eta Y_s} \bigg] \bigg\}  \notag  \notag \\
    &\quad\qquad  + \sum_{j,\ d_j>0} \left|Q^j_s\right|\cdot \bigg\{\bigg|D\sigma^j_{v} \left(s,\theta_s\right) \left(\dr_\eta Y_s,\dr_\eta Y_s,\dr_\eta u_s\right)\bigg|\cdot \left|\dr_\eta u^j_s\right| \notag \\
    &\qquad\qquad\qquad\qquad\qquad + \bigg|D\sigma^j_{x} \left(s,\theta_s\right) \left(\dr_\eta Y_s,\dr_\eta Y_s,\dr_\eta u_s\right)\bigg|\cdot \left|\dr_\eta Y_s\right| \notag \\
    &\qquad\qquad\qquad\qquad\qquad +\he\bigg[D\bigg(D_y\frac{d\sigma^j_x}{d\nu}\bigg)\left(s,\theta_s\right)\left(\hys\right)\left(\dr_\eta Y_s,\widehat{\dr_\eta Y_s},\dr_\eta u_s\right) \widehat{\dr_\eta Y_s} \bigg]\bigg\} ds \notag\\
    &\quad\qquad +  \left(\dr_\eta P_0\right)^\top \eta \Bigg\}. \label{thm:Gateaux_3}
\end{align}
From the cone property \eqref{cone:PQ:without_v} and the condition \eqref{generic:condition:b'}, we know that 
\begin{align*}
    &\e\int_0^T |P_s|\cdot \bigg\{ \left|Db_{v}\left(s,\theta_s\right)\left(\dr_\eta Y_s,\dr_\eta Y_s,\dr_\eta u_s\right)\right|\cdot \left|\dr_\eta u^0_s \right| + \left|Db_{x}\left(s,\theta_s\right)\left(\dr_\eta Y_s,\dr_\eta Y_s,\dr_\eta u_s\right)\right|\cdot \left|\dr_\eta Y_s\right| \\
    &\qquad\qquad\qquad +\he\bigg[   D\left(D_y\frac{db_x}{d\nu}\right)\left(s,\theta_s\right)\left(\hys\right)\left(\dr_\eta Y_s,\widehat{\dr_\eta Y_s},\dr_\eta u_s\right) \widehat{\dr_\eta Y_s} \bigg] \bigg\} ds \\
    \le\ &  \frac{L^2}{\lambda_0}\Big[ 5 L_0 \|\dr_\eta Y_s\|_2^2 + 4L_1 \|\dr_\eta Y_s\|_2 \cdot \|\dr_\eta u^0_s\|_2 +L_2 \|\dr_\eta u^0_s\|_2^2\Big];
\end{align*}
and similarly, from the cone property \eqref{cone:PQ:without_v} and the condition \eqref{generic:condition:A'}, we know that for $1\le j\le n$ with $d_j>0$, 
\begin{align*}
    &\left|Q^j_s\right|\cdot \bigg\{ \left|D\sigma^j_{v}\left(s,\theta_s\right)\left(\dr_\eta Y_s,\dr_\eta Y_s,\dr_\eta u_s\right)\right|\cdot \left|\dr_\eta u^j_s \right| + \left|D\sigma^j_{x}\left(s,\theta_s\right)\left(\dr_\eta Y_s,\dr_\eta Y_s,\dr_\eta u_s\right)\right|\cdot \left|\dr_\eta Y_s\right| \\
    &\quad\qquad +\he\bigg[   D\left(D_y\frac{d\sigma^j_x}{d\nu}\right)\left(s,\theta_s\right)\left(\hys\right)\left(\dr_\eta Y_s,\widehat{\dr_\eta Y_s},\dr_\eta u_s\right) \widehat{\dr_\eta Y_s} \bigg] \bigg\}\\
    \le\ & \frac{L^2}{\lambda_0} \left[ 5 L_0 \|\dr_\eta Y_s\|_2^2 + 4L_1 \|\dr_\eta Y_s\|_2 \cdot \|\dr_\eta u^j_s\|_2 +L_2 \|\dr_\eta u^j_s\|_2^2\right].
\end{align*}
Substituting the last two inequalities into \eqref{thm:Gateaux_3}, we have
\begin{align*}
    & \int_0^T 2\lambda_v\left\|\dr_\eta u_s\right\|_2^2+2(\lambda_x+\lambda_m)\left\|\dr_\eta Y_s\right\|_2^2 ds \notag \\
    \le\ & \int_0^T  \frac{L^2}{\lambda_0} \left[ 5(l+1)L_0 \|\dr_\eta Y_s\|_2^2 + 4\sqrt{l+1}L_1 \|\dr_\eta Y_s\|_2\cdot \left\|\dr_\eta u_s\right\|_2+L_2 \left\|\dr_\eta u_s\right\|_2^2 \right] ds +  \left\|\dr_\eta P_0\right\|_2 \cdot \|\eta\|_2,
\end{align*}
which implies 
\begin{align*}
    & \int_0^T \left(2\lambda_v- \frac{L^2L_2}{\lambda_0} - \frac{4(l+1)L^4 L_1^2}{\lambda_0\left[2\lambda_x\lambda_0+2\lambda_m\lambda_0-5(l+1)L^2L_0\right]}\right)\left\|\dr_\eta u_s\right\|_2^2\\
    &\quad +\Bigg(\sqrt{2\lambda_x+2\lambda_m- \frac{5(l+1)L^2L_0}{\lambda_0} }\left\|\dr_\eta Y_s\right\|_2 \\
    &\qquad\qquad - \frac{ 2\sqrt{l+1}L^2L_1}{\sqrt{\lambda_0\left[2\lambda_x\lambda_0+2\lambda_m\lambda_0- 5(l+1)L^2L_0\right] }}\left\|\dr_\eta u_s\right\|_2\bigg)^2 ds \notag \\
    \le\ & \left\|\dr_\eta P_0\right\|_2 \cdot \|\eta\|_2.
\end{align*}
From Condition \eqref{thm:suff:condition}-(ii), we then know that for any $\epsilon>0$,
\begin{align}
    & \int_0^T \left\|\dr_\eta u_s\right\|_2^2 ds \le C_0 \left\|\dr_\eta P_0\right\|_2 \cdot \|\eta\|_2\le\epsilon \left\|\dr_\eta P_0\right\|_2^2 +\frac{C_0}{4\epsilon}\|\eta\|_2^2, \label{thm:Gateaux_4}
\end{align}
where $C_0$ is a constant depending only on $(l,L,L_0,L_1,L_2,\lambda_0,\lambda_x,\lambda_m,\lambda_v)$. Applying a similar approach on the SDE in \eqref{FBSDE:Jaco} as the proof of \eqref{lem:control-state-1}, from the estimates \eqref{generic:condition:b'} and \eqref{generic:condition:A'}, we have
\begin{align}
    &\e\bigg[\sup_{0\le s\le T} \left|\dr_\eta Y_s\right|^2\bigg] \le C(L,T) \e\bigg[|\eta|^2+\int_0^T \left|\dr_\eta u_s\right|^2 ds\bigg]. \label{thm:Gateaux_5}
\end{align}
Then, by applying a similar approach to the BSDE with jump in \eqref{FBSDE:Jaco} as the proof of \eqref{adjoint:boundedness}, and using Assumption (A2), the cone property \eqref{cone:PQ:without_v}, and the estimates \eqref{generic:condition:b'} and \eqref{generic:condition:A'}, we have 
\begin{align}
        &\e\bigg[\sup_{0\le t\le T} \left|\dr_\eta P_s\right|^2+\int_0^T \left(\left|\dr_\eta Q_s\right|^2+\int_E \left|\dr_\eta R_s(e)\right|^2 \lambda(de)\right) ds \bigg] \notag \\
        \le\ & C(L,T,\lambda_0) \e\bigg[\left|\dr_\eta Y_T\right|^2+\int_0^T \left(\left|\dr_\eta Y_s\right|^2+\left|\dr_\eta u_s\right|^2 \right) ds\bigg] \notag \\
        \le\ & C(L,T,\lambda_0) \left(\|\eta\|_2^2+\e\int_0^T\left|\dr_\eta u_s\right|^2 ds\right). \label{thm:Gateaux_6}
\end{align}
Substituting \eqref{thm:Gateaux_6} into \eqref{thm:Gateaux_4}, we know that
\begin{align*}
    &\left[1- \epsilon C(L,T,\lambda_0)\right] \int_0^T \left\|\dr_\eta u_s\right\|_2^2 ds \le \left[\frac{C_0}{4\epsilon}+\epsilon C(L,T,\lambda_0)\right]\|\eta\|_2^2, 
\end{align*}
and therefore, by choosing $\epsilon=\frac{1}{2C(L,T,\lambda_0)}$, we have
\begin{align}
    & \int_0^T \left\|\dr_\eta u_s\right\|_2^2 ds \le C \|\eta\|_2^2, \label{thm:Gateaux_7}
\end{align}
for some $C>0$ depending only on $(l,L,L_0,L_1,L_2,\lambda_0,\lambda_x,\lambda_m,\lambda_v,T)$. Combining \eqref{thm:Gateaux_5}, \eqref{thm:Gateaux_6} and \eqref{thm:Gateaux_7}, we obtain \eqref{thm:Gateaux_boundedness}. 

\section{Proof of statements in Section~\ref{sec:value}}\label{sec:pf:value}

\subsection{Proof of Theorem~\ref{thm:value_regu_mu}}\label{pf:thm:value_regu_mu}

The growth estimate \eqref{V:growth} is a direct consequence of \eqref{value_equation}. We now prove \eqref{DV_P}. For $\xi,\xi'\in \lr^2_{\f_t}$ both independent of the Brownian motion $B$ and the Poisson process $N$, we have
\begin{equation}\label{thm:value_regu_mu_1}
	  J_{t,\xi'}\left(u^{t,\xi'}\right)-J_{t,\xi}\left(u^{t,\xi'}\right)\le V\left(t,\xi'\right)-V(t,\xi)\le J_{t,\xi'}\left(u^{t,\xi}\right)-J_{t,\xi}\left(u^{t,\xi}\right).
\end{equation}	
For the upper bound, from Assumption (A2), we have
\begin{align}
	&J_{t,\xi'}\left(u^{t,\xi}\right)-J_{t,\xi}\left(u^{t,\xi}\right) \notag \\
	=\ &\e\bigg\{\int_t^T\left[f\left(s,X^{t,\xi',u^{t,\xi}}_s,\lr\left(X^{t,\xi',u^{t,\xi}}_s\right),u^{t,\xi}_s\right)-f\left(s,Y^{t,\xi}_s,\lr\left(Y^{t,\xi}_s\right),u^{t,\xi}_s\right)\right] ds \notag \\
    &\quad + g\left(X^{t,\xi',u^{t,\xi}}_T,\lr\left(X^{t,\xi',u^{t,\xi}}_T\right)\right)-g\left(Y^{t,\xi}_T,\lr\left(Y^{t,\xi}_T\right)\right) \bigg\} \notag \\
	\le\ & \e\bigg\{\int_t^T  f_x\left(s,\theta^{t,\xi}_s\right)^\top \left(X^{t,\xi',u^{t,\xi}}_s-Y^{t,\xi}_s\right) +\he\bigg[\left(D_y\frac{df}{d\nu}\left(s,\theta_s^{t,\xi}\right)\left(\widehat{Y_s^{t,\xi}}\right)\right)^\top \left(\widehat{X^{t,\xi',u^{t,\xi}}_s}-\widehat{Y^{t,\xi}_s}\right) \bigg] ds \notag \\
	&\quad + g_x\left(Y^{t,\xi}_T,\lr\left(Y^{t,\xi}_T\right)\right)^\top\left(X^{t,\xi',u^{t,\xi}}_T-Y^{t,\xi}_T\right) \notag \\
    &\quad +\he\bigg[\left(D_y\frac{dg}{d\nu}\left(Y^{t,\xi}_T,\lr\left(Y^{t,\xi}_T\right)\right)\left(\widehat{Y_T^{t,\xi}}\right)\right)^\top \left(\widehat{X^{t,\xi',u^{t,\xi}}_T}-\widehat{Y^{t,\xi}_T}\right)\bigg]\bigg\} \notag \\
	&+C(L,T)\sup_{t\le s\le T}\left\|X^{t,\xi',u^{t,\xi}}_s-Y^{t,\xi}_s\right\|^2_2, \label{thm:value_regu_mu_2}
\end{align}
where $\theta^{t,\xi}_s=\left(Y_s^{t,\xi},\lr\left(Y_s^{t,\xi}\right),u_s^{t,\xi}\right)$. %for $s\in[t,T]$. 
By applying It\^o's lemma on $\left(P^{t,\xi}_s\right)^\top \left(X^{t,\xi',u^{t,\xi}}_s-Y^{t,\xi}_s\right)$ and taking expectation, and using Fubini's theorem (similar as in \eqref{fact:Fubini}) and the fact that $X^{t,\xi',u^{t,\xi}}_s,Y^{t,\xi}\in\sr_\f^2(t,s) $ (similar as in \eqref{fact:llrc}), we know that 
\begin{align*}
    &\e\left[\left(P^{t,\xi}_T\right)^\top \left(X^{t,\xi',u^{t,\xi}}_T-Y^{t,\xi}_T\right)-\left(P^{t,\xi}_t\right)^\top \left(\xi'-\xi\right)\right]\\
    =\ & \e \int_t^T \left(P^{t,\xi}_s\right)^\top \bigg\{B\left(s,X^{t,\xi',u^{t,\xi}}_s,\lr\left(X^{t,\xi',u^{t,\xi}}_s\right),u^{t,\xi,0}_s\right) -B\left(s,Y^{t,\xi}_s,\lr\left(Y^{t,\xi}_s\right),u^{t,\xi,0}_s\right) \\
    &\quad\qquad\qquad\qquad -B_x \left(s,Y^{t,\xi}_s,\lr\left(Y^{t,\xi}_s\right),u^{t,\xi,0}_s\right) \left(X^{t,\xi',u^{t,\xi}}_s-Y^{t,\xi}_s\right)  \\
    &\quad\qquad\qquad\qquad - \he\bigg[\left(D_y\frac{dB}{d\nu}\left(s,Y^{t,\xi}_s,\lr\left(Y^{t,\xi}_s\right),u^{t,\xi,0}_s\right)\left(\widehat{Y_s^{t,\xi}}\right)\right) \left(\widehat{X^{t,\xi',u^{t,\xi}}_s}-\widehat{Y^{t,\xi}_s}\right) \bigg]\bigg\} \\
    &\qquad + \sum_{j,\ d_j>0} \left(Q^{t,\xi,j}_s\right)^\top \bigg\{A^j\left(s,X^{t,\xi',u^{t,\xi}}_s,\lr\left(X^{t,\xi',u^{t,\xi}}_s\right),u^{t,\xi,j}_s\right)-A^j\left(s,Y^{t,\xi}_s,\lr\left(Y^{t,\xi}_s\right),u^{t,\xi,j}_s\right)\\
    &\quad\qquad\qquad\qquad\qquad - A^j_x \left(s,Y^{t,\xi}_s,\lr\left(Y^{t,\xi}_s\right),u^{t,\xi,j}_s\right) \left(X^{t,\xi',u^{t,\xi}}_s-Y^{t,\xi}_s\right) \\
    &\quad\qquad\qquad\qquad\qquad - \he\bigg[\left(D_y\frac{dA^j}{d\nu}\left(s,Y^{t,\xi}_s,\lr\left(Y^{t,\xi}_s\right),u^{t,\xi,j}_s\right)\left(\widehat{Y_s^{t,\xi}}\right)\right) \left(\widehat{X^{t,\xi',u^{t,\xi}}_s}-\widehat{Y^{t,\xi}_s}\right) \bigg] \bigg\} \\
    &\qquad -f_x\left(s,\theta^{t,\xi}_s\right)^\top \left(X^{t,\xi',u^{t,\xi}}_s-Y^{t,\xi}_s\right) -\he\bigg[ \left(D_y\frac{df}{d\nu}\left(s,\theta^{t,\xi}_s\right)\left(\widehat{Y_s^{t,\xi}}\right)\right)^\top \left(\widehat{X^{t,\xi',u^{t,\xi}}_s}-\widehat{Y^{t,\xi}_s} \right) \bigg] ds;
\end{align*}
here, similar approach is also used in \eqref{thm:suff_1} and \eqref{thm:Gateaux_1}.  Similarly, we have
\begin{align*}
    &\e\left[\left(P^{t,\xi}_T\right)^\top \left(X^{t,\xi',u^{t,\xi}}_T-Y^{t,\xi}_T\right)\right]\\
    =\ & \e\bigg\{ g_x\left(Y_T^{t,\xi},\lr\left(Y_T^{t,\xi}\right)\right)^\top \left(X^{t,\xi',u^{t,\xi}}_T-Y^{t,\xi}_T\right)\\
    &\quad +\he\left[D_y\frac{dg}{d\nu}\left(\widehat{Y_T^{t,\xi}},\lr\left(Y^{t,\xi}_T\right)\right)\left(Y^{t,\xi}_T\right)\right]^\top \left(X^{t,\xi',u^{t,\xi}}_T-Y^{t,\xi}_T\right) \bigg\}\\
    =\ & \e\bigg\{g_x\left(Y^{t,\xi}_T,\lr\left(Y^{t,\xi}_T\right)\right)^\top\left(X^{t,\xi',u^{t,\xi}}_T-Y^{t,\xi}_T\right) \\
    &\quad +\he\bigg[\left(D_y\frac{dg}{d\nu}\left(Y^{t,\xi}_T,\lr\left(Y^{t,\xi}_T\right)\right)\left(\widehat{Y_T^{t,\xi}}\right)\right)^\top \left(\widehat{X^{t,\xi',u^{t,\xi}}_T}-\widehat{Y^{t,\xi}_T}\right)\bigg]\bigg\}.
\end{align*}
Substituting the last two equations into \eqref{thm:value_regu_mu_2}, we have
\begin{align}
	&J_{t,\xi'}\left(u^{t,\xi}\right)-J_{t,\xi}\left(u^{t,\xi}\right) \notag \\
	\le\ & \e\Bigg\{\int_t^T  \left(P^{t,\xi}_s\right)^\top \Bigg\{\left[\int_0^1 B_x\left|^{\left(s,X^{h}_s,\lr\left(X^{h}_s\right),u^{t,\xi,0}_s\right)}_{\left(s,Y^{t,\xi}_s,\lr\left(Y^{t,\xi}_s\right),u^{t,\xi,0}_s\right)}\right. dh  \right] \left(X^{t,\xi',u^{t,\xi}}_s-Y^{t,\xi}_s\right) \notag \\
    &\quad\qquad\qquad\qquad + \he\Bigg[\Bigg(\int_0^1 D_y\frac{dB}{d\nu}\left|^{\left(s,X^{h}_s,\lr\left(X^{h}_s\right),u^{t,\xi,0}_s\right)\left(\widehat{X_s^{h}}\right)}_{\left(s,Y^{t,\xi}_s,\lr\left(Y^{t,\xi}_s\right),u^{t,\xi,0}_s\right)\left(\widehat{Y_s^{t,\xi}}\right)} \right. dh \Bigg) \left(\widehat{X^{t,\xi',u^{t,\xi}}_s}-\widehat{Y^{t,\xi}_s}\right) \Bigg]\Bigg\} \notag \\
    &\qquad + \sum_{j,\ d_j>0} \left(Q^{t,\xi,j}_s\right)^\top \Bigg\{\left[\int_0^1 A^j_x \left|^{\left(s,X^{h}_s,\lr\left(X^{h}_s\right),u^{t,\xi,j}_s\right)}_{\left(s,Y^{t,\xi}_s,\lr\left(Y^{t,\xi}_s\right),u^{t,\xi,j}_s\right)} \right. dh\right] \left(X^{t,\xi',u^{t,\xi}}_s-Y^{t,\xi}_s\right) \notag \\
    &\qquad\qquad\qquad\qquad + \he\Bigg[\int_0^1 \Bigg(D_y\frac{dA^j}{d\nu}\left|^{\left(s,X^{h}_s,\lr\left(X^{h}_s\right),u^{t,\xi,j}_s\right)\left(\widehat{X_s^{h}}\right)}_{\left(s,Y^{t,\xi}_s,\lr\left(Y^{t,\xi}_s\right),u^{t,\xi,j}_s\right)\left(\widehat{Y_s^{t,\xi}}\right)} \right. dh \Bigg) \left(\widehat{X^{t,\xi',u^{t,\xi}}_s}-\widehat{Y^{t,\xi}_s}\right) \Bigg] \Bigg\} ds \notag \\
    &\qquad + \left(P^{t,\xi}_t\right)^\top \left(\xi'-\xi\right)\Bigg\} +C(L,T)\sup_{t\le s\le T}\left\|X^{t,\xi',u^{t,\xi}}_s-Y^{t,\xi}_s\right\|^2_2. \label{thm:value_regu_mu_3}
\end{align}
where $X_s^{h}:=Y_s^{t,\xi}+h\left(X_s^{t,\xi',u^{t,\xi}}-Y_s^{t,\xi}\right)$ for $s\in[t,T]$ and $h\in[0,1]$. From the cone property \eqref{cone:PQ:without_v} and Assumption (A3)-(i), we know that 
\begin{align*}
    &\e \int_t^T  \left(P^{t,\xi}_s\right)^\top \Bigg\{\left[\int_0^1 B_x\left|^{\left(s,X^{h}_s,\lr\left(X^{h}_s\right),u^{t,\xi,0}_s\right)}_{\left(s,Y^{t,\xi}_s,\lr\left(Y^{t,\xi}_s\right),u^{t,\xi,0}_s\right)}\right. dh  \right] \left(X^{t,\xi',u^{t,\xi}}_s-Y^{t,\xi}_s\right) \notag \\
    &\quad\qquad\qquad\qquad + \he\Bigg[\Bigg(\int_0^1 D_y\frac{dB}{d\nu}\left|^{\left(s,X^{h}_s,\lr\left(X^{h}_s\right),u^{t,\xi,0}_s\right)\left(\widehat{X_s^{h}}\right)}_{\left(s,Y^{t,\xi}_s,\lr\left(Y^{t,\xi}_s\right),u^{t,\xi,0}_s\right)\left(\widehat{Y_s^{t,\xi}}\right)} \right. dh \Bigg) \left(\widehat{X^{t,\xi',u^{t,\xi}}_s}-\widehat{Y^{t,\xi}_s}\right) \Bigg]\Bigg\} ds\\
    \le\ & \frac{5L^2L_0}{2\lambda_0} \int_t^T \left\|X^{t,\xi',u^{t,\xi}}_s-Y^{t,\xi}_s\right\|^2_2 ds,
\end{align*}
and similarly, for $1\le j\le n$ with $d_j>0$, 
\begin{align*}
    &\e \int_t^T \left(Q^{t,\xi,j}_s\right)^\top \Bigg\{\left[\int_0^1 A^j_x \left|^{\left(s,X^{h}_s,\lr\left(X^{h}_s\right),u^{t,\xi,j}_s\right)}_{\left(s,Y^{t,\xi}_s,\lr\left(Y^{t,\xi}_s\right),u^{t,\xi,j}_s\right)} \right. dh\right] \left(X^{t,\xi',u^{t,\xi}}_s-Y^{t,\xi}_s\right) \notag \\
    &\qquad\qquad\qquad\qquad + \he\Bigg[\int_0^1 \Bigg(D_y\frac{dA^j}{d\nu}\left|^{\left(s,X^{h}_s,\lr\left(X^{h}_s\right),u^{t,\xi,j}_s\right)\left(\widehat{X_s^{h}}\right)}_{\left(s,Y^{t,\xi}_s,\lr\left(Y^{t,\xi}_s\right),u^{t,\xi,j}_s\right)\left(\widehat{Y_s^{t,\xi}}\right)} \right. dh \Bigg) \left(\widehat{X^{t,\xi',u^{t,\xi}}_s}-\widehat{Y^{t,\xi}_s}\right) \Bigg] \Bigg\} ds\\
    \le\ & \frac{5L^2L_0}{2\lambda_0}\int_t^T \left\|X^{t,\xi',u^{t,\xi}}_s-Y^{t,\xi}_s\right\|^2_2 ds.
\end{align*}
Substituting the last two estimates back into \eqref{thm:value_regu_mu_3}, we know that 
\begin{align*}
	&J_{t,\xi'}\left(u^{t,\xi}\right)-J_{t,\xi}\left(u^{t,\xi}\right) \notag \\
	\le\ & \e\left[\left(P^{t,\xi}_t\right)^\top \left(\xi'-\xi\right)\right] +C(l,L,\lambda_0,T)\sup_{t\le s\le T}\left\|X^{t,\xi',u^{t,\xi}}_s-Y^{t,\xi}_s\right\|^2_2. 
\end{align*}
From Assumption (A1) and the Gr\"onwall's inequality, similar as the proof of Lemma~\ref{lem:control-state}, we have the following estimate
\begin{equation*}
	\begin{split}
		\sup_{t\le s\le T}\left\|X^{t,\xi',u^{t,\xi}}_s-Y^{t,\xi}_s\right\|_2 \le C(L,T)\left\|\xi'-\xi\right\|_2,
	\end{split}
\end{equation*}
therefore,
\begin{align}
	&J_{t,\xi'}\left(u^{t,\xi}\right)-J_{t,\xi}\left(u^{t,\xi}\right)\le \e\left[\left(P^{t,\xi}_t\right)^\top \left(\xi'-\xi\right)\right] +C(l,L,\lambda_0,T)\left\|\xi'-\xi\right\|_2^2. \label{thm:value_regu_mu_4}
\end{align}
In a similar way, we can also obtain the lower bound
\begin{equation}\label{thm:value_regu_mu_5}
	\begin{split}
		&J_{t,\xi'}\left(u^{t,\xi'}\right)-J_{t,\xi}\left(u^{t,\xi'}\right)\geq \e\left[\left(P^{t,\xi'}_t\right)^\top \left(\xi'-\xi\right)\right] -C(l,L,\lambda_0,T)\left\|\xi'-\xi\right\|_2^2,
	\end{split}
\end{equation}
which is also similar as in our previous work \cite{AB10}. From the estimate \eqref{lem:kappa_0}, we know that 
\begin{align}\label{thm:value_regu_mu_6}
    \left\|P^{t,\xi'}_t-P^{t,\xi}_t \right\|\le C\left\|\xi'-\xi\right\|_2.
\end{align}
Here, $C$ is a constant depending only on $(l,L,L_0,L_1,L_2,\lambda_0,\lambda_x,\lambda_m,\lambda_v,T)$. From \eqref{thm:value_regu_mu_4}-\eqref{thm:value_regu_mu_6} and \eqref{thm:value_regu_mu_1}, we conclude that 
\begin{align*}
    &\left|V\left(t,\lr(\xi')\right)-V(t,\lr(\xi))-\e\left[\left(P^{t,\xi}_t\right)^\top \left(\xi'-\xi\right)\right]\right|\le C \left\|\xi'-\xi\right\|_2^2,
\end{align*}
from which we know that $D_\xi  V(t,\lr(\xi))=P^{t,\xi}_t$. Then, the estimate \eqref{DV:growth} is a direct consequence of Estimate \eqref{lem:kappa_0}, and from Theorem~\ref{thm:Gateaux}, we obtain the relation for the second-order derivative in \eqref{DV_P} and also the estimate \eqref{DDV:growth}. Besides, from the twice G\^ateaux differentiability of $V$ in $\xi$ and Estimates \eqref{V:growth}-\eqref{DDV:growth}, following similar arguments as in \cite[Theorem 7.1]{AB10}, we can show that $V$ is twice linearly functional-differentiable; and from standard relations between the G\^ateaux derivatives and linear functional-derivatives as in \eqref{lem01_1} (also see \cite{AB11,AB10,ABmfg1,AB5}), we obtain \eqref{relation:DV-dV}.

\subsection{Proof of Lemma~\ref{lem:Theta_regu_s}}\label{pf:lem:Theta_regu_s}

From the SDE for $Y^{t,\xi}$ in \eqref{FBSDE:t,xi}, we know that for any $t\le s\le s'\le T$, similar as in Lemma~\ref{lem:control-state}, by Cauchy-Schwarz inequality, we have 
\begin{align}
    \left\|Y^{t,\xi}_{s'}-Y^{t,\xi}_s\right\|_2^2=\ & \e \left[\left|\int_s^{s'} b\left(r,\theta^{t,\xi}_r\right) dr +\int_s^{s'} \sigma \left(r,\theta^{t,\xi}_r\right)dB_r+\int_s^{s'} \int_E \gamma\left(r,\theta^{t,\xi}_{r-},e\right) \mathring{N}(de,dr)\right|^2 \right] \notag \\
    \le\ & C(L,T) \int_s^{s'} \left(1+\left\|Y_r^{t,\xi}\right\|_2^2+ \left\|u_r^{t,\xi}\right\|_2^2 \right) dr. \label{lem:Theta_regu_s_1}
\end{align}
From \eqref{cone:u} and Assumption (A3'), we know that 
\begin{align}\label{lem:Theta_regu_s_2}
    \left|u^{t,\xi}_s\right|\le\ & C(l,L,\lambda_v)\left[1+\left|Y^{t,\xi}_{s-}\right|+\left|\lr\left(Y^{t,\xi}_{s-}\right)\right|_1+\left|P^{t,\xi}_{s-}\right|\right].
\end{align}
From \eqref{lem:Theta_regu_s_1}, \eqref{lem:Theta_regu_s_2} and Estimate \eqref{lem:kappa_0}, we have
\begin{align}
    \left\|Y^{t,\xi}_{s'}-Y^{t,\xi}_s\right\|_2^2 \le\ & C(l,L,\lambda_v,T) \int_s^{s'} \left(1+\left\|Y_r^{t,\xi}\right\|_2^2+ \left\|u_r^{t,\xi}\right\|_2^2 +\left\|P_r^{t,\xi}\right\|_2^2\right) dr \notag\\
    \le\ & C \left(1+\|\xi\|_2^2\right)|s'-s|. \label{lem:Theta_regu_s_3}
\end{align}
Here $C$ is a constant depending only on $(l,L,L_0,L_1,L_2,\lambda_0,\lambda_x,\lambda_m,\lambda_v,T)$. For the process $P^{t,\xi}$ in \eqref{FBSDE:t,xi}, we note that
\begin{align*}
	&P^{t,\xi}_{s'}-P^{t,\xi}_s \notag \\
    =\ & P^{t,\xi}_{s'}-\e\left[P^{t,\xi}_{s'}\big|\f_s\right] -\int_s^{s'}\e\Bigg\{\bigg\{ b_x\left(r,\theta^{t,\xi}_r\right)^\top P^{t,\xi}_r+\sum_{j=1}^n \sigma^j_x \left(r,\theta^{t,\xi}_r\right)^\top Q^{t,\xi,j}_r\\
    &\qquad +\int_E \gamma_x\left(r,\theta^{t,\xi}_r,e\right)^\top R^{t,\xi}_r(e)\lambda(de)+f_x\left(r,\theta^{t,\xi}_r\right)\\
    &\qquad +\he\bigg[\left(D_y\frac{db}{d\nu}\left(r,\widehat{\theta^{t,\xi}_r}\right)\left(Y_r^{t,\xi}\right)\right)^\top \widehat{P^{t,\xi}_r}+\sum_{j=1}^n \left(D_y\frac{d\sigma^j}{d\nu}\left(r,\widehat{\theta_r^{t,\xi}}\right)\left(Y^{t,\xi}_r\right)\right)^\top \widehat{Q_r^{t,\xi,j}} \\
    &\qquad\qquad +\int_E \left(D_y\frac{d\gamma}{d\nu}\left(r,\widehat{\theta_r^{t,\xi}},e\right)\left(Y_r^{t,\xi}\right)\right)^\top \widehat{R_r^{t,\xi}}(e) \lambda(de)+ D_y\frac{df}{d\nu}\left(r,\widehat{\theta^{t,\xi}_r}\right)\left(Y_r^{t,\xi}\right) \bigg]\bigg\}\Bigg|\f_s\Bigg\}dr. 
\end{align*}  
From Cauchy-Schwarz inequality and the second estimate in \eqref{lem:kappa_0}, we deduce that
\begin{align}
	&\e\left[ \left|P^{t,\xi}_{s'}-P^{t,\xi}_s\right|^2 \right] \notag \\
	\le\ &  |s'-s|\cdot C(L) \int_s^{s'} \left(\left\|Y^{t,\xi}_r\right\|_2^2+\left\|u^{t,\xi}_r\right\|_2^2+\left\|P^{t,\xi}_r\right\|_2^2+\left\|Q^{t,\xi}_r\right\|_2^2+\left\|R^{t,\xi}_r(\cdot)\right\|_{L^2(\Omega;L_\lambda^2)} \right) dr \notag \\
	\le\ & C\left(1+\|\xi\|_2^2\right) |s'-s|.\label{lem:Theta_regu_s_3.5}
\end{align}
We next aim to establish the continuity of $u^{t,\xi}_s$ in $s$. From the optimality condition and Assumption (A3'), we know that
\begin{align}
    0=\ & \Bigg[b_v\left|^{\left(s',Y^{t,\xi}_{s'-},\lr\left(Y^{t,\xi}_{s'-}\right),u^{t,\xi}_{s'}\right)}_{\left(s,Y^{t,\xi}_{s-},\lr\left(Y^{t,\xi}_{s-}\right),u^{t,\xi}_s\right)} \right. \Bigg]^\top P^{t,\xi}_{s'-}+b_{v}\left(s,Y^{t,\xi}_{s-},\lr\left(Y^{t,\xi}_{s-}\right),u^{t,\xi}_s\right)^\top \left(P^{t,\xi}_{s'-}-P^{t,\xi}_{s-}\right) \notag \\
    &+f_{v}\left|^{\left(s',Y^{t,\xi}_{s'-},\lr\left(Y^{t,\xi}_{s'-}\right),u^{t,\xi}_{s'}\right)}_{\left(s,Y^{t,\xi}_{s-},\lr\left(Y^{t,\xi}_{s-}\right),u^{t,\xi}_s\right)} \right. , \label{lem:Theta_regu_s_4}
\end{align}
and therefore, 
\begin{align*}
    &-\left(f_{v}\left|^{\left(s,Y^{t,\xi}_{s-},\lr\left(Y^{t,\xi}_{s-}\right),u^{t,\xi}_{s'}\right)}_{\left(s,Y^{t,\xi}_{s-},\lr\left(Y^{t,\xi}_{s-}\right),u^{t,\xi}_s\right)} \right.\right)^\top \left(u^{t,\xi}_{s'}-u^{t,\xi}_s\right)\\
    =\ & \left(u^{t,\xi}_{s'}-u^{t,\xi}_s\right)^\top \Bigg[b_{v}\left|^{\left(s',Y^{t,\xi}_{s'-},\lr\left(Y^{t,\xi}_{s'-}\right),u^{t,\xi}_{s'}\right)}_{\left(s,Y^{t,\xi}_{s-},\lr\left(Y^{t,\xi}_{s-}\right),u^{t,\xi}_s\right)} \right. \Bigg]^\top P^{t,\xi}_{s'-}\\
    &+\left(u^{t,\xi}_{s'}-u^{t,\xi}_s\right)^\top b_{v}\left(s,Y^{t,\xi}_{s-},\lr\left(Y^{t,\xi}_{s-}\right),u^{t,\xi}_s\right)^\top \left(P^{t,\xi}_{s'-}-P^{t,\xi}_{s-}\right)\\
    &+\left(f_{v}\left|^{\left(s',Y^{t,\xi}_{s'-},\lr\left(Y^{t,\xi}_{s'-}\right),u^{t,\xi}_{s'}\right)}_{\left(s,Y^{t,\xi}_{s-},\lr\left(Y^{t,\xi}_{s-}\right),u^{t,\xi}_{s'}\right)} \right.\right)^\top \left(u^{t,\xi}_{s'}-u^{t,\xi}_s\right).
\end{align*}
Then, from the convexity of $f$ in (A4), the cone property \eqref{cone:PQ:without_v}, Assumption (A3'), we know that
\begin{align*}
    &-2\lambda_v \left|u^{t,\xi}_{s'}-u^{t,\xi}_s\right|^2\\
    \geq\ & \left|u^{t,\xi}_{s'}-u^{t,\xi}_s\right|\cdot \frac{L^2}{\lambda_0} \left(L|s'-s|^{\frac{1}{2}}+L_0\left|Y^{t,\xi}_{s'-}-Y^{t,\xi}_{s-}\right|+L_0\left\|Y^{t,\xi}_{s'-}-Y^{t,\xi}_{s-}\right\|_2+L_2 \left|u^{t,\xi}_{s'}-u^{t,\xi}_s\right| \right)\\
    &+\left|u^{t,\xi}_{s'}-u^{t,\xi}_s\right|\cdot L \left|P^{t,\xi}_{s'-}-P^{t,\xi}_{s-}\right|\\
    &+L\left(|s'-s|^{\frac{1}{2}}+\left|Y^{t,\xi}_{s'-}-Y^{t,\xi}_{s-}\right|+\left\|Y^{t,\xi}_{s'-}-Y^{t,\xi}_{s-}\right\|_2\right) \cdot \left|u^{t,\xi}_{s'}-u^{t,\xi}_s\right|,
\end{align*}
and therefore, from Condition \eqref{thm:suff:condition}-(i), 
\begin{align}
    &\left|u^{t,\xi}_{s'}-u^{t,\xi}_s\right| \notag\\
    \le\ & \left(2\lambda_v- \frac{L^2L_2}{\lambda_0}\right)^{-1} L\left(1+ \frac{L^2}{\lambda_0}\right) \left(|s'-s|^{\frac{1}{2}}+\left|Y^{t,\xi}_{s'-}-Y^{t,\xi}_{s-}\right|+\left\|Y^{t,\xi}_{s'-}-Y^{t,\xi}_{s-}\right\|_2\right) \notag \\
    &+\left(2\lambda_v- \frac{L^2L_2}{\lambda_0}\right)^{-1} L \left|P^{t,\xi}_{s'-}-P^{t,\xi}_{s-}\right|. \label{lem:Theta_regu_s_5}
\end{align}
Then, from \eqref{lem:Theta_regu_s_3} and \eqref{lem:Theta_regu_s_3.5}, we know that for a.e. $s,s'\in[t,T]$, 
\begin{align*}
    \left\|u^{t,\xi}_{s'}-u^{t,\xi}_s\right\|^2_2\le\ & C\left(1+\|\xi\|_2^2\right)|s'-s|,
\end{align*}
so we complete the proof. 

\subsection{Proof of Lemma~\ref{lem:Ito}}\label{pf:lem:Ito}

Note that the process $X$ defined in \eqref{Ito:X} has jumps induced by the Poisson random measure, whereas $\lr(X_s)$ is continuous in $s$; see \cite{agram2023stochastic} for instance. From the functional differentiability of $F$, we first note that
\begin{align}\label{Ito_1}
    &\frac{d}{ds}F(s,\lr(X_s))=\frac{\dd F}{\dd s}(s,\lr(X_s))+\left(\frac{d}{ds} \e\left[ \frac{dF}{d\nu}(t,\mu) (X_s) \right]\right)\bigg|_{(t,\mu)=(s,\lr(X_s))};
\end{align}
we also refer to \cite[Theorem 7.1]{BR} for similar results. Then, by It\^o's formula for jump diffusion (see \cite[Theorem 1.16]{Oksendal-Sulem} for instance), we have
\begin{align*}
    &\left(\frac{d}{ds} \e\left[ \frac{dF}{d\nu}(t,\mu) (X_s) \right]\right)\bigg|_{(t,\mu)=(s,\lr(X_s))}\\
    =\ & \e\bigg\{ \left(D_y\frac{dF}{d\nu}(s,\lr(X_s))(X_s)\right)^\top b_s+ \frac{1}{2}\text{Tr}\left[\left(\sigma_s\sigma_s^\top\right) D_y^2\frac{dF}{d\nu}(s,\lr(X_s))(X_s)\right]\\
    &\quad+\int_E \bigg[\frac{dF}{d\nu}(s,\lr(X_s))(X_{s-}+\gamma_s(e))- \frac{dF}{d\nu}(s,\lr(X_s))(X_{s-})\\
    &\qquad\qquad - \left(D_y \frac{dF}{d\nu}(s,\lr(X_s))(X_{s-})\right)^\top \gamma_s(e) \bigg] \lambda(de) \bigg\}.
\end{align*}
Combining \eqref{Ito_1} and the last equality, we obtain \eqref{Ito_0}. 

\section{Proof of Proposition~\ref{thm:value_regu_mu'}}\label{pf:thm:value_regu_mu'}

From Theorem~\ref{thm:value_regu_mu} and the study for the processes $\Theta^{t,y,\mu}$ and $D_y \Theta^{t,y,\mu}$ in Subsection~\ref{subsec:D_y}, we obtain \eqref{relation:DV}; we also refer to \cite{AB10,AB5} for a detailed discussion on the linearly functional-derivative of $V$ and the processes $\Theta^{t,y,\mu}$. With \eqref{relation:DV}, Estimate \eqref{boundedness:DdV} is a consequence of \eqref{add-4}, while Estimate \eqref{thm:D^2dV_02} is a consequence of \eqref{add-5}. Now it remains to prove \eqref{lem:Q&V_00}. For the sake of notational convenience, we only prove the case when $n=d=1$, and denote by $a(t,x,m):=\frac{1}{2}|\sigma(s,x,m)|^2$. From Lemma~\ref{thm:value_regu_mu} and the flow property, we know that
\begin{align}\label{lem:Q&V_0.5}
    P^{t,\xi}_s=P^{s,Y^{t,\xi}_s}_s=D_y\frac{dV}{d\nu}\left(s,\lr\left(Y_s^{t,\xi}\right)\right)\left(Y_s^{t,\xi}\right). 
\end{align}
We can use mollifier for $V$ making the resulting convolution smooth enough; and we shall see later (in \eqref{lem:Q&V_2}) that the higher order derivatives only appear in intermediate steps, and they shall eventually cancel out each other, therefore, without loss of generality, we suppose that the derivatives 
\begin{align*}
    \left(D_y^3\frac{dV}{d\nu}(t,\mu)(y),D_{y'}D_y^2\frac{d^2V}{d\nu^2}(t,\mu)(y,y')\right)
\end{align*}
exist and are continuous for any $(t,\mu,y)\in[0,T]\times\pr_2(\brn)\times\brn$. From the mean field version of It\^o's formula in Lemma~\ref{lem:Ito}, we can compute that 
\begin{align}
    dP^{t,\xi}_s=\ & D_y\frac{d(\frac{\dd V}{\dd t})}{d\nu}\left(s,\lr\left(Y_s^{t,\xi}\right)\right)\left(Y_s^{t,\xi}\right)ds + D_y^2\frac{dV}{d\nu}\left(s,\lr\left(Y_s^{t,\xi}\right)\right)\left(Y_s^{t,\xi}\right)[\bm{b}(s)ds+\bm{\sigma}(s)dB_s] \notag \\
    &+D_y^3\frac{dV}{d\nu}\left(s,\lr\left(Y_s^{t,\xi}\right)\right)\left(Y_s^{t,\xi}\right) \bm{a}(s) ds  \notag \\
    &+\int_E \bigg[D_y\frac{dV}{d\nu}\left(s,\lr\left(Y_s^{t,\xi}\right)\right)\left(Y_{s-}^{t,\xi}+\bm{\gamma}(s,e)\right)-D_y\frac{dV}{d\nu}\left(s,\lr\left(Y_s^{t,\xi}\right)\right)\left(Y_{s-}^{t,\xi}\right) \notag \\
    &\quad\qquad -D_y^2 \frac{dV}{d\nu}\left(s,\lr\left(Y_s^{t,\xi}\right)\right)\left(Y_{s-}^{t,\xi}\right)\bm{\gamma}(s,e)\lambda(de)\bigg]ds \notag \\
    &+\int_E \bigg[D_y\frac{dV}{d\nu}\left(s,\lr\left(Y_s^{t,\xi}\right)\right)\left(Y_{s-}^{t,\xi}+\bm{\gamma}(s,e)\right)-D_y\frac{dV}{d\nu}\left(s,\lr\left(Y_s^{t,\xi}\right)\right)\left(Y_{s-}^{t,\xi}\right)\bigg] \mathring{N}(de,ds) \notag \\
    &+\he\Bigg\{D_{y'}D_y\frac{d^2V}{d\nu^2}\left(s,\lr\left(Y_s^{t,\xi}\right)\right)\left(Y_s^{t,\xi},\widehat{Y_s^{t,\xi}}\right)\widehat{\bm{b}(s)} + D_{y'}^2 D_y\frac{d^2V}{d\nu^2}\left(s,\lr\left(Y_s^{t,\xi}\right)\right)\left(Y_s^{t,\xi},\widehat{Y_s^{t,\xi}}\right)\widehat{\bm{a}(s)} \notag \\
    &\qquad + \int_E \bigg[ D_y\frac{d^2V}{d\nu^2}\left(s,\lr\left(Y_s^{t,\xi}\right)\right)\left(Y_s^{t,\xi},\widehat{Y_{s-}^{t,\xi}}+\widehat{\bm{\gamma}(s,e)}\right)- D_y\frac{d^2V}{d\nu^2}\left(s,\lr\left(Y_s^{t,\xi}\right)\right)\left(Y_s^{t,\xi},\widehat{Y_{s-}^{t,\xi}}\right) \notag \\
    &\quad\qquad\qquad - D_{y'}D_y\frac{d^2V}{d\nu^2}\left(s,\lr\left(Y_s^{t,\xi}\right)\right)\left(Y_s^{t,\xi},\widehat{Y_{s-}^{t,\xi}}\right)\widehat{\bm{\gamma}(s,e)}\bigg]\lambda(de)\Bigg\}ds, \label{lem:Q&V_1}
\end{align}
where we simply denote by 
\begin{align*}
    &\bm{b}(s):=b\left(s,\theta_s^{t,\xi}\right),\quad \bm{\gamma}(s,e):=\gamma\left(s,Y_s^{t,\xi},\lr\left(Y_s^{t,\xi}\right),e\right),\\
    &\bm{\sigma}(s):=\sigma\left(s,Y_s^{t,\xi},\lr\left(Y_s^{t,\xi}\right)\right),\quad \bm{a}(s):=a\left(s,Y_s^{t,\xi},\lr\left(Y_s^{t,\xi}\right)\right).
\end{align*}
From Theorem~\ref{thm:V_t} and the definition of $u^{t,y,\mu}_t$, we know that \eqref{thm:V_t_0} also writes
\begin{align*}
    \frac{\dd V}{\dd t}(t,\mu)=\ & -\int_\brn \Bigg\{\frac{1}{2}\text{Tr}\left[\left(\sigma\sigma^\top\right) \left(t,y,\mu\right) D_y^2\frac{dV}{d\nu}(t,\mu)(y)\right] \notag \\
    &\quad\qquad +\left(D_y\frac{dV}{d\nu}(t,\mu)(y)\right)^\top b\left(t,y,\mu,u^{t,y,\mu}_t\right)  +f\left(t,y,\mu,u^{t,y,\mu}_t\right)\notag \\
    &\quad\qquad +\int_E \bigg[\frac{dV}{d\nu}(t,\mu)\left(y+\gamma\left(t,y,\mu,e\right)\right)-\frac{dV}{d\nu}(t,\mu)(y)\\
    &\qquad\qquad\qquad -\left(D_y\frac{dV}{d\nu}(t,\mu)(y)\right)^\top \gamma\left(t,y,\mu,e\right) \bigg] \lambda(de) \Bigg\} \mu(dy),
\end{align*}
therefore, we can compute that
\begin{align*}
    &D_y \frac{d (\frac{\dd V}{\dd t})}{d\nu}(t,\mu)(\xi)\\
    =\ & - a\left(t,\xi,\mu\right) D^3_y\frac{dV}{d\nu}(t,\mu)(\xi)-a_x\left(t,\xi,\mu\right) D^2_y\frac{dV}{d\nu}(t,\mu)(\xi) \notag \\
    &- b\left(t,\xi,\mu,u^{t,\xi}_t\right) D^2_y\frac{dV}{d\nu}(t,\mu)(\xi)-b_x\left(t,\xi,\mu,u^{t,\xi}_t\right) D_y\frac{dV}{d\nu}(t,\mu)(\xi)-f_x\left(t,\xi,\mu,u^{t,\xi}_t\right)\notag \\
    &- \int_E \bigg[D_y\frac{dV}{d\nu}(t,\mu)\left(\xi+\gamma\left(t,\xi,\mu,e\right)\right)\left[1+\gamma_x(t,\xi,\mu,e)\right]-D_y\frac{dV}{d\nu}(t,\mu)(\xi)\\
    &\quad\qquad - D^2_y\frac{dV}{d\nu}(t,\mu)(\xi) \gamma\left(t,\xi,\mu,e\right)- D_y\frac{dV}{d\nu}(t,\mu)(\xi) \gamma_x\left(t,\xi,\mu,e\right) \bigg] \lambda(de) \\
    &-\left[b_v\left(t,\xi,\mu,u^{t,\xi}_t\right) D_y\frac{dV}{d\nu}(t,\mu)(\xi)+f_v\left(t,\xi,\mu,u^{t,\xi}_t\right)\right]D_y u_t^{t,\xi}\\
    &-\he \Bigg\{a \left(t,\widehat{\xi},\mu\right) D_{y'}D_y^2\frac{d^2V}{d\nu^2}(t,\mu)\left(\widehat{\xi},\xi\right) +  D_y\frac{da}{d\nu}\left(t,\widehat{\xi},\mu\right)(\xi) D_y^2\frac{dV}{d\nu}(t,\mu)\left(\widehat{\xi}\right) \notag \\
    &\qquad + b\left(t,\widehat{\xi},\mu,u^{t,\widehat{\xi}}_t\right) D_{y'}D_y\frac{dV}{d\nu}(t,\mu)\left(\widehat{\xi},\xi\right) + D_y\frac{db}{d\nu}\left(t,\widehat{\xi},\mu,u^{t,\widehat{\xi}}_t\right)(\xi) D_y\frac{dV}{d\nu}(t,\mu)\left(\widehat{\xi}\right) \\
    &\qquad +D_y\frac{df}{d\nu}\left(t,\widehat{\xi},\mu,u^{t,\widehat{\xi}}_t\right)(\xi)\notag \\
    &\qquad +\int_E \bigg[D_{y'}\frac{d^2V}{d\nu^2}(t,\mu)\left(\widehat{\xi}+\gamma\left(t,\widehat{\xi},\mu,e\right),\xi\right) \notag \\
    &\quad\qquad\qquad +D_y\frac{dV}{d\nu}(t,\mu)\left(\widehat{\xi}+\gamma\left(t,\widehat{\xi},\mu,e\right)\right)D_y\frac{d\gamma}{d\nu}\left(t,\widehat{\xi},\mu,e\right)(\xi) \\
    &\quad\qquad\qquad -D_{y'}\frac{d^2V}{d\nu^2}(t,\mu)\left(\widehat{\xi},\xi\right)- \gamma\left(t,\widehat{\xi},\mu,e\right) D_{y'}D_y\frac{d^2V}{d\nu^2}(t,\mu)\left(\widehat{\xi},\xi\right) \\
    &\quad\qquad\qquad - D_y\frac{d\gamma}{d\nu}\left(t,\widehat{\xi},\mu,e\right)(\xi) D_y\frac{dV}{d\nu}(t,\mu)\left(\widehat{\xi}\right)  \bigg] \lambda(de) \\
    &\qquad + \left[b_v \left(t,\widehat{\xi},\mu,u^{t,\widehat{\xi}}_t\right) D_y\frac{dV}{d\nu}(t,\mu)\left(\widehat{\xi}\right)+f_v \left(t,\widehat{\xi},\mu,u^{t,\widehat{\xi}}_t\right)\right]D_\xi  u^{t,\widehat{\xi}}_t \Bigg\},
\end{align*}
Therefore, from the flow properties
\begin{align*}
    u^{s,Y^{t,\xi}_s}_s=u^{t,\xi}_s,\quad D_{Y^{t,\xi}_s}\   u^{s,\widehat{Y^{t,\xi}_s}}_s=D_\xi  \ u^{t,\widehat{\xi}}_s,
\end{align*}
we know that
\begin{align}
    &D_y\frac{d (\frac{\dd V}{\dd t})}{d\nu}\left(s,\lr\left(Y_s^{t,\xi}\right)\right)\left(Y_s^{t,\xi}\right) \notag \\
    =\ & - \bm{a}(s) D^3_y\frac{dV}{d\nu}\left(s,\lr\left(Y_s^{t,\xi}\right)\right)\left(Y_s^{t,\xi}\right)-\sigma_x\left(s,Y_s^{t,\xi},\lr\left(Y_s^{t,\xi}\right)\right)\bm{\sigma}(s) D^2_y\frac{dV}{d\nu}\left(s,\lr\left(Y_s^{t,\xi}\right)\right)\left(Y_s^{t,\xi}\right) \notag \\
    &- \bm{b}(s) D^2_y\frac{dV}{d\nu}\left(s,\lr\left(Y_s^{t,\xi}\right)\right)\left(Y_s^{t,\xi}\right) -b_x\left(s,\theta_s^{t,\xi}\right) D_y\frac{dV}{d\nu}\left(s,\lr\left(Y_s^{t,\xi}\right)\right)\left(Y_s^{t,\xi}\right)-f_x\left(s,\theta_s^{t,\xi}\right)\notag \\
    &- \int_E \bigg[D_y\frac{dV}{d\nu}\left(s,\lr\left(Y_s^{t,\xi}\right)\right)\left(Y_s^{t,\xi}+\bm{\gamma}(s,e)\right) \notag \\
    &\quad\qquad +D_y\frac{dV}{d\nu}\left(s,\lr\left(Y_s^{t,\xi}\right)\right)\left(Y_s^{t,\xi}+\bm{\gamma}(s,e)\right)\gamma_x\left(s,Y_s^{t,\xi},\lr\left(Y_s^{t,\xi}\right),e\right) \notag \\
    &\quad\qquad -D_y\frac{dV}{d\nu}\left(s,\lr\left(Y_s^{t,\xi}\right)\right)\left(Y_s^{t,\xi}\right) - D^2_y\frac{dV}{d\nu}\left(s,\lr\left(Y_s^{t,\xi}\right)\right)\left(Y_s^{t,\xi}\right) \bm{\gamma}(s,e) \notag \\
    &\quad\qquad - D_y\frac{dV}{d\nu}\left(s,\lr\left(Y_s^{t,\xi}\right)\right)\left(Y_s^{t,\xi}\right) \gamma_x\left(s,Y_s^{t,\xi},\lr\left(Y_s^{t,\xi}\right),e\right) \bigg] \lambda(de) \notag \\
    &-\bigg[b_v\left(s,\theta_s^{t,\xi}\right) D_y\frac{dV}{d\nu}\left(s,\lr\left(Y_s^{t,\xi}\right)\right)\left(Y_s^{t,\xi}\right)+f_v\left(s,\theta_s^{t,\xi}\right)\bigg]D_y u_s^{t,\xi} \notag \\
    &-\he \Bigg\{D_{y'}D^2_y\frac{d^2V}{d\nu^2}\left(s,\lr\left(Y_s^{t,\xi}\right)\right)\left(\widehat{Y^{t,\xi}_s},Y^{t,\xi}_s\right)\widehat{\bm{a}(s)} \notag \\
    &\qquad +D_y \frac{d\sigma}{d\nu}\left(s,\widehat{Y^{t,\xi}_s},\lr\left(Y_s^{t,\xi}\right)\right)\left(Y^{t,\xi}_s\right) \widehat{\bm{\sigma}(s)}D^2_y\frac{dV}{d\nu}\left(s,\lr\left(Y_s^{t,\xi}\right)\right)\left(\widehat{Y^{t,\xi}_s}\right) \notag \\
    &\qquad + D_{y'}D_y\frac{d^2V}{d\nu^2}\left(s,\lr\left(Y_s^{t,\xi}\right)\right)\left(\widehat{Y^{t,\xi}_s},Y^{t,\xi}_s\right) \widehat{\bm{b}(s)} \notag \\
    &\qquad + D_y\frac{db}{d\nu}\left(s,\widehat{\theta^{t,\xi}_s}\right)\left(Y^{t,\xi}_s\right) D_y\frac{dV}{d\nu}\left(s,\lr\left(Y_s^{t,\xi}\right)\right)\left( \widehat{Y^{t,\xi}_s} \right) +D_y\frac{df}{d\nu}\left(s,\widehat{\theta^{t,\xi}_s}\right)\left(Y^{t,\xi}_s\right) \notag \\
    &\qquad +\int_E \bigg[D_{y'}\frac{d^2V}{d\nu^2}\left(s,\lr\left(Y_s^{t,\xi}\right)\right)\left(\widehat{Y^{t,\xi}_s}+\widehat{\bm{\gamma}(s,e)},Y_s^{t,\xi}\right) \notag \\
    &\quad\qquad\qquad + D_{y}\frac{dV}{d\nu}\left(s,\lr\left(Y_s^{t,\xi}\right)\right)\left(\widehat{Y^{t,\xi}_s}+\widehat{\bm{\gamma}(s,e)}\right)D_y\frac{d\gamma}{d\nu}\left(s,\widehat{Y^{t,\xi}_s},\lr\left(Y_s^{t,\xi}\right),e\right)\left(Y_s^{t,\xi}\right) \notag \\
    &\quad\qquad\qquad -D_{y'}\frac{d^2V}{d\nu^2}\left(s,\lr\left(Y_s^{t,\xi}\right)\right)\left(\widehat{Y_s^{t,\xi}},Y_s^{t,\xi}\right) \notag \\
    &\quad\qquad\qquad - D_{y'}D_y\frac{d^2V}{d\nu^2}\left(s,\lr\left(Y_s^{t,\xi}\right)\right)\left(\widehat{Y_s^{t,\xi}},Y_s^{t,\xi}\right) \widehat{\bm{\gamma}(s,e)} \notag \\
    &\quad\qquad\qquad - D_y\frac{dV}{d\nu}\left(s,\lr\left(Y_s^{t,\xi}\right)\right)\left(\widehat{Y_s^{t,\xi}}\right) D_y\frac{d\gamma}{d\nu}\left(s,\widehat{Y_s^{t,\xi}},\lr\left(Y_s^{t,\xi}\right),e\right)\left(Y_s^{t,\xi}\right) \bigg] \lambda(de) \notag\\
    &\qquad + \bigg[b_v \left(s,\widehat{\theta^{t,\xi}_s}\right) D_y\frac{dV}{d\nu}\left(s,\lr\left(Y_s^{t,\xi}\right)\right)\left(\widehat{Y_s^{t,\xi}}\right)+f_v \left(s,\widehat{\theta^{t,\xi}_s}\right)\bigg]D_\xi  u^{t,\widehat{\xi}}_s \Bigg\}. \label{lem:Q&V_1.5}
\end{align}
Substituting the last equation into \eqref{lem:Q&V_1} and using \eqref{lem:Q&V_0.5}, we know that for a.e. $s\in[t,T]$,
\begin{align}
    dP^{t,\xi}_s=\ & \mathbf{Q}^{t,\xi}_s dB_s +\int_E \mathbf{R}^{t,\xi}_s(e)\mathring{N}(de,ds) \notag \\
    &-\bigg\{ L_x\left(s,Y^{t,\xi}_s,\lr\left(Y^{t,\xi}_s\right),u^{t,\xi}_s,P^{t,\xi}_s,\mathbf{Q}^{t,\xi}_s,\mathbf{R}^{t,\xi}_s\right) \notag \\
    &\qquad +L_v\left(s,Y^{t,\xi}_s,\lr\left(Y^{t,\xi}_s\right),u^{t,\xi}_s,P^{t,\xi}_s,\mathbf{Q}^{t,\xi}_s\right)D_y u_s^{t,\xi} \notag \\
    &\qquad +\he\bigg[D_y\frac{dL}{d\nu}\left(s,\widehat{Y^{t,\xi}_s},\lr\left(Y^{t,\xi}_s\right),\widehat{u^{t,\xi}_s},\widehat{P^{t,\xi}_s},\widehat{\mathbf{Q}^{t,\xi}_s}\right)\left(Y^{t,\xi}_s\right) \notag \\
    &\qquad\qquad + L_v\left(s,\widehat{Y^{t,\xi}_s},\lr\left(Y^{t,\xi}_s\right),\widehat{u^{t,\xi}_s},\widehat{P^{t,\xi}_s},\widehat{\mathbf{Q}^{t,\xi}_s}\right)D_\xi  u^{t,\widehat{\xi}}_s  \bigg] \bigg\}ds, \label{lem:Q&V_2}
\end{align}
where
\begin{equation}\label{lem:Q&V_3}
\begin{aligned}
    &\mathbf{Q}^{t,\xi}_s:=D^2_y\frac{dV}{d\nu}\left(s,\lr\left(Y_s^{t,\xi}\right)\right)\left(Y_s^{t,\xi}\right)\bm{\sigma}(s),\\
    &\mathbf{R}^{t,\xi}_s(e):=D_y\frac{dV}{d\nu}\left(s,\lr\left(Y_s^{t,\xi}\right)\right)\left(Y_{s-}^{t,\xi}+\bm{\gamma}(s,e)\right)-D_y\frac{dV}{d\nu}\left(s,\lr\left(Y_s^{t,\xi}\right)\right)\left(Y_{s-}^{t,\xi}\right).
\end{aligned}
\end{equation}
Therefore, from the uniqueness of the solution of the BSDE for $\left(P^{t,\xi},Q^{t,\xi},R^{t,\xi}\right)$, we know that
\begin{align}\label{lem:Q&V_4}
    Q^{t,\xi}_s=\mathbf{Q}^{t,\xi}_s,\quad R^{t,\xi}_s=\mathbf{R}^{t,\xi}_s.
\end{align}
Then, from \eqref{lem:Q&V_3} and \eqref{lem:Q&V_4} with $s=t$, we know that 
\begin{equation}\label{lem:Q&V_0}
    \begin{aligned}
        &Q^{t,\xi,j}_t=D_y^2\frac{dV}{d\nu}\left(t,\lr(\xi)\right)\left(\xi\right)\sigma^j\left(t,\xi,\lr(\xi)\right),\\
        &R^{t,\xi}_t(e)= D_y\frac{dV}{d\nu}\left(t,\lr\left(\xi\right)\right)\left(\xi+\gamma\left(t,\xi,\lr(\xi),e\right)\right)-D_y\frac{dV}{d\nu}\left(t,\lr\left(\xi\right)\right)\left(\xi\right).
    \end{aligned}
\end{equation}
Then, from \eqref{lem:Q&V_0} and the study of the processes $\Theta^{t,y,\mu}$ in Subsection~\ref{subsec:D_y}, we obtain \eqref{lem:Q&V_00}. %Moreover, note that in \eqref{lem:Q&V_2} the higher order derivatives of $V$ do not appear, which only appear as intermediate processes in \eqref{lem:Q&V_1} and \eqref{lem:Q&V_1.5}, and eventually cancel out. 
%Therefore, for $V$ satisfying conditions in Theorems~\ref{thm:value_regu_mu} and \ref{thm:V_t}, we can also obtain \eqref{lem:Q&V_00}. 

%%%%%%%%%%%%%%%%%%%%%%%%%%%%%%%%%%%%%%%%%%%%%%%%%%%%%%%%%%%%%%%%%%%%

%%%%%%%%%%%%%%%%%%%%%%%%%%%%%%%%%%%%%%%%%%%%%%%%%%%%%%%%%%%%%%%%%%%%%%%%%%%%

\end{document}